\documentclass[11pt,reqno,a4paper]{amsart}

\usepackage{amsmath,amssymb,amsfonts}
\usepackage{tikz}
\usepackage{hyperref}
\usepackage{appendix}
\usepackage{scalerel}
\usepackage{graphicx}
\usepackage{subcaption}
\usepackage{amsthm}
\usepackage{multirow}
\usepackage{pdflscape}
\usepackage{afterpage}
\usepackage{capt-of} 
\usepackage{lipsum}
\usepackage{booktabs}
\usepackage{siunitx}
\usepackage{esvect}
\usepackage{textcomp}
\usepackage{bbm}

\makeatletter
\def\set@curr@file#1{%
  \begingroup
    \escapechar\m@ne
    \xdef\@curr@file{\expandafter\string\csname #1\endcsname}%
  \endgroup
}
\def\quote@name#1{"\quote@@name#1\@gobble""}
\def\quote@@name#1"{#1\quote@@name}
\def\unquote@name#1{\quote@@name#1\@gobble"}
\makeatother
\usepackage{graphics}

\theoremstyle{plain}
\newtheorem{thm}{Theorem}
\newtheorem{lem}[thm]{Lemma}
\newtheorem{prop}[thm]{Proposition}
\newtheorem{cor}[thm]{Corollary}

\theoremstyle{definition}
\newtheorem{defn}{Definition}[section]

\theoremstyle{remark}
\newtheorem*{rem}{Remark}

\theoremstyle{plain}
\newtheorem{fact}{Fact}

\newcommand{\Rr}{\mathbb{R}}

\newcommand{\NN}{{\mathbb{N}}}

\newcommand{\RR}{\mathbb{R}}

\newcommand{\inter}{{\rm int}}

\newcommand{\RN}[1]{%
  \textup{\uppercase\expandafter{\romannumeral#1}}%
}

\newcommand{\skPoin}[2]{\pi_{#2}}
\newcommand{\skDomPoin}[2]{\Pi_{#2}}

\newcommand{\skProjDom}[2]{\Delta^{#1}_{#2}}
\newcommand{\skProjPoin}[2]{\hat{\pi}_{#2}}

\author[Telmo Peixe]{Telmo Peixe}
\address{\small{ISEG-Lisbon School of Economics \& Management, Universidade de Lisboa, REM-Research in Economics and Mathematics, CEMAPRE-Centro de Matem\'atica Aplicada \`a Previs\~ao e Decis\~ao Econ\'omica.}}

\author[Alexandre A. Rodrigues]{Alexandre A. Rodrigues}
\address{\small{Centro de Matem\'atica and Faculdade de Ci\^encias, Universidade do Porto.}}

\email{telmop@iseg.ulisboa.pt, alexandre.rodrigues@fc.up.pt}

\begin{document}


\title[Stability of heteroclinic cycles: a new approach ]
{Stability of heteroclinic cycles: \\ a new approach }

\date{\today}

 \begin{abstract}
 
This paper analyses the stability of cycles within a heteroclinic network  lying in a three-dimensional manifold formed by six  cycles, for a one-parameter model developed in the context of game theory.
We show the asymptotic stability of the network for a range of parameter values  compatible with the existence of an interior equilibrium and we describe an asymptotic technique to decide which cycle (within the network) is visible in numerics. The technique consists of  reducing the relevant dynamics to a suitable one-dimensional map, the so called \emph{projective map}. 
Stability of the fixed points of the projective map determines the stability of the associated cycles.  
The description of this new asymptotic approach is  applicable to more general types of networks and is potentially useful in computacional dynamics. 
 \end{abstract}

\maketitle


\section{Introduction}\label{sec:intro}
 
Recent studies in several areas have emphasized ways in which heteroclinic cycles and networks   may be responsible for \emph{intermittent dynamics} in nonlinear systems. They may be seen as the skeleton for the understanding of complicated dynamics. 
In this article, a \emph{heteroclinic cycle} is the union of hyperbolic equilibria  and solutions that connect them in a cyclic fashion \cite{field1991stationary,podvigina2011local, krupa1995asymptotic, rodrigues2013persistent}. A \emph{heteroclinic network} is a connected union of heteroclinic cycles (possibly infinite in number), such that for any pair of nodes in the network, there is a sequence of heteroclinic connections connecting them.

 Heteroclinic cycles or networks do not exist in a generic dynamical system, because small perturbations break connections between saddles. However, they may exist in systems where some constraints are imposed and are robust with respect to perturbations that respect these restrictions. Typically, these constraints create flow-invariant subspaces where the connection is of \emph{saddle-sink type} \cite{krupa1995asymptotic, field2020lectures}.

 In Lotka-Volterra modelled by systems in $(\RR_0^+)^n$, $n\in \NN$, the cartesian hyperplanes, also called by ``\emph{extinction subspaces}'', are flow-invariant. Similarly, such hyperplanes are invariant subspaces for systems on a simplex, a usual state space in \emph{Evolutionary Game Theory} (EGT) \cite{hofbauer1987permanence, hofbauer1998evolutionary, gaunersdorfer1995fictitious}. These conditions prompts the occurrence of heteroclinic networks associated to hyperbolic equilibria.

When a network is \emph{asymptotically stable}, the transition times between saddles increase geometrically   \cite{gaunersdorfer1995fictitious, labouriau2017takens}. 
Within a heteroclinic network,  no individual heteroclinic cycle can be asymptotically stable. However, the cycles can exhibit intermediate levels of stability, namely \emph{essential}  and \emph{fragmentary} asymptotic stability, important to decide the visibility of cycles in numerical simulations  \cite{podvigina2015simple, podvigina2020asymptotic, melbourne1991example}.
Useful conditions for asymptotic stability of some types of heteroclinic cycles have been established by \cite{podvigina2015simple, podvigina2020asymptotic, podvigina2012stability}.   

A classification of the complex networks as \emph{simple}, \emph{pseudo-simple} and \emph{quasi-simple} (among others) has been proposed by several authors, namely Krupa and Melbourne \cite{krupa1995asymptotic}, Podvigina and Chossat \cite{podvigina2017asymptotic}, Garrido-da-Silva and Castro \cite{garrido2019stability},   Podvigina \emph{et al} \cite{podvigina2020asymptotic}.
 A fruitful tool for quantifying stability of heteroclinic cycles is the \emph{local stability index} of  Podvigina and Ashwin  \cite{podvigina2011local} and Lohse \cite{lohse2015unstable}.

Given a heteroclinic cycle, the derivation of conditions for its stability involves the construction of an appropriate first return map, which typically is a highly non-trivial problem. The existence of various itineraries along a network that can be followed by nearby trajectories makes the study of the stability of networks a hard problem. This is why there are just a few instances of networks whose asymptotic stability was proven.

In this paper, we describe a method to study the heteroclinic dynamics of a differential equation arising in the context of a \emph{polymatrix game}.

We consider a one-parameter family of Ordinary Differential Equations (ODE) modelling the dynamics of a population divided in \emph{three} groups, each one with \emph{two} possible competitive strategies.
Interactions between individuals of any two groups are allowed, including the same group. The differential equation associated to a polymatrix game, that we designate as \emph{polymatrix replicator}, is defined in  a product of three simplices. Examples of such dynamical systems arise naturally in the context of \emph{Evolutionary Game Theory} (EGT) developed by   \cite{smith1973logic, peixe2021persistent} (see also references therein).

\subsection*{Novelty}

By making use of the theory developed in  \cite{peixe2021persistent, alishah2019asymptotic}, we start by showing the asymptotic stability of a network (containing six cycles) for a one-parameter family of autonomous differential equations, where the parameter lies in a interval compatible with the existence of an  \emph{interior equilibrium}.  

By computing periodic points of a one-dimensional map (\emph{projective map}), for parameter values ensuring the asymptotic stability of the network, we show that, if one of the cycles is attracting, then the others  are completely unstable. We also show the cycle where   a manifold containing the two-dimensional invariant manifold of the interior equilibrium, accumulates.

Our techniques are computationally applicable not only to networks in the EGT context (Lotka-Volterra systems), but to more general cases. 

We consider a quasi-change of coordinates  (near the network) to compute the preferred attracting cycle of the network. The basin of attraction of each cycle defines a sector in the dual set, whose asymptotic dynamics may be analysed through a piecewise smooth one-dimensional map on an interval -- \emph{the projective map}. Using the classical Perron-Frobenius Theory applied to linear operators, we conclude about the existence of a \emph{bijection} between stable periodic points of the projective map and stable heteroclinic cycles.

\subsection*{Structure}
This article is organised as follows.
 In Section \ref{sec:model} we introduce the one-parameter family of polymatrix replicators that will be interested in.
Once we have defined the main concepts used throughout the article, in Section \ref{sec:bif_analysis} we concentrate our analysis on  a parameter interval  where a single interior equilibrium exists, and we completely describe the dynamics on the boundary of the phase space, relating it with other equilibria on the cube's boundary. In particular, we describe the attracting heteroclinic network $\Sigma$ formed by the edges and vertices of the cube.

We present in Section~\ref{sec:asym_dyn} a piecewise linear model from where we  analyse the asymptotic dynamics  near the  network $\mathcal{H}$ introduced in Section \ref{sec:bif_analysis}.
In Section \ref{sec:analysis_on_dual}, we apply the previously established theory  to study the stability of all cycles in $\mathcal{H}$.
Our  method  is algorithmic and in Section \ref{sec:software_code} we address the reader to the \emph{Mathematica} code we developed to study polymatrix replicators.
Finally, in Section \ref{sec:discussion} we relate our main results with others in the literature.
We have endeavoured to make a self contained exposition bringing together all topics related to the method and the proofs. 
 
 In Appendices \ref{Appendix} and \ref{AppendixB}, we add some tables that will help the reader to understand our article, as well as the notation for constants and auxiliary functions.

\section{Model}\label{sec:model}
We analyse a particular case of a \emph{polymatrix game} whose phase space may be identified with a cube in $\Rr^3$.
Consider a population divided in \textbf{three} groups where individuals of each group have  \textbf{two} strategies to interact with other members of the population.
The model that we will consider to study the time evolution of the chosen strategies is the \emph{polymatrix game} and may be formalised as:

\begin{equation}\label{eq:poly_rep_1}
\dot{x}_i^\alpha(t) = x_i^\alpha (t) \left( (P x(t))_i^\alpha-\sum_{j=1}^2 (x_j^{\alpha}(t)) (P x(t))_j^\alpha \right), \alpha\in\{1,2,3\}, 
i\in\{1,2\},
\end{equation}
where $\dot{x}_i^\alpha(t)$ represents the time derivative of $x_i^\alpha(t)$, $P\in M_{6\times 6}(\Rr)$ is the payoff matrix,
$$
x(t)=\left(x_1^1(t),x_2^1(t),x_1^2(t),x_2^2(t),x_1^3(t),x_2^3(t)\right)
$$
and
$$
x_1^1(t)+x_2^1(t)=x_1^2(t)+x_2^2(t)=x_1^3(t)+x_2^3(t)=1.
$$
The indices may be interpreted as:
 \begin{eqnarray*}
\alpha &\mapsto & \text{subgroup of the population}; \\ 
i &\mapsto & \text{strategy of the associated subgroup.} 
 \end{eqnarray*}

For simplicity of notation, we  will write $x$ instead of $x(t)$.
The \emph{payoff matrix} $P$ can be represented as a matrix,  
$$
P=\left(
\begin{array}{c|c|c}
P^{1,1}  & P^{1,2} & P^{1,3} \\
\hline\\[-4mm]
P^{2,1}  & P^{2,2} & P^{2,3} \\
\hline\\[-4mm]
P^{3,1}  & P^{3,2} & P^{3,3}
\end{array}
\right) = \left(
\begin{array}{cc|cc|cc}
p_{1,1}^{1,1}  & p_{1,2}^{1,1} & p_{1,1}^{1,2}  & p_{1,2}^{1,2} & p_{1,1}^{1,3}  & p_{1,2}^{1,3} \\[2mm]
p_{2,1}^{1,1}  & p_{2,2}^{1,1} & p_{2,1}^{1,2}  & p_{2,2}^{1,2} & p_{2,1}^{1,3}  & p_{2,2}^{1,3} \\[1mm]
\hline\\[-3mm]
p_{1,1}^{2,1}  & p_{1,2}^{2,1} & p_{1,1}^{2,2}  & p_{1,2}^{2,2} & p_{1,1}^{2,3}  & p_{1,2}^{2,3} \\[2mm]
p_{2,1}^{2,1}  & p_{2,2}^{2,1} & p_{2,1}^{2,2}  & p_{2,2}^{2,2} & p_{2,1}^{2,3}  & p_{2,2}^{2,3} \\[1mm]
\hline\\[-3mm]
p_{1,1}^{3,1}  & p_{1,2}^{3,1} & p_{1,1}^{3,2}  & p_{1,2}^{3,2} & p_{1,1}^{3,3}  & p_{1,2}^{3,3} \\[2mm]
p_{2,1}^{3,1}  & p_{2,2}^{3,1} & p_{2,1}^{3,2}  & p_{2,2}^{3,2} & p_{2,1}^{3,3}  & p_{2,2}^{3,3}
\end{array}
\right)\,,
$$
where each block $P^{\alpha, \beta}$, $\alpha, \beta \in \{1,2,3\}$, represents the payoff of the individuals of the group $\alpha$ when interacting with individuals of the group $\beta$, and where each entry $p_{i,j}^{\alpha, \beta}$ represents the average payoff of an individual of the group $\alpha$ using strategy $i$ when interacting with an individual of the group $\beta$ using strategy $j$.

System \eqref{eq:poly_rep_1} is designated as a \textit{polymatrix replicator} \cite{peixe2021persistent, alishah2015hamiltonian,alishah2015conservative,peixe2019permanence}.
Assuming random encounters between individuals of the population, for each group $\alpha\in\{1,2,3\}$,
the average payoff for a strategy $i\in\{1,2\}$, is given by
$$
(Px)_i^\alpha=\sum_{\beta=1}^3 \left( P^{\alpha,\beta} \right)_i^\alpha x^\beta
= \sum_{\beta=1}^3 \sum_{k=1}^2 p_{i,k}^{\alpha,\beta}x_k^\beta\,,
$$
the average payoff of all strategies in $\alpha$ is given by
$$
\sum_{i=1}^2 x_i^{\alpha} \left( Px \right)_i^\alpha = \sum_{\beta=1}^3 (x^{\alpha})^T P^{\alpha,\beta} x^{\beta}\,,
$$
and the growth rate $ \dfrac{\dot{x}_i^{\alpha}}{x_i^{\alpha}}$ of the frequency of each strategy $i\in\{1,2\}$
is equal
to the payoff difference
$$ (Px)_i^\alpha - \sum_{\beta=1}^3 (x^{\alpha})^T P^{\alpha,\beta} x^{\beta}. $$
If $x=(x_1,x_2,x_3,x_4,x_5,x_6)$ is such that
\begin{equation}\label{sum1}
x_1+x_2=x_3+x_4=x_5+x_6=1,
\end{equation}
the system~\eqref{eq:poly_rep_1} may be written as
\begin{equation}\label{eq:poly_rep_2}
\left \{ \begin{array}{l}
\dot{x}_i=x_i\left((P x)_i-x_i(P x)_i-x_{i+1}(P x)_{i+1}\right) \\[2mm]
\dot{x}_{i+1}=x_{i+1}\left((P x)_{i+1}-x_i(P x)_i-x_{i+1}(Px)_{i+1}\right)
\end{array} \right. , \quad i\in\{1,3,5\}.
\end{equation}
By Lemma 1 of \cite{peixe2021persistent}, system~\eqref{eq:poly_rep_2} is equivalent to
\begin{equation} \label{eq:poly_rep_3}
\left \{ \begin{array}{l}
\dot{x}_1=x_1(1-x_1)\left((P x)_1-(Px)_2\right) \\[2mm]
\dot{x}_3=x_3(1-x_3)\left((P x)_3-(Px)_4\right) \\[2mm]
\dot{x}_5=x_5(1-x_5)\left((P x)_5-(Px)_6\right) 
\end{array} \right. ,
\end{equation}
where $\dot{x}_2=-\dot{x}_1$, $\dot{x}_4=-\dot{x}_3$, and $\dot{x}_6=-\dot{x}_5$.
Its phase space is
$$
\Gamma_{(2,2,2)} :=\Delta^1\times \Delta^1 \times \Delta^1 \subset\Rr^6,
$$
 a three-dimensional submanifold of $\RR^6$, where
$$
\Delta^1 = \{ (x_i,x_{i+1})\in\Rr^2 \,|\, x_i+x_{i+1}=1,\, x_i,x_{i+1}\geq 0 \}, \qquad  i\in\{1,3,5\}.
$$

Fixing a referential on $\Rr^3$, by~\eqref{sum1} we define a bijection between $ \Gamma_{(2,2,2)} \subset\Rr^6$ and $[0,1]^3\subset \Rr^3$.
 In Table~\ref{tbl:Representation_of_vs} (left) we associate each vertex of the cube $[0,1]^3$ with a vertex on $\Gamma_{(2,2,2)}$, where $(1,0,1,0,1,0)\in\Gamma_{(2,2,2)}$ and $(0,0,0)\in [0,1]^3$ are identified.

\begin{center}
\footnotesize{
\begin{tabular}{|c|c|c|} \toprule
Vertex  									&  $\Rr^3$ 					 	&  $\Rr^6$  		\\ \bottomrule \toprule
$\quad v_1\quad$ 	& $\quad (0,0,0)\quad$ 	& $\quad (1,0,1,0,1,0)\quad$	\\ \midrule
$v_2$ 					& $(0,0,1)$ 						& $(1,0,1,0,0,1)$  				\\ \midrule
$v_3$ 					& $(0,1,0)$ 						& $(1,0,0,1,1,0)$  				\\ \midrule
$v_4$ 					& $(0,1,1)$ 						& $(1,0,0,1,0,1)$  				\\ \midrule
$v_5$ 					& $(1,0,0)$ 						& $(0,1,1,0,1,0)$  				\\ \midrule
$v_6$ 					& $(1,0,1)$ 						& $(0,1,1,0,0,1)$  				\\ \midrule
$v_7$ 					& $(1,1,0)$ 						& $(0,1,0,1,1,0)$  				\\ \midrule
$v_8$ 					& $(1,1,1)$ 						& $(0,1,0,1,0,1)$  				\\ \bottomrule
\end{tabular}
\quad 
\begin{tabular}{|c|c|} \toprule
Face									& Vertices					 	   	\\ \bottomrule \toprule
$\quad  \sigma_1 \quad$ 	& $\{ v_5, v_6, v_7, v_8 \}$	\\ \midrule
$\quad  \sigma_2 \quad$ 	& $\{ v_1, v_2, v_3, v_4 \}$	\\ \midrule
$\quad  \sigma_3 \quad$ 	& $\{ v_3, v_4, v_7, v_8 \}$	\\ \midrule
$\quad  \sigma_4 \quad$ 	& $\{ v_1, v_2, v_5, v_6 \}$	\\ \midrule
$\quad  \sigma_5 \quad$  	& $\{ v_2, v_4, v_6, v_8 \}$	\\ \midrule
$\quad  \sigma_6\quad$ 	& $\{ v_1, v_3, v_5, v_7 \}$ \\ \bottomrule
\end{tabular}
\vspace{.2cm}
\captionof{table}{\small{Representation of the eight vertices of $[0,1]^3$ in $\Rr^3$ and  $\Gamma_{(2,2,2)}$ in $\Rr^6$, and the identification of the six faces according to vertices they contain.}}
        \label{tbl:Representation_of_vs}
}
\end{center}

Given the polymatrix replicator~\eqref{eq:poly_rep_1}, by~\cite[Proposition 1]{alishah2015conservative}, we may obtain an equivalent game with another payoff matrix whose second row of each group has $0$'s in all of its entries.
From now on, we will consider system~\eqref{eq:poly_rep_3} with payoff matrix
$$
P_{\boldsymbol{\mu}}=\left(
\begin{array}{cccccc}
102  &  \boldsymbol{\mu} & 0 & -158 & -18 & -9 \\
 0 & 0 & 0 & 0 & 0 & 0 \\
-51 & 51 & 0 & 0 & -9 & 18 \\
 0 & 0 & 0 & 0 & 0 & 0 \\
-102 & -153 & 237 & 0 & 27 & 9 \\
 0 & 0 & 0 & 0 & 0 & 0 \\
\end{array}
\right)\,.
$$
It defines a polynomial vector field on the compact flow-invariant set $\Gamma_{(2,2,2)}$ (for system~\eqref{eq:poly_rep_3}).
By compactness of $\Gamma_{(2,2,2)}$, the flow associated to system~\eqref{eq:poly_rep_3} is \emph{complete}, i.e. all solutions are defined for all $t\in \RR$.

From now on, let $\left( (2,2,2), P_{\boldsymbol{\mu}} \right)$ be the \emph{polymatrix game} associated to~\eqref{eq:poly_rep_3}. For $P=P_{\boldsymbol{\mu}}$, system~\eqref{eq:poly_rep_3} becomes
\begin{equation}\label{eq:poly_rep_4}
\left \{ \begin{array}{l}
\dot{x_1}=x_1(1-x_1)(P_{\boldsymbol{\mu}}\, x)_1 \\[2mm]
\dot{x_3}=x_3(1-x_3)(P_{\boldsymbol{\mu}}\, x)_3 \\[2mm]
\dot{x_5}=x_5(1-x_5)(P_{\boldsymbol{\mu}}\, x)_5 \\
\end{array} \right. .
\end{equation}
\label{systems are equivalent}
Considering $x=x_2$, $y=x_4$, $z=x_6$ and using~\eqref{sum1}, equation \eqref{eq:poly_rep_4} is equivalent to the following equation defined on the cube $[0,1]^3$: 
\begin{equation}\label{eq:poly_rep_5}
\left \{ \begin{array}{l}
\dot{x}=x(1-x)\left(-84+(102-\boldsymbol{\mu}) x +158 y - 9 z \right) \\[2mm]
\dot{y}=y(1-y)\left( 60 -102 x - 27 z \right) \\[2mm]
\dot{z}=z(1-z)\left( -162 + 51 x + 237 y + 18 z \right)
\end{array} \right. . \\
\end{equation}

Vertices, edges and faces of the cube  are flow-invariant. 
In order to lighten the notation, when there is no risk of misunderstanding, the one-parameter vector fields associated to \eqref{eq:poly_rep_4} and \eqref{eq:poly_rep_5} will be denoted by $f_{\boldsymbol{\mu}}$ and its flow by $\varphi(t,u_0)$, $t\in \RR_0^+$, $u_0\in \Gamma_{(2,2,2)}$ (for \eqref{eq:poly_rep_4}) and $u_0\in [0,1]^3$ (for \eqref{eq:poly_rep_5}). When there is no risk of misunderstanding, we  omit the dependence on $\boldsymbol{\mu}$.

\begin{rem}
As performed in \cite{peixe2021persistent}, in the transition from \eqref{eq:poly_rep_4} to \eqref{eq:poly_rep_5},
we have identified the point $(1,0,1,0,1,0)\in\Gamma_{(2,2,2)}$, associated to a \emph{pure strategy} in the original polymatrix replicator, with   $(0,0,0)\in \Rr^3$ (cf. Table \ref{tbl:Representation_of_vs}).
\end{rem}

\subsection*{Notation}
The following terminology will be used throughout the manuscript:
\begin{eqnarray*}
\mathcal{V} &\mapsto&  \{v_1, ..., v_8\};\\
 \mathcal{F} &\mapsto & \text{set of all faces of the cube } [0,1]^3; \\
  \mathcal{F}_v&\mapsto & \text{set of faces $\sigma_j$, for which the component $x_j$ of $v$ are zero}, v\in \mathcal{V}.  
 \end{eqnarray*}

\section{Preliminaries}\label{sec:prelimin}
In this section we define the main concepts used throughout the article.
For $n\in \NN$, we are considering the Banach space $\RR^n$ endowed with the usual norm $\| \star\|$ and the usual euclidian metric \emph{dist}.
 The symbol $\ell$ denotes the Lebesgue measure of a Borel subset of $\RR^n$.

\subsection{Admissible path and heteroclinic cycle}
For $n\in\NN$, we consider  a smooth one-parameter family of vector fields $f_{\boldsymbol{\mu}}$ on $\RR^{n}$,
with flow given by the unique solution  $u(t)=\varphi(t, u_{0})$  of
\begin{equation}
\dot{u}=f_{\boldsymbol{\mu}}(u),\qquad \varphi(0, u_0)=u_{0}, 
\label{sistema geral}
\end{equation}
where $\dot{u}=\frac{du}{dt}$, $u_{0} \in \RR^{n},$ $t\in \RR$, and $\boldsymbol{\mu}$ is a real parameter. 
If $A\subseteq \RR^n$, we denote by $\inter\left(A\right)$, $\overline{A}$ and $\partial A$
the topological \emph{interior},  \emph{closure} and \emph{boundary} of $A$, respectively.

\subsubsection{$\alpha$ and $\omega$-limit set}
For a solution of~\eqref{sistema geral} passing through $u_0\in \RR^n$, the set of its accumulation points as $t$ goes to $+\infty$ is the $\omega$-limit set of $u_0$ and will be denoted by $\omega(u_0)$. More formally, 
$$
\omega(u_0)=\bigcap_{T=0}^{+\infty} \overline{\left(\bigcup_{t>T}\varphi (t, u_0)\right)}.
$$ 
The set $\omega(u_0)$  is closed and flow-invariant, and if the $\varphi$-trajectory of $u_0$ is contained in a compact set, then 
$\omega(u_0)$ is non-empty \cite{guckenheimer2013nonlinear}. If $Y\subset \RR^n$, we define $\omega(Y)$ as the union of all $\omega$-limits of $y\in Y$.
 We define analogously,  the $\alpha$-limit set by reversing the evolution of  $t$.

\subsubsection{Heteroclinic cycles}
We introduce the concept of heteroclinic connection, heteroclinic path, heteroclinic cycle and network associated to a finite set of hyperbolic equilibria.  We address the reader to Field  \cite{field2020lectures} for more information on the subject.
 \begin{defn} 
For $m\in\NN$, given two hyperbolic equilibria of saddle-type
$A$ and $B$ associated to the flow of~\eqref{sistema geral}, an $m$-dimensional \emph{heteroclinic connection } from $A$ to $B$, denoted $[A\rightarrow B]$, is an $m$-dimensional connected and flow-invariant manifold contained in $W^{u}(A)\cap W^{s}(B)$. 
 \end{defn}

 \begin{defn} 
For $k\in \NN$, given a sequence of one-dimensional heteroclinic connections  $\{\gamma_0, ..., \gamma_k\}$ for ~\eqref{sistema geral}, we say that it is an \emph{admissible path} if for all $j\in \{0,1, ..., k-1\}$, we have $\omega(\gamma_j)=\alpha(\gamma_{j+1})$. If $\omega(\gamma_k)= \alpha(\gamma_0)$, this sequence is called a \textit{heteroclinic cycle}.
A \textit{heteroclinic network} is a connected union of heteroclinic cycles.
 \end{defn}
When there is no risk of misunderstanding, we represent the cycles and networks by the \emph{ordered set} of their associated saddles as in Definitions 6.7 and 6.22 of  \cite{field2020lectures}. In general, heteroclinic networks are represented by \emph{directed graphs} where the vertices represent the equilibria and the oriented edges represent heteroclinic connections.

\subsection{Stability}
We recall the following stability  definitions that can be found in  \cite{podvigina2020asymptotic, podvigina2019stability}.
 In what follows $X,Y \subset \RR^n$   are compact flow-invariant sets for the system \eqref{sistema geral}.

\begin{defn} 
\label{def: 1}
\begin{enumerate}
\item The set $X$ is \emph{Lyapunov stable} if for any neighbournood $U$ of $X$, there exists a neighbourhood $V$ of $X$ such that
$$
\forall x\in V,\quad  \forall t\in \RR^+, \qquad \varphi(t, x)\in U.
$$

\item The set $X$ is \emph{asymptotically stable} if it is Lyapunov stable and in addition the neighbourhood $V$ can be chosen such that:
$$
\forall x\in V, \quad \lim_{t \rightarrow +\infty} dist(\varphi(t, x), X)=0.
$$

\item The set $X$ is \emph{globally asymptotically stable in Y} if it attracts all trajectories starting at $Y$.\\
\item The set $X$ is \emph{unstable} if it is not Lyapunov stable.
\end{enumerate}

\end{defn}

 A heteroclinic cycle that belongs to a network (not reduced to a single cycle) cannot be asymptotically stable because it does not contain the entire unstable manifolds of all its equilibria (according to \cite{podvigina2020asymptotic}, it is not \emph{clean}).  Various intermediate notions of stability have been introduced over the last decades -- we address the reader to \cite{podvigina2020asymptotic, podvigina2019stability}\footnote{There is an abundance of references in the literature. We choose to mention only two, based on our personal preferences.  The reader interested in further detail may use the references within those we mention.} for a nice description of these different levels of stability.  
\subsection{Likely limit-set}
We now introduce two concepts respecting system~\eqref{sistema geral}, that will used throughout the article.

\begin{defn}
If $X$ is a compact invariant subset of $\RR^n$, the basin of attraction of $X$, denoted by $ \mathcal{B}(X) $, is the set $$ \{x \in \RR^n:\, \,  \omega(x)\subset X\}.$$
\end{defn}
 
\begin{defn}[\cite{milnor1985concept}]
If  $Y \subset \RR^n$ is a measurable forward invariant set with $\ell(Y) >0$, the \emph{likely limit set}  of $Y$, denoted by $\mathcal{L}(Y)$,  is the smallest closed  invariant subset of $Y$ that contains all $\omega$-limit sets except for a subset of $Y$ of zero Lebesgue measure. 
\end{defn}
When we restrict the flow to a compact set, $\mathcal{L}(Y)$ is non-empty, compact and forward invariant \cite{milnor1985concept}.

\begin{figure}[h]
  	\includegraphics[width=12.1cm]{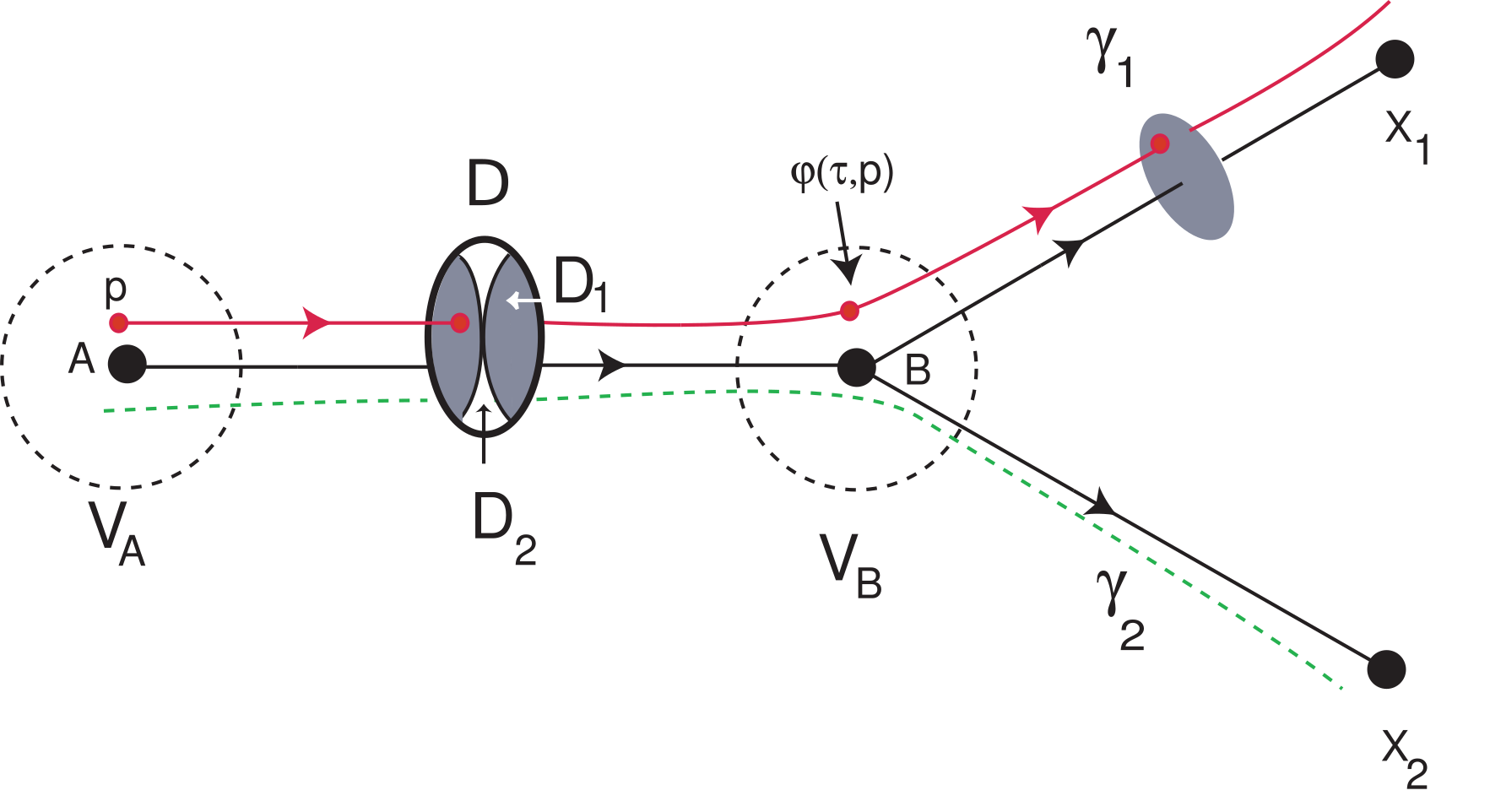}
	     \caption{\footnotesize{ Illustration of a switching node (B): for $i\in\{1,2\}$, there are initial conditions in $D_i$ whose trajectories follow $\gamma_i$.}}\label{fig:switching1}
\end{figure}

\subsection{Switching node}
The next definition is adapted from \cite{rodrigues2013persistent, castro2010heteroclinic}.
Let $A, B, X_1$ and $X_2$ be four saddle equilibria of \eqref{sistema geral}. Given a neighbourhood $V_A, V_B$ of $A$ and $ B$, respectively,  we say: \\

\begin{enumerate}

\item  there is \emph{switching at the node $B$} (or  $B$ is a \emph{switching node}) if given a neighbourhood $V_{B}$ of $B$, for any $\varepsilon>0 $, and for any $(n-1)$-dimensional disk $D$ that meets the connection $\gamma:=[A \rightarrow B]$ transversely, there are points in $D$ that follow each of the connections $\gamma_1:=[B \rightarrow X_1]$ and $\gamma_2:=[B \rightarrow X_2]$  at a distance $\varepsilon $ (Figure \ref{fig:switching1}).\\

\item  a point $p$ \emph{follows the connection} $[A\rightarrow B]$  at a  distance $\varepsilon>0$ if there is a $\tau>0$ such that $\varphi(0, p) \equiv p \in V_A$,  $\varphi(\tau, p)\in V_B$,   and such that for all $t \in [0,\tau]$ the trajectory $\varphi(t,p)$ lies at a distance less than $\varepsilon$ from the connection $\gamma:=[A\rightarrow B]$ (Figure \ref{fig:switching1}). \\
 
\item a point $p\in V_A$ \emph{follows the admissible path} $\{\gamma_0, ..., \gamma_k\}$, $k\in \NN$, with distance $\varepsilon>0$ if there exist $q\in \RR^n$ and two monotonically increasing sequences of times $(t_i)_{i \in \{0,1, ..., k+1\}}$ and $(s_i)_{i \in \{0,1, ..., k\}}$ such that for all $i \in \{0, ..., k\}$ we have $t_i<s_i<t_{i+1}$ and
\begin{itemize}
\item $\varphi(t, p)$ lies in a $\varepsilon$-tubular neighbourhood of $\{\gamma_0, ..., \gamma_k\}$ for all $t\in[t_i, t_{i+1}]$;
\item $\varphi(t_i, q) \in N_{\alpha(\gamma_i)}$ and $\varphi(s_i, q)$ lies in a $\varepsilon$--tubular neighbourhood of $\gamma_i$ disjoint from $N_{\alpha(\gamma_i)}$ and $N_{\omega(\gamma_i)}$;
\item for all $t \in [s_i, s_{i+1}]$, the trajectory $\varphi(t, p)$ does not visit  the neighbourhood of any other saddle except that of $\omega(\gamma_{i})$.\\
\end{itemize}
\end{enumerate}

Under the previous hypotheses, if $B$ is a switching node we may define $D_1, D_2\subset D$ such that initial conditions within $D_1, D_2$ follow the connections $\gamma_1= [B\to X_1]$ and $\gamma_2=[B\to X_2]$, respectively (Figure \ref{fig:switching1}).

\section{Bifurcation analysis}\label{sec:bif_analysis}

    We proceed to the analysis of the one-parameter family of differential equations~\eqref{eq:poly_rep_5} in $[0,1]^3$.
Our analysis will be focused on   $\boldsymbol{\mu}\in \mathcal{I}:=\left[\frac{850}{11},\frac{544}{5} \right]$ since 
for all $\boldsymbol{\mu}\in \inter \left( \mathcal{I} \right)$ there exists a unique equilibrium in $\inter \left( [0,1]^3\right) $
In what follows, we list some assertions that have been found (both analytical and numerically).

\subsection{Boundary dynamics}
We describe a list of equilibria that appear on  $\partial [0,1]^3$, as function on the parameter $\boldsymbol{\mu}$.
We also emphasise the bifurcations the equilibria undergo.

\begin{figure}[h]
    \centering
	\includegraphics[width=4.1cm]{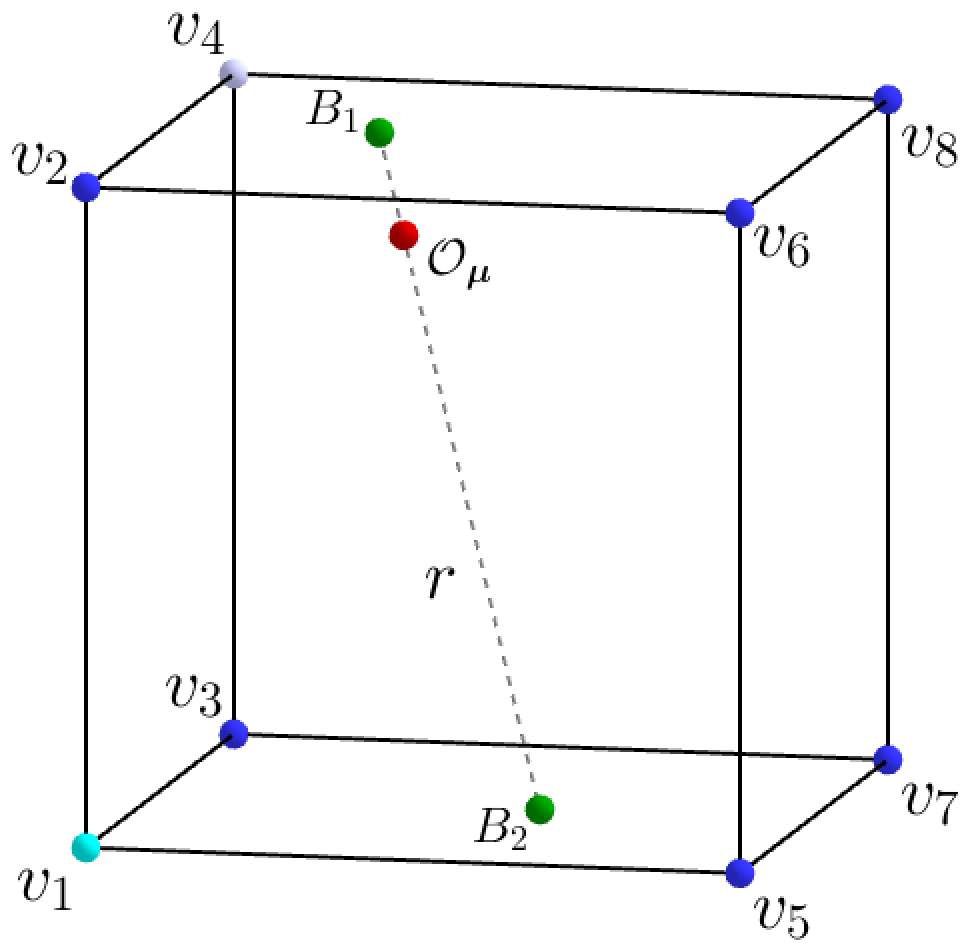}
	\includegraphics[width=4.1cm]{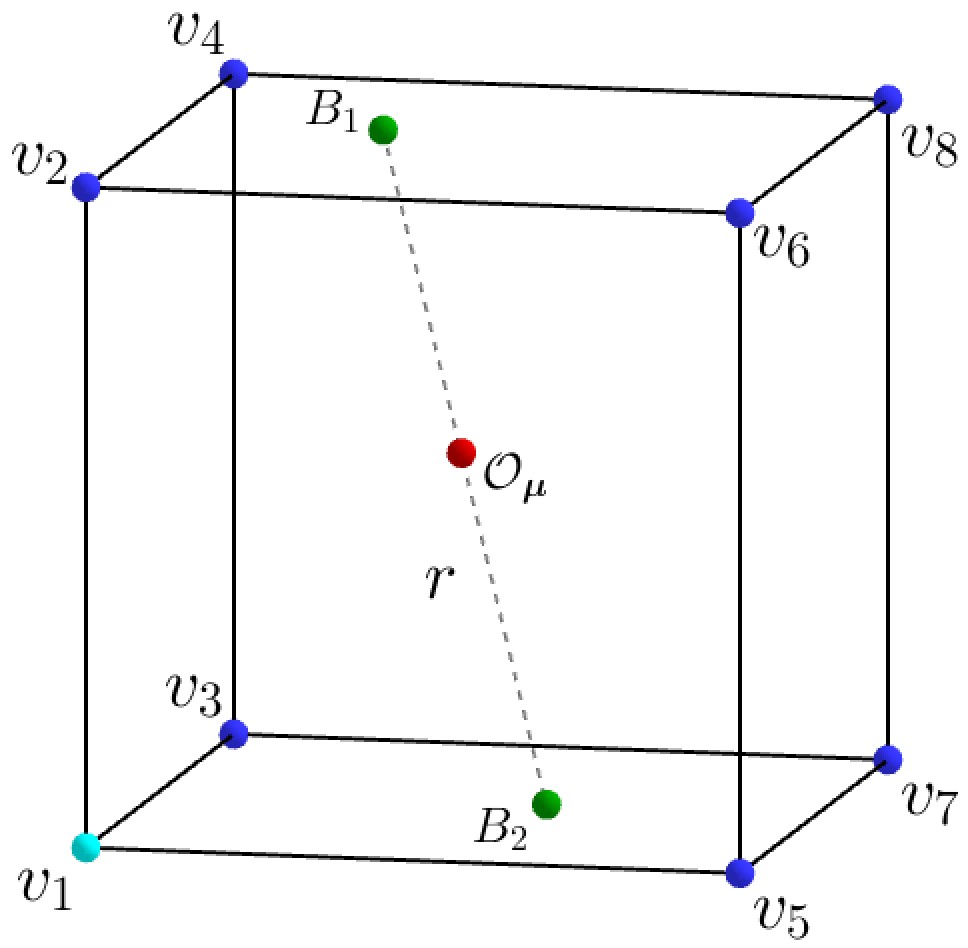}
	\includegraphics[width=4.1cm]{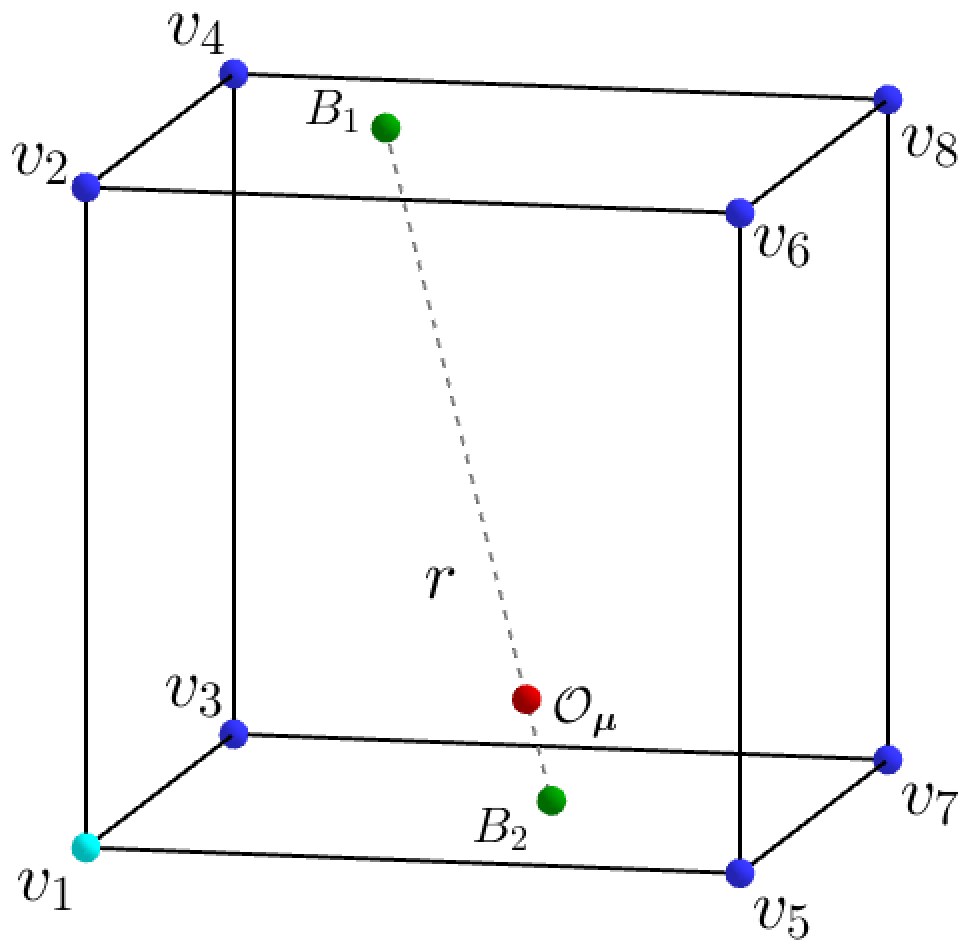}
    \caption{\footnotesize{The phase space and the corresponding equilibria of  \eqref{eq:poly_rep_5}: the eight vertices $v_1, \dots , v_8$ (blue), two equilibria on faces, $B_1, B_2$ (green), and the interior equilibrium $\mathcal{O}_{\boldsymbol{\mu}}$ (in red), for $\boldsymbol{\mu}=85$ (left), $\boldsymbol{\mu}=97$ (center) and $\boldsymbol{\mu}=106$ (right). The interior equilibrium $\mathcal{O}_{\boldsymbol{\mu}}$ lies on the line segment $r$ that connects $B_1$ to $B_2$ (Lemma~\ref{lem:line_segment}).}}\label{fig:equilibria_on_cube}
\end{figure}

From now on, all figures with numerical plots of the flow of \eqref{eq:poly_rep_5} on  $[0,1]^3$ are in the same position of Figure~\ref{fig:equilibria_on_cube} where $v_1=(0,0,0)$ is the vertex 
in light blue located in the lower left front corner. The cube has six faces defined, for  $i\in\{1, 2, 3 \}$, by
\begin{align*}
\sigma_{2i-1} &:= \{ (x_1,x_2,x_3)\in\partial [0,1]^3  \, :\,  x_i = 1 \}, \\
\sigma_{2i} &:= \{ (x_1,x_2,x_3)\in\partial [0,1]^3  \, :\,  x_i = 0 \}.
\end{align*}
In Table~\ref{tbl:Representation_of_vs} we identify the vertices that belong to each face. As suggested in Figure \ref{fig:equilibria_on_cube}, we set the notation $B_j$, $j=1, 2$ for the equilibria on the interior of the faces $\sigma_5$ and $\sigma_6$.
Formally, the $B$'s equilibria depend on $\boldsymbol{\mu}$ but, once again, we omit their dependence on the parameter.

\begin{lem}\label{lem:eq_cube_boundary}
For $\boldsymbol{\mu} \in \mathcal{I}$, the   vertices $v_1, \dots , v_8$ and
$$
B_1=\left( \frac{11}{34}, \frac{2040+11\boldsymbol{\mu}}{5372},1 \right)
\qquad\textrm{and}\qquad
B_2=\left( \frac{10}{17}, \frac{204+5\boldsymbol{\mu}}{1343},0 \right)
$$
are equilibria of \eqref{eq:poly_rep_5} and belong to the cube's boundary.
\end{lem}

The proof of Lemma \ref{lem:eq_cube_boundary} is straightforward by computing zeros of $f_{\boldsymbol{\mu}}$ and taking into account that equilibria lie in $\partial [0,1]^3$.
The eigenvalues and eigendirections of the vertices and the $B$'s are summarised in Tables~\ref{tbl:Eigenv_of_vs} and \ref{tbl:Eigenv_of_Bs} in Appendix \ref{Appendix}, respectively.

\begin{table}[h] 
\centering
\begin{tabular}{c||cccccc}
Eq./Eignv. & $\sigma_1$  & $\sigma_2$  & $\sigma_3$  & $\sigma_4$  & $\sigma_5$ & $\sigma_6$ \\
\hline \hline \\[-4mm]
\,$v_1$           &  $*$   									&  $-84$  &  $*$   	&  $60$   	&  $*$   		&  $-162$  \\
\,$v_2$           &  $*$   									&  $-93$  &  $*$   	&  $33$   	&  $144$   	&  $*$  \\
\,$v_3$           &  $*$   									&  $74$  	&  $-60$   &  $*$   								&  $*$   		&  $75$  \\
\,$v_4$           &  $*$   									&  $65$  	&  $-33$   &  $*$   								&  $-93$   	&  $*$  \\
\,$v_5$           &  $\boldsymbol{\mu}-18$   	&  $*$  		&  $*$   	&  $-42$   &  $*$   		&  $-111$  \\
\,$v_6$           &  $\boldsymbol{\mu}-9$   	&  $*$  		&  $*$   	&  $-69$   &  $93$   		&  $*$  \\
\,$v_7$           &  $\boldsymbol{\mu}-176$   	&  $*$  		&  $42$   	&  $*$   								&  $*$   		&  $126$  \\
\,$v_8$           &  $\boldsymbol{\mu}-167$   &  $*$  		&  $69$   	&  $*$   								&  $-144$   	&  $*$  \\
\hline \hline
\end{tabular}
\vspace{.2cm}
\caption{\footnotesize{The eigenvalues of system~\eqref{eq:poly_rep_5} at the vertices, where
the entry at line $i$ and row $j$ is the eigenvalue of the vertex $v_j$ in the orthogonal direction to the face $\sigma_i$,
and the symbol $*$ means that the vertex $v_i$ does not belong to the face $\sigma_j$ of the cube $[0,1]^3$.}}
\label{tbl:Eigenv_of_vs}
\end{table}

If $A$ is a saddle-focus for system  \eqref{eq:poly_rep_5}, we say that it is of \emph{type} $(1,2)$ if $Df_{\boldsymbol{\mu}} (A)$ has a pair of non-real complex eigenvalues with $\dim W^s(A)=1$ and $\dim W^u(A)=2$.

\begin{table}[h]
\begin{small}
\begin{tabular}{|c|c|c|c|c|} \toprule
Eq.               & Eigenvalues & $\boldsymbol{\mu}$      			    &  On face        & On the interior  \\ \toprule
\multirow{3}{*}{$B_1$} & \multirow{3}{*}{$\left\{ \frac{2550-33 \boldsymbol{\mu}}{68}, z_1, \bar{z}_1 \right\}$}
			& $\frac{850}{11}<\boldsymbol{\mu}<102$	& $(+,+)_\textbf{C}$ 		& $(-)$ \\ \cmidrule{3-5}
		&	& $\boldsymbol{\mu}=102$							& $(0,0)_\textbf{C}$		& $(-)$ \\ \cmidrule{3-5}
		&	& $102<\boldsymbol{\mu}<\frac{544}{5}$		& $(-,-)_\textbf{C}$ 		& $(-)$ \\ \midrule
\multirow{3}{*}{$B_2$} & \multirow{3}{*}{$\left\{ \frac{15 \boldsymbol{\mu}-1632}{17}, z_2, \bar{z}_2 \right\}$}
			& $\frac{850}{11}<\boldsymbol{\mu}<102$	& $(+,+)_\textbf{C}$ 		& $(-)$ \\ \cmidrule{3-5}
		&	& $\boldsymbol{\mu}=102$							& $(0,0)_\textbf{C}$		& $(-)$ \\ \cmidrule{3-5}
		&	& $102<\boldsymbol{\mu}<\frac{544}{5}$		& $(-,-)_\textbf{C}$ 		& $(-)$ \\ \bottomrule
\end{tabular}
\end{small}

\begin{tiny}
\begin{align*}
z_1 &= \frac{2038674-19987 \boldsymbol{\mu} + \sqrt{19987} \sqrt{44671 \boldsymbol{\mu}^2-6976596 \boldsymbol{\mu} -1178700372}}{182648} \\[2mm]
z_2 &= \frac{282030 -2765 \boldsymbol{\mu} + \sqrt{2765} \sqrt{12965 \boldsymbol{\mu}^2-2471460 \boldsymbol{\mu} -66034188}}{22831}
\end{align*}
\end{tiny}
\vspace{-.2cm}
\captionof{table}{\small{Eigenvalues of equilibria $B_1$ and $B_2$, depending on $\boldsymbol{\mu}$, at the corresponding faces and pointing to the interior, where the signs   $(-)$, $(0)$, and $(+)$ mean that the eigenvalues are real negative, zero, or positive, respectively, and $(+,+)_\textbf{C}$ (respectively, $(-,-)_\textbf{C}$) means that the eigenvalues are conjugate (non-real) with positive (respectively, negative) real part and $(0,0)_\textbf{C}$ means that the eigenvalues are  pure imaginary. 
}}\label{tbl:Eigenv_of_Bs}
\end{table}

 The evolution of the eigenvalues' sign as function of $\boldsymbol{\mu}$ allows us to  locate \emph{transcritical bifurcations}, which are summarised in the following paragraph.
We   consider sub-intervals of $\mathcal{I}$ based on the values of $\boldsymbol{\mu}$ for which this bifurcation occurs. Observing   Table~\ref{tbl:Eigenv_of_Bs}, we may easily conclude that:

\begin{lem}
\label{lem:B_1 and B_2 stability}
For $\boldsymbol{\mu} \in \mathcal{I}$, the following assertions hold for  \eqref{eq:poly_rep_5}:\\
\begin{enumerate}
	\item if $\frac{850}{11}<\boldsymbol{\mu}<102$,  then $B_1$ and $B_2$ are saddle-foci of type $(1,2)$; \\
	\item if $\boldsymbol{\mu}=102$, then  $B_1$ and $B_2$ are non-hyperbolic when restricted to the corresponding faces\footnote{In fact, when restricted to the corresponding faces, $B_1$ and $B_2$ are centers.  The proof follows from Section 4 of \cite{alishah2015conservative}.};\\
	\item if $102<\boldsymbol{\mu}<\frac{544}{5}$,  then  $B_1$ and $B_2$ are sinks. \\
\end{enumerate} 
\end{lem}

Since there are no more invariant sets on the faces (for $\boldsymbol{\mu} \neq 102$), besides $B_1$, $B_2$ and the vertices, we may conclude that: 
\begin{lem}
\label{lem: dynamics on the faces}
With respect to system~\eqref{eq:poly_rep_5}, the following assertions hold:\\
\begin{enumerate}
\item For $\boldsymbol{\mu} \in \mathcal{I}$, if $p\in \inter(\sigma_i)$, $i=1,2,3,4$, then $\omega(p)$ is a vertex. \\
\item For $\boldsymbol{\mu} \in \left[\frac{850}{11}, 102\right[$:
\begin{enumerate}
 \item if $p\in \inter(\sigma_5)$,  then $\omega(p)$ is the cycle defined by $\{v_2,v_4,v_8,v_6\}$;
\item if $p\in \inter(\sigma_6)$,  then $\omega(p)$ is the cycle defined by $\{v_1,v_3,v_7,v_5\}$.\\
\end{enumerate}

\item For $\boldsymbol{\mu} \in  \left]102, \frac{544}{5}\right]$:
\begin{enumerate}
 \item if $p\in \inter(\sigma_5)$,  then $\omega(p)=\{B_1\}$;
\item if $p\in \inter(\sigma_6)$,  then $\omega(p)=\{B_2\}$. \\
\end{enumerate}
\end{enumerate}
\end{lem}


\subsection{Interior equilibrium}
\label{ss:interior}
In this subsection, we focus our attention on the interior equilibrium and its relation to others on the cube's boundary.

\begin{lem}
\label{lem:int_equilib}
For $\boldsymbol{\mu} \in \inter\left(\mathcal{I}\right)$, system~\eqref{eq:poly_rep_5} has a unique interior equilibrium,
whose expression is
\begin{equation}\label{int_equil}\nonumber
\mathcal{O}_{\boldsymbol{\mu}}:=\left(\frac{68}{442-3\boldsymbol{\mu}},
\frac{2(9180-61 \boldsymbol{\mu})}{79(442-3 \boldsymbol{\mu})},
\frac{4(544-5 \boldsymbol{\mu})}{1326-9 \boldsymbol{\mu}}\right).
\end{equation}
\end{lem}

\begin{proof}
The proof is immediate by computing the non-trivial zeros of the vector field $f_{\boldsymbol{\mu}}$ of ~\eqref{eq:poly_rep_5}.
\end{proof}

Taking into account that  $\mathcal{O}_{\boldsymbol{\mu}}$, $B_1$ and $B_2$ depend on $\boldsymbol{\mu}$, it is worth to notice that 
$$ \lim_{\boldsymbol{\mu} \rightarrow \frac{850}{11}} \mathcal{O}_{\boldsymbol{\mu}}= \lim_{\boldsymbol{\mu} \rightarrow \frac{850}{11}} B_1 = \left( \frac{11}{34},\frac{85}{158},1\right) \in \sigma_5$$ and 
$$\lim_{\boldsymbol{\mu} \rightarrow \frac{544}{5}}\mathcal{O}_{\boldsymbol{\mu}}=\lim_{\boldsymbol{\mu} \to \frac{544}{5}} B_2 = \left(\frac{10}{17}, \frac{44}{79}, 0\right) \in \sigma_6, 
$$
which means that along $\mathcal{I}$, the point $\mathcal{O}_ {\boldsymbol{\mu}}$ travels from the face $\sigma_5$ to $\sigma_6$.
The following result shows an elegant relative position of the equilibria $B_1$, $B_2$ and $\mathcal{O}_{\boldsymbol{\mu}}$ (see Figure \ref{fig:equilibria_on_cube}).

\begin{lem}
\label{lem:line_segment}
For $\boldsymbol{\mu} \in \mathcal{I} $, the interior equilibrium $\mathcal{O}_{\boldsymbol{\mu}}$ belongs to the segment $[B_1 B_2]$.
\end{lem}

\begin{proof}
Let $r$ be the segment $[B_1 B_2]$ defined by
$$
r\,: \quad (x,y,z)=B_1+k \vv{B_1 B_2}, \quad \textrm{for } k\in [0,1]. 
$$
By a simple computation we have that
$$
\mathcal{O}_{\boldsymbol{\mu}} \in r
\quad\Leftrightarrow\quad
k=\frac{850-11 \boldsymbol{\mu}}{9 \boldsymbol{\mu}-1326}\in [0,1]
\quad\Leftrightarrow\quad
\boldsymbol{\mu} \in \mathcal{I}.
$$
\end{proof}

\begin{lem}
\label{lem:jacobian_eigenv}
There exists $\boldsymbol{\mu}_2 \in \mathcal{I}$ such that the equilibrium $\mathcal{O}_{\boldsymbol{\mu}}$ undergoes a supercritical Hopf bifurcation.
\end{lem}

\begin{proof}
For $\boldsymbol{\mu} \in \mathcal{I}$, $D f_{\boldsymbol{\mu}}\left( \mathcal{O}_{\boldsymbol{\mu}} \right)$ depends on ${\boldsymbol{\mu}}$ and is explicitly given by 
\vspace{.2cm}
$$
\begin{small}
\left(
\begin{array}{ccc}
 \frac{68 (\boldsymbol{\mu} -102) (3 \boldsymbol{\mu} -374)}{(442-3 \boldsymbol{\mu} )^2} & -\frac{10744 (3 \boldsymbol{\mu} -374)}{(442-3 \boldsymbol{\mu} )^2} & \frac{612 (3 \boldsymbol{\mu} -374)}{(442-3 \boldsymbol{\mu} )^2} \\[3mm]
 -\frac{204 (61 \boldsymbol{\mu} -9180) (115 \boldsymbol{\mu} -16558)}{6241 (442-3 \boldsymbol{\mu} )^2} & 0 & -\frac{54 (61 \boldsymbol{\mu} -9180) (115 \boldsymbol{\mu} -16558)}{6241 (442-3 \boldsymbol{\mu} )^2} \\[3mm]
 -\frac{68 (5 \boldsymbol{\mu} -544) (11 \boldsymbol{\mu} -850)}{3 (442-3 \boldsymbol{\mu} )^2} & -\frac{316 (5 \boldsymbol{\mu} -544) (11 \boldsymbol{\mu} -850)}{3 (442-3 \boldsymbol{\mu} )^2} & -\frac{8 (5 \boldsymbol{\mu} -544) (11 \boldsymbol{\mu} -850)}{(442-3 \boldsymbol{\mu} )^2} \\
\end{array}
\right),
\end{small}
\vspace{.2cm}
$$
whose characteristic polynomial has three roots, which depend on $\boldsymbol{\mu}$.
Although these three functions have an intractable analytical expression,
it is possible to show the existence of 
$\boldsymbol{\mu}_2\approx 105.04 \in \mathcal{I}$ such that   $D f_{\boldsymbol{\mu}}\left( \mathcal{O}_{\boldsymbol{\mu}} \right)$ exhibits a pair of  complex  (non-real) eigenvalues of the type $\alpha(\boldsymbol{\mu})\pm i \beta(\boldsymbol{\mu})$ such that $\alpha, \beta$ are $C^1$ maps, depend on $\boldsymbol{\mu}$ and: \\
\begin{enumerate}
	\item  $\beta(\boldsymbol{\mu}_2)>0$ (Figure~\ref{fig:Real_Egvl_and_Imaginary_part_Complex_int_eq_func_mu_2});
	\item $\alpha$  is positive for  $\boldsymbol{\mu}<\boldsymbol{\mu}_2$;
	\item $\alpha$ is negative for  $\boldsymbol{\mu}>\boldsymbol{\mu}_2$ (Figure~\ref{fig:Real_part_Complex_Egvls_int_eq_func_mu_2}).\\
\end{enumerate}
As suggested by Figure~\ref{fig:Real_part_Complex_Egvls_int_eq_func_mu_2} (right), the complex (non-real) eigenvalues cross the imaginary axis with  positive speed as $ \boldsymbol{\mu}$ passes through $\boldsymbol{\mu}_2$, confirming that:
$$
\frac{   d\, \alpha}{d\, \boldsymbol{\mu} }\Big|_{\boldsymbol{\mu} = \boldsymbol{\mu}_2} \neq 0. 
$$
  This means that at $\boldsymbol{\mu} =\boldsymbol{\mu}_2$, the equilibrium $\mathcal{O}_{\boldsymbol{\mu}}$ undergoes a  Hopf bifurcation (destroying an attracting periodic solution, say $\mathcal{C}_{\boldsymbol{\mu}}$).  \end{proof}

\begin{figure}[h]
	\includegraphics[width=6cm]{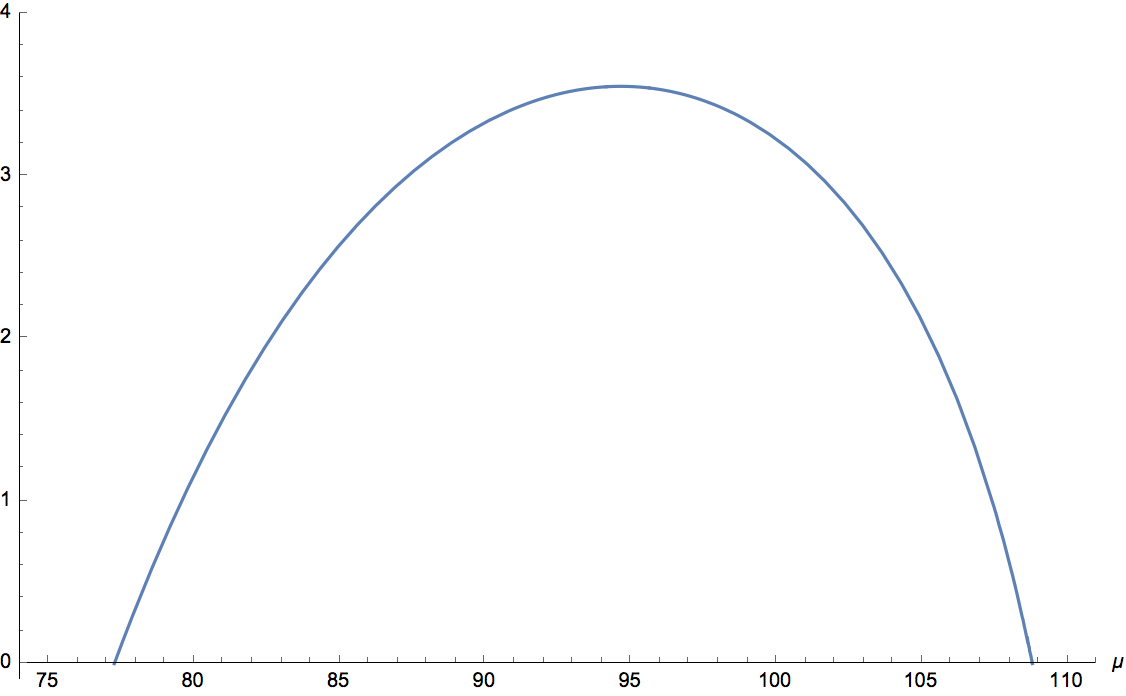}
	\hspace{.3cm}
	\includegraphics[width=6cm]{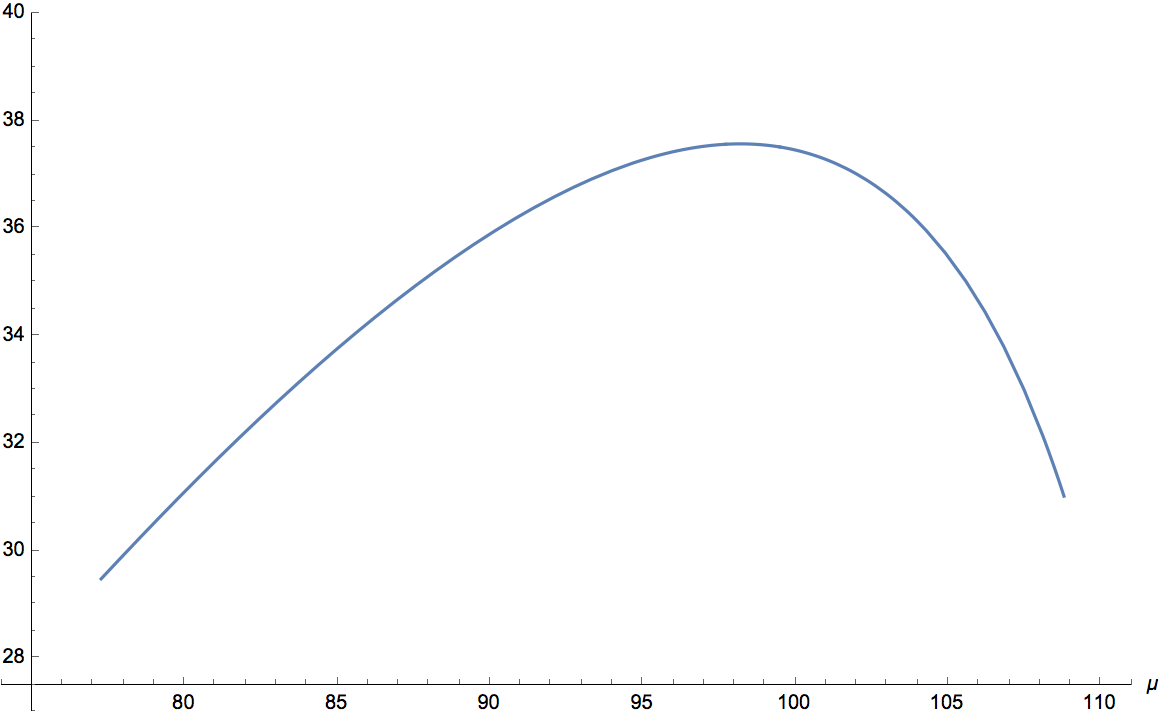}
    \caption{\small Graph of  the real  eigenvalue of $Df_{\boldsymbol{\mu}} \left(\mathcal{O}_{\boldsymbol{\mu}} \right)$ (left) and graph of  the imaginary part of the complex eigenvalues of $Df_{\boldsymbol{\mu}} \left(\mathcal{O}_{\boldsymbol{\mu}} \right)$ (right) where $\boldsymbol{\mu} \in \mathcal{I}$, for system~\eqref{eq:poly_rep_5}. } \label{fig:Real_Egvl_and_Imaginary_part_Complex_int_eq_func_mu_2}
\end{figure}

\begin{figure}[h]
	\includegraphics[width=13cm]{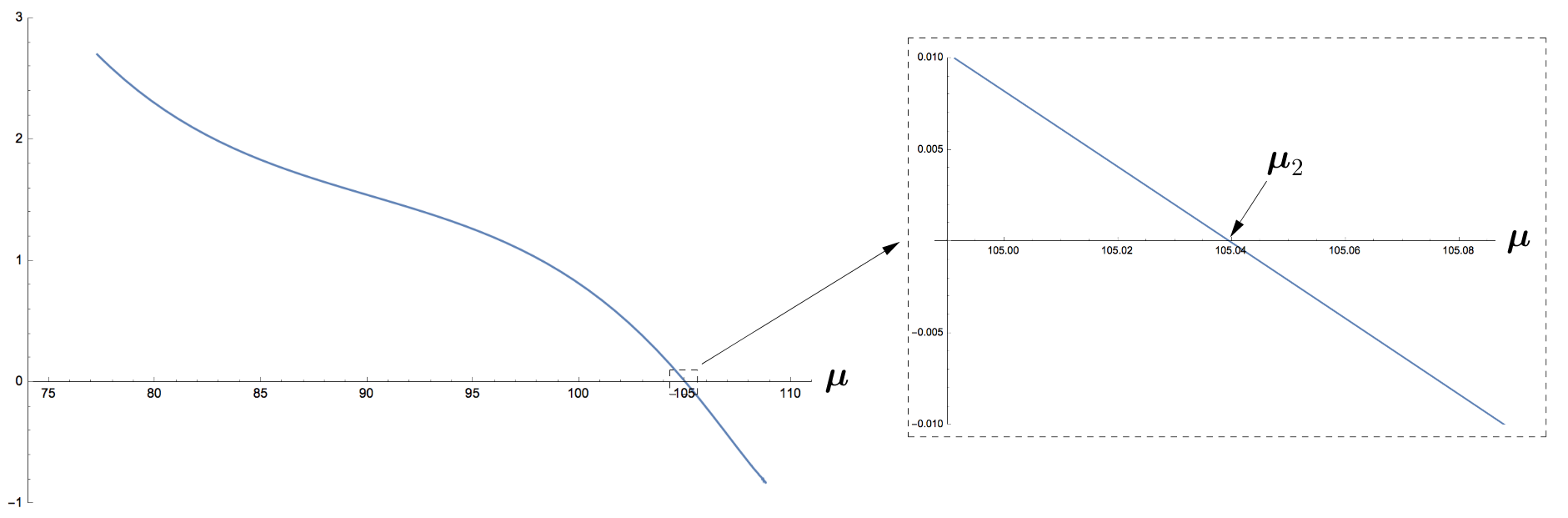}
    \caption{\small Graph of  the real part of the eigenvalues of $Df_{\boldsymbol{\mu}} \left(\mathcal{O}_{\boldsymbol{\mu}} \right)$ for $\boldsymbol{\mu} \in \mathcal{I}$ (left) and its zoom around $\boldsymbol{\mu}_2$, with $\boldsymbol{\mu} \in \left[ 104.99, 105.09 \right]$ (right), for system~\eqref{eq:poly_rep_5}.
 }
\label{fig:Real_part_Complex_Egvls_int_eq_func_mu_2}
\end{figure}

\subsection*{Terminology}\label{subsec:termin}
For $\boldsymbol{\mu} < \boldsymbol{\mu}_2$ and $j=1,2$, the interior equilibrium $\mathcal{O}_{\boldsymbol{\mu}}$ is a source and the tangent space $\RR^3$ may be   decomposed as two $Df_{\boldsymbol{\mu}}$--invariant subspaces $E^u_1$ and $E^u_2$ (in direct sum) 
such that $\dim E^u_j=j$. The set $E^u_1$ is  the eigendirection associated to the real positive eigenvalue and $E^u_2$ is associated to the complex (non-real) eigenvalues. We denote by $W^{u}_2(\mathcal{O}_{\boldsymbol{\mu}})$ the part of the invariant manifold whose tangent space at $\mathcal{O}_{\boldsymbol{\mu}}$ is $E^{u}_2$. 
Let
$$
\mathcal{I}_1=\left[ \frac{850}{11}, \boldsymbol{\mu}_1 \right[,
\quad
\mathcal{I}_2=\left] \boldsymbol{\mu}_1, \boldsymbol{\mu}_2 \right[,
\quad\textrm{and}\quad
\mathcal{I}_3=\left] \boldsymbol{\mu}_2, \frac{544}{5} \right],
$$
where $\boldsymbol{\mu}_1 = 102$ and $\boldsymbol{\mu}_2\approx 105.04$. We have that\\
\begin{eqnarray*}
\mathcal{I}&=&\mathcal{I}_1 \cup \{ \boldsymbol{\mu}_1  \} \cup \mathcal{I}_2 \cup \{ \boldsymbol{\mu}_2  \} \cup \mathcal{I}_3 \\\\
 \boldsymbol{\mu}_1=102& \mapsto&  \text{Transcritical bifurcation of } B_1 \text{  and  } B_2 \, \,  (\text{Lemma \ref{lem:B_1 and B_2 stability}});\\\\
  \boldsymbol{\mu}_2\approx 105.04 & \mapsto &  \text{Supercritical Hopf bifurcation of } \mathcal{O}_{\boldsymbol{\mu}} \text{ destroying} \\
  &  &  \text{the  periodic solution } C_{ \boldsymbol{\mu}} \, \,   (\text{Lemma \ref{lem:jacobian_eigenv}}). \\
\end{eqnarray*}

\subsection{Heteroclinic network}
\label{ss:heteroclinic}
In this subsection, we show that \eqref{eq:poly_rep_5} exhibits a heteroclinic network   formed by six cycles. 
\begin{lem}
\label{Lem:het_cycles}
For $ \boldsymbol{\mu}\in \mathcal{I}$,
the flow associated to~\eqref{eq:poly_rep_5} has six heteroclinic cycles whose connections are associated to the following set of equilibria (Figure~\ref{fig:het_cycles}):
\begin{enumerate}
	\item $\mathcal{H}_1:=\{v_2,v_4,v_8,v_6,v_2\}$;
	\item $\mathcal{H}_2:=\{v_1,v_3,v_4,v_8,v_6,v_2,v_1\}$;
	\item $\mathcal{H}_3:=\{v_1,v_3,v_4,v_8,v_6,v_5,v_1\}$;
	\item $\mathcal{H}_4:=\{v_1,v_3,v_7,v_8,v_6,v_2,v_1\}$;
	\item $\mathcal{H}_5:=\{v_1,v_3,v_7,v_8,v_6,v_5,v_1\}$;
	\item $\mathcal{H}_6: =\{v_1,v_3,v_7,v_6,v_1\}$.
\end{enumerate}
\end{lem}

\begin{proof}
Since there are no  equilibria on the edges besides the vertices, analysing the eigenvalues of system~\eqref{eq:poly_rep_5} at the vertices (see Table~\ref{tbl:Eigenv_of_vs}), the result follows.
\end{proof}

From now on, denote by $\mathcal{H}$ the heteroclinic network $\mathcal{H}_1\cup\dots \cup\mathcal{H}_6$.

\begin{figure}[h]
\includegraphics[width=13cm]{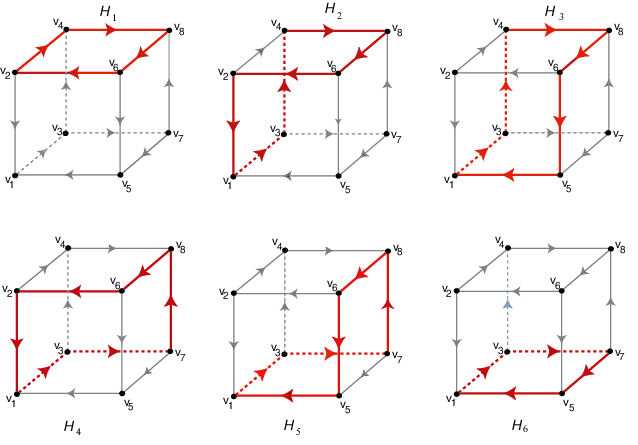}
\caption{\small{\emph{Illustration of $\mathcal{H}$ and its heteroclinic cycles  on $\partial [0,1]^3$, for $\boldsymbol{\mu} \in \mathcal{I}$.  }}
 } \label{fig:het_cycles}
\end{figure}

\begin{lem}
\label{Lem:switching}
For $ \boldsymbol{\mu}\in \mathcal{I}$,  the equilibria $v_2, v_3, v_6, v_7$ are switching nodes for system  \eqref{eq:poly_rep_5}. \\
\end{lem}

\begin{proof}
The proof follows from observing Table \ref{tbl:Eigenv_of_vs}. At these equilibria there are two positive real eigenvalues ($\Leftrightarrow$ two arrows leave the equilibrium in the corresponding graph).
\end{proof}

\subsection*{Numerics }
We list some numerical evidences, hereafter called by \emph{Facts}, about system~\eqref{eq:poly_rep_5}.  \\

\begin{fact}\label{fact:no_more_invariants}
For $\boldsymbol{\mu} \in \mathcal{I}\backslash\{102\}$, there exists an open $2$-dimensional invariant manifold $\mathcal{M}_{\boldsymbol{\mu}}$ containing $W^{s,u}_2 \left( \mathcal{O}_{\boldsymbol{\mu}} \right)$,  such that $\mathcal{H} \subset \overline{\mathcal{M}_{\boldsymbol{\mu}}}$ and there are no more compact invariant sets in $\inter \left( [0,1]^3 \right)\setminus \mathcal{M}_{\boldsymbol{\mu}}$. \\
\end{fact}

\begin{fact}\label{fact: 1D heteroclinic connection}
For $\boldsymbol{\mu} \in\mathcal{I} $, there are two  one-dimensional heteroclinic connections $[\mathcal{O}_{\boldsymbol{\mu}} \rightarrow B_1]$ and $[\mathcal{O}_{\boldsymbol{\mu}} \rightarrow B_2]$.\\
\end{fact}

It is possible to observe numerically that  $\mathcal{M}_{\boldsymbol{\mu}}$ of Fact \ref{fact:no_more_invariants}:
\begin{itemize}
\item coincides with $W^u_2(\mathcal{O}_{\boldsymbol{\mu}} )$ for $\boldsymbol{\mu} \in \mathcal{I}_1$;
\item contains $\{\mathcal{O}_{\boldsymbol{\mu}}\} \cup W^s (\mathcal{C}_{\boldsymbol{\mu}} )$ for $\boldsymbol{\mu} \in \mathcal{I}_2$, where $\mathcal{C}_{\boldsymbol{\mu}} $ is the periodic solution  associated to the Hopf Bifurcation described in Lemma \ref{lem:jacobian_eigenv};
\item coincides with $W^s(\mathcal{O}_{\boldsymbol{\mu}} )$ for $\boldsymbol{\mu} \in \mathcal{I}_3$.
\end{itemize}
We do not explore  the dynamics within the surface $\mathcal{M}_{\boldsymbol{\mu}}$ because it will not be used in the sequel. \\ 

\subsection{Stability of $\mathcal{H}$}
The next result asserts that the   network  $\mathcal{H}$ is globally asymptotically stable in  $[0,1]^3 \backslash\{ \mathcal{M}_{\boldsymbol{\mu}}, \left[\mathcal{O}_{\boldsymbol{\mu}}\to B_1\right], \left[\mathcal{O}_{\boldsymbol{\mu}}\to B_2\right] \}$, for  $\boldsymbol{\mu}\in\mathcal{I}_1$. \\

\begin{lem}
\label{Lem:asymptotic_stable_cycles}
For $\boldsymbol{\mu}\in\mathcal{I}_1$, $\mathcal{B}(\mathcal{H})= [0,1]^3\backslash\{ \mathcal{M}_{\boldsymbol{\mu}}, \left[\mathcal{O}_{\boldsymbol{\mu}}\to B_1\right], \left[\mathcal{O}_{\boldsymbol{\mu}}\to B_2\right] \}$.
\end{lem}

\begin{proof}
 If $u_0\in \inter([0,1]^3) \backslash\{ \mathcal{M}_{\boldsymbol{\mu}}, \left[\mathcal{O}_{\boldsymbol{\mu}}\to B_1\right], \left[\mathcal{O}_{\boldsymbol{\mu}}\to B_2\right] \}$,  then $\varphi(t,u_0)$  accumulates on a compact invariant set. Since  there are no more invariant sets in $\inter \left( [0,1]^3 \right)\setminus \mathcal{M}_{\boldsymbol{\mu}}$ (Fact \ref{fact:no_more_invariants}), then $\varphi(t,u_0)$  accumulates on the boundary's cube. Since the equilibria  $B_1$ and $B_2$ are sources in the corresponding faces (Lemma \ref{lem:B_1 and B_2 stability}), the result follows. \\
 \end{proof}

\bigbreak

 \subsection{Questions}
At the moment, motivated by numerical simulations, there are questions that are worth to be answered concerning the dynamics of ~\eqref{eq:poly_rep_5}. \\

\begin{description}
\item[\textbf{1st}]  For $\boldsymbol{\mu} \in \mathcal{I}_1$, the network  $\mathcal{H}$ is globally asymptotically stable  in $[0,1]^3\backslash\{ \mathcal{M}_{\boldsymbol{\mu}}, \left[\mathcal{O}_{\boldsymbol{\mu}}\to B_1\right], \left[\mathcal{O}_{\boldsymbol{\mu}}\to B_2\right] \}$.  
What is the likely limit set of $ \mathcal{H}$? In other words, is there some preferred cycle to where Lebesgue-almost all solutions are attracted? \\

\item[\textbf{2nd}] For $\boldsymbol{\mu} \in \mathcal{I}_2\cup  \mathcal{I}_3 $, the network $\mathcal{H}$ is not asymptotically stable, Lebesgue-almost all points in $ [0,1]^3\backslash \mathcal{M}_{\boldsymbol{\mu}}$ are attracted to $B_1$ and $B_2$, and $\overline{\mathcal{M}_{\boldsymbol{\mu}}}$ seems to accumulate on a cycle of $ \mathcal{H}$. Could we describe which one? \\

\end{description}

In the following sections, we develop a general method to answer the previous  questions. Although we describe  a  technique implemented to model described in Section \ref{sec:model}, the (affirmative) answers to the questions  are given as a series of results that are applicable to other networks  of polymatrix replicators  and to more general types of networks.

\section{Asymptotic dynamics: the theory}
\label{sec:asym_dyn}

We describe a piecewise linear model from where we may analyse the dynamics associated to the asymptotic dynamics near the heteroclinic network $\mathcal{H}$ of Lemma \ref{Lem:het_cycles}. This piecewise linear map is easily computed.
Here, we study the system \eqref{eq:poly_rep_4} bearing in mind that it is equivalent to \eqref{eq:poly_rep_5}, as observed at the end of Section \ref{sec:model}. The extension of the theory to other attracting networks is straightforward.

\begin{figure}[h]
\includegraphics[width=10cm]{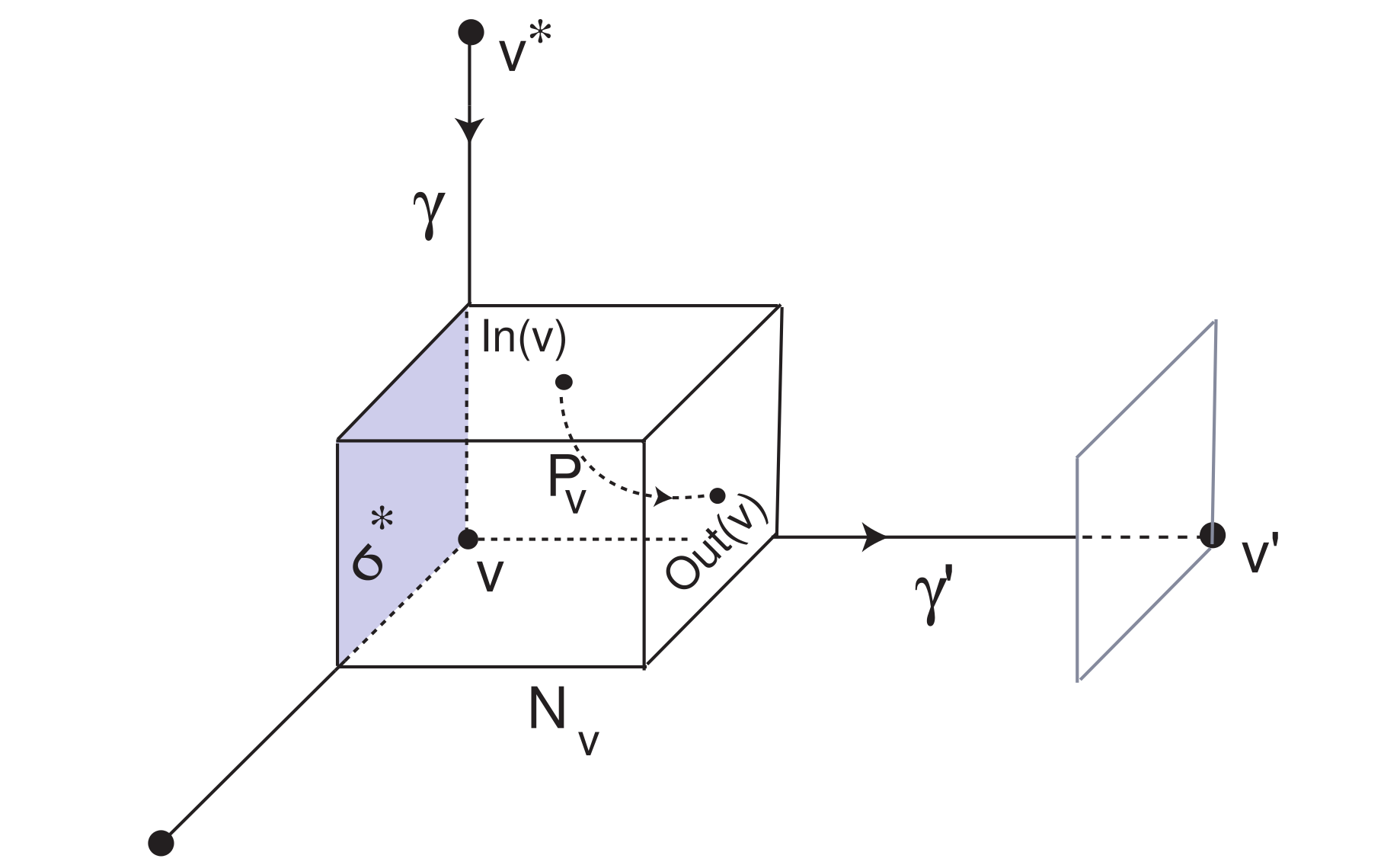}
\caption{\small Illustration of $N_v$ and the \emph{local map} $P_v:In(v)\rightarrow Out(v)$. The face $\sigma_\ast$ is orthogonal to $\gamma'$ at $v$ . } 
\label{fig:Terminology1}
\end{figure}

\subsection{Non-resonance hypothesis}\label{subsec:terminology}
Let $\mathcal{H}\subset \Gamma_{(2,2,2)}\subset \RR^6 $ be a heteroclinic network associated to the set of hyperbolic saddles $\mathcal{V}=\{v_1, ..., v_8\}$ and one-dimensional heteroclinic connections $\mathcal{E}$. 
\medbreak

Given $v=(x_1, x_2, x_3, x_4, x_5, x_6) \in \mathcal{V}$, we denote by $\mathcal{F}_v$ the set of three faces $\sigma_j$ with $j\in\{1, ...,6\}$, for which the component $x_j$ of $v$ are zero. Geometrically, this means that for each $v\in\mathcal{V}$, $\mathcal{F}_v$ is the set of the three faces  whose intersection is $v$. All saddles lying in $\mathcal{V}$ are of saddle-type and hyperbolic (cf. Table \ref{tbl:Eigenv_of_vs}). From now on, we assume the following technical condition:
\bigbreak

\textbf{(TH)} For each $v\in \mathcal{V}$ the eigenvalues of $Df_{\boldsymbol{\mu}} (v)$ are non-resonant in the terminology of Ruelle \cite{ruelle2014elements}:\footnote{This hypothesis is equivalent to the Condition (c) of Definition 3.1 of \cite{alishah2019asymptotic}.}  
$$
Re(\lambda_i) = Re(\lambda_j) + Re(\lambda_k),
$$
where $Re(\lambda)$ denotes the real part of $\lambda\in \mathbb{C}$ and $\lambda_i, \lambda_j$ and $\lambda_k$ are the eigenvalues of the linear part of the vector field  \eqref{eq:poly_rep_4} evaluated at the equilibrium $v\in \mathcal{V}$.
\bigbreak
The necessary and sufficient conditions for $C^1$--linearization of Ruelle show that linearization is not possible for subsets of points on the lines described by the restrictions above. These restrictions correspond to a set of zero Lebesgue measure in parameter space and place no serious constraint on the analysis that follows.

 \subsection{$C^1$--Linearization and global map}
 \label{ss: global map}
Since  $v\in \mathcal{V}$ is hyperbolic, assuming the non-resonance condition \textbf{(TH)} of ${Df_{\boldsymbol{\mu}}}(v)$, it is
possible to define an open cubic neighbourhood of $v$,  $N_v $, such that the flow associated to  \eqref{eq:poly_rep_4}  is $C^1$--conjugated to that of $\dot{x}={Df_{\boldsymbol{\mu}}}(v)(x-v)$, $x\in \RR^6$.
  In particular, it is possible to define two cross sections, $In(v)\subset \overline{N_v}$ and $Out(v)\subset \overline{N_v}$,
such that solutions  starting in  $In(v)\setminus W^s(v)$ enter in $ {N}_v$ in positive time,  spend some time there and leaves the cube through $Out(v)$ -- see Figure \ref{fig:Terminology1}. It induces the local diffeomorphism:
$$
P_v\,:\, In(v)\setminus W^s(v) \to Out(v).
$$

Using local adapted coordinates associated to system  \eqref{eq:poly_rep_4}, the cubic neighbourhood $N_v$ may be defined by: 
\begin{equation}
\label{Nv_def}
N_v :=\{p \in \Gamma_{(2,2,2)}: 0\ < x_j(p)< 1 \qquad \text{for} \qquad 1\leq j \leq 6\}
\end{equation}
where $(x_1,x_2, x_3,x_4,x_5, x_6)$ is a system of linear coordinates around $v$ which assigns coordinates $(0, 0, 0,0, 0, 0)$  to $v$.

\bigbreak
For $  v^*, v \in \mathcal{V}$, given a one-dimensional heteroclinic connection of the type $\gamma:=[v^* \rightarrow v]$, we may also define an invertible  map from a small neighbourhood of $Out(v^*)\cap \gamma$ to $In(v)\cap \gamma$, that is called the \emph{global map} and will be denoted by $P_{\gamma}$. This map is a diffeomorphism \cite[Ch. 2]{palis1982local} and is depicted in Figure \ref{fig:Terminology2}.
\bigbreak
 Let $\mathcal{T}_\varepsilon$ a tubular neighbourhood of $\mathcal{H}$. It can be written as the ``\emph{system of connected pipes}'':
\begin{equation}
\label{tubular_neigh1}
\mathcal{T}_\varepsilon= \left(\bigcup_{v\in \mathcal{V}}N_v\right)  \bigcup \left(\bigcup_{\gamma \in \mathcal{E}}N_\gamma\right) 
\end{equation}
where:
\begin{itemize}
\item  $N_v$ is the neighbourhood of $v$ (see \eqref{Nv_def});
\item $N_\gamma$ is the tubular neighbourhood of $\gamma\in \mathcal{E}$ (of radius $\varepsilon$) defined by:
\begin{eqnarray*}
N_\gamma= \{q\in \Gamma_{(2,2,2)}\backslash (N_{v^*}\cup N_v) &:& x_j \leq \varepsilon \quad \text{for all }\,  j\,  \text{such that}\,  \gamma\subset \sigma_j \quad \text{and}\\
 &&    x_j =0\quad \text{for all }\,  j\,  \text{such that}\,   \sigma_j \not\subset \gamma \}\\
\end{eqnarray*} 
 \end{itemize}

\begin{rem}
In order to define correctly the set $\mathcal{T}_\varepsilon$ we might need to shrink either the cubic neighbourhoods of the saddles or the tubular neighbourhoods of the connections. This is possible by decreasing $\varepsilon$ finitely many times (if necessary). 
\end{rem}

 \begin{figure}[h]
\includegraphics[width=12cm]{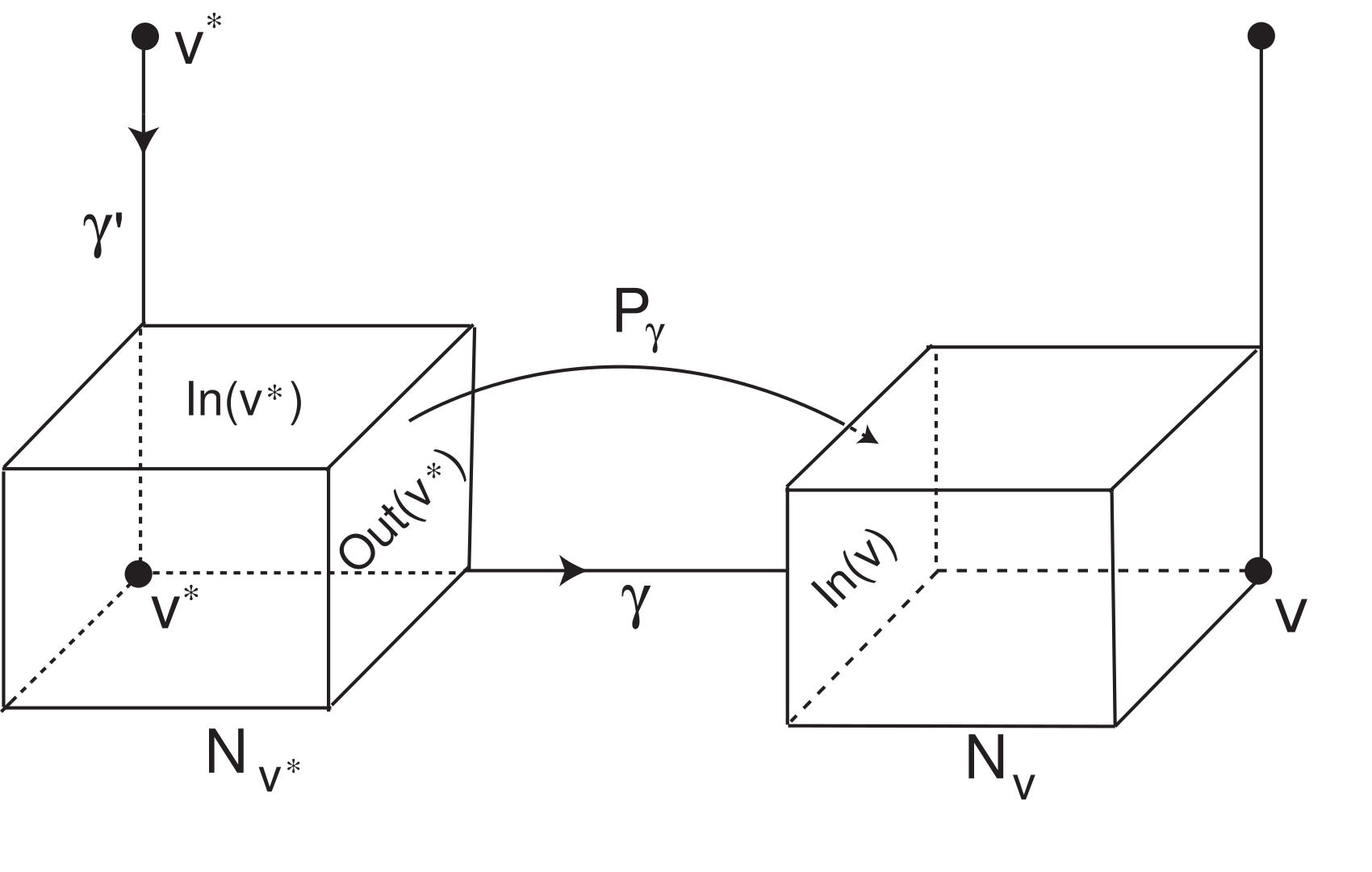}
\caption{\small Sketch of the \emph{global map} $P_{\gamma}: Out(v^*) \rightarrow In(v)$. } 
\label{fig:Terminology2}
\end{figure}

 \subsection{Quasi-change of coordinates}\label{subsec:dual_cone}

We introduce now the stage where the asymptotic piecewise linear dynamics play its role. This space is a subset of $\Rr^6$ and may be seen as a finite union of subsets of $\RR_+^6$, each one called by \emph{sector}.

We describe a rescaling change of coordinates $\Psi_\varepsilon$, depending on the parameter $\varepsilon>0$.
Since the tubular neighbourhood $\mathcal{T}_\varepsilon$ may be written as in \eqref{tubular_neigh1}, the map $\Psi_\varepsilon$ acts in different ways according to the point $q$ lies on $N_v$ or in $N_\gamma$, where $v\in \mathcal{V}$, $\gamma\in \mathcal{E}$. The variable $\varepsilon$ plays the role of \emph{blow-up} parameter as we proceed to explain. The examples are related with  system \eqref{eq:poly_rep_4} and the index $j$ runs over the set $\{1, ..., 6\}$.

\subsubsection{Action of $\Psi_\varepsilon$ on $N_v$}
In the first case, if $q\in N_v $, the rescaling change of coordinates $\Psi_\varepsilon$ takes points $q= (x_1, x_2, x_3, x_4, x_5, x_6)$  to points in the sector $\{(y_j)_{\sigma_j \in \mathcal{F}}\}$ according to the law: \\
\begin{itemize}
\item   $y_j= -\varepsilon^2 \log x_j(q) \geq 0$ if the face $\sigma_j$  contains $v$  {(for all $\sigma_j\in \mathcal{F}_v$)}; \\
\item    $y_j = 0$ if  the face $\sigma_j$ does not contain $v$  {(for all $\sigma_j\notin \mathcal{F}_v$)}.\\
\end{itemize}

\textbf{Example:} 
Assume we have enumerated $\mathcal{F}_v$ so that the faces through $v=(1, 0, 1, 0, 1,0)$ are precisely
$\sigma_2, \sigma_4, \sigma_6$. The map $\Psi_\varepsilon$ is defined on the neighbourhood $N_v\backslash  \Gamma_{(2,2,2)}$ by
$$
\Psi_\varepsilon(q)= (0,-\varepsilon^2 \log x_2(q), 0, -\varepsilon^2 \log x_4(q),0,  -\varepsilon^2 \log x_6(q))
$$
where $(x_1, x_2, x_3, x_4, x_5, x_6)$ stands for the system of affine coordinates introduced above. \\

\textbf{Notation:} $\Pi_v:=\Psi_\varepsilon(N_v) =\{(u_j)_j \in \RR_+^6: u_j=0, \quad \forall \sigma_j \notin \mathcal{F}_v\}$  is well defined as a 3-dimensional subset of $(\RR^+_0)^6$.

\subsubsection{Action of $\Psi_\varepsilon$ on $N_\gamma$}
Similarly, given an edge $\gamma=[v^* \rightarrow v]$, the map $\Psi_\varepsilon$ takes points in the neighbourhood $N_\gamma$ of $\gamma$ to points in the sector $\{(y_j)_{\sigma_j\in \mathcal{F}} \}$ such that:\\
\begin{itemize}
\item   $y_j= -\varepsilon^2 \log x_j(q) \geq 0$ if the face $\sigma_j$  contains $\gamma$; \\
\item   $y_j=0$ if the face $\sigma_j$ does not contain $\gamma$. \\
\end{itemize}

\bigbreak
\textbf{Example:} 
For $\gamma_5=[v_1\rightarrow v_3]$ we know that $\gamma_5=\sigma_2 \cap \sigma_6$. If $q\in N_\gamma\backslash \Gamma_{(2,2,2)}$, then  the expression of the map $\Psi_\varepsilon$ is  is given by:
$$
\Psi_\varepsilon(q)= (0,-\varepsilon^2 \log x_2(q), 0,  0, 0,  -\varepsilon^2 \log x_6(q)),
$$
where $(x_1, x_2, x_3, x_4, x_5, x_6)$ stands for the system of affine coordinates introduced above. \\

\textbf{Notation:} $\Pi_\gamma:=\Psi_\varepsilon({N_\gamma})=\Pi_{v^*}\, \cap\,  \Pi_v$  is a 2-dimensional subset of $(\RR^+_0)^6$. 

\begin{rem}
 Observe  that 
 \begin{equation}
 \label{rank1}
 rank(\Psi_\varepsilon({N_v}))=3 \qquad \text{and} \qquad rank(\Psi_\varepsilon({N_\gamma}))=2.
 \end{equation} 
 In particular, the map $\Psi_\varepsilon $ is not injective when restricted to  $N_\gamma$.  We know precisely how the loss of injectivity  is performed; the map $\Psi_\varepsilon|_{N_\gamma}$ identify all points in the same trajectory on $\Gamma_{(2,2,2)}$.  This loss of injectivity will not affect the validity of our results. This is why we say that the map is $\Psi_\varepsilon$ is a \emph{quasi-change of coordinates}.
\end{rem}

\begin{defn}
The \emph{dual cone associated to the network $\mathcal{H}$} is given by $\bigcup_{v\in \mathcal{V}}\Pi_v$. 
\end{defn}
The map $\Psi_\varepsilon$ is not well defined in  $\partial \Gamma_{(2,2,2)}$. When a trajectory is approaching the network $\mathcal{H}$, the non-zero coordinates of its image under $\Psi_\varepsilon $ go to $\infty$ in the dual cone. This is why we say that $\varepsilon>0$ plays the role of blow-up parameter.

 \bigbreak

\subsection{Skeleton character at an equilibrium}
\label{ss:skeleton}
For $v\in \mathcal{V}$, the main result of this subsection relates the asymptotic dynamics of $(\Psi_{\varepsilon})_* f$, the \emph{push-forward} of $f$ by $\Psi_{\varepsilon}$  (restricted to $N_v$), with a constant vector field on the dual cone. We omit the dependence of $f$ on $\boldsymbol{\mu}\in \mathcal{I}$ to lighten the notation. Let us see the definition of this constant vector field:

\begin{defn}\label{def:character_1}
For a given $v\in \mathcal{V}$, we define the  map $\chi^v$  as:
\begin{equation} 
\label{def:caracter}
\chi^v_j=\left \{ \begin{array}{l}
- \,\text{eigenvalue of }Df  (v) \text{ in the orthogonal direction to $\sigma_j$}, \,\, \text{if}\,\, \sigma_j \in \mathcal{F}_v  \\[2mm]
0, \,\, \text{otherwise}
\end{array} \right. ,
\end{equation}
where $j\in \{1, ..., 6\}$ is the component of the vector. 
 For an equilibrium $v\in \mathcal{V}$, the vector field $\chi^v=(\chi_j^v)_{j\in \{1, ..., 6\}}$ is called the \emph{skeleton character at $v$}. Note that for each $v\in \mathcal{V}$, three components of this map are zero. 
\end{defn}

The next result asserts that the vector field $(\Psi_{\varepsilon})_* f $ rescaled by the factor $\varepsilon^{-2}$ converges to the constant vector field $\chi^v$ on the subspace $\Pi_v$. In particular the trajectories associated to the push-forward vector field  $\varepsilon^{-2}(\Psi_{\varepsilon})_\ast f$   are asymptotically linearized to lines \emph{i.e.} there exists $T>0$ such that the solution with initial condition $y\in \Pi_v$ is the  segment defined by $y+t\chi^v $, $t\in [0, T]$, $y\in \Pi_v$.
 
\bigbreak
In order to be precise in the results' statement, we introduce the following definition.
 \begin{defn}
 For $\tau>0$, let $(F_\lambda)_{\lambda\in [0,\tau]}$ be  a one-parameter family of maps  defined on $\mathcal{D}\subset (\RR_0^+)^6$, and $F$ be another function with the same domain.   We say that 
 $  F_\lambda$ \emph{converges in the $C^1$--topology to} $F$,  as $\lambda$ tends to $0$, and we write
 $$
 \lim_{\lambda\rightarrow 0}F_\lambda=F,
 $$ 
  to mean that for every compact set $K\subset \mathcal{D}$, the following equality holds:
  $$\lim_{\lambda \rightarrow 0^+} \max \left\{        \sup_{u\in K} [F_\lambda(u)-F(u)] ,  \, \,        \sup_{u\in K}  D   [F_\lambda(u)-F(u)]      \right\}=0,$$
  where $D$ denotes the usual first order \emph{Fr\'echet derivative}.
 \end{defn}
 
 \bigbreak If a map is the composition of finitely many maps, the domain should be understood as the domain where the composition is well defined. From now on, let us define (in the dual cone): 
 \begin{eqnarray*}
 \Pi_v(\varepsilon) &=&\{y \in \Pi_v: \quad y_j \geq \varepsilon, \qquad \forall \sigma_j\in \mathcal{F}_v\},\\
 \Pi_\gamma(\varepsilon) &=&\{y \in \Pi_\gamma : \quad y_j \geq \varepsilon, \qquad \forall \sigma_j\in \mathcal{F}_{v^\star}\cap \mathcal{F}_v\}.
 \end{eqnarray*} 

We omit the dependence on $\varepsilon$ of $ \Pi_v(\varepsilon)$ and $ \Pi_\gamma(\varepsilon)$ to lighten the reading.

\medbreak
 
In order to get an approximation of Lemmas \ref{prop:limit} and \ref{global in the dual} in topology $C^r$, $r>1$, we might need to rescale the radius of $\mathcal{T}_\varepsilon$ defined in \eqref{tubular_neigh1}. This is not necessary to the scope of the present work since  conclusions on stability of cycles hold in the $C^1$--topology. We now state the main result of this Subsection.
 
\bigbreak
 
\begin{lem}~\cite[Lemma 5.6]{alishah2019asymptotic}\label{prop:limit}
The following equality holds for $v\in \mathcal{V}$:
$$
 \lim_{\varepsilon\rightarrow 0} \varepsilon^{-2} (\Psi_\varepsilon)_\ast f \, |_{\Pi_v (\varepsilon) }=\chi^v.
 $$ 
\end{lem}

\subsection{Global map ``viewed'' in the dual cone}
\label{ss:5.5}
 For $v_*, v  \in \mathcal{V}$ and  $\gamma:= [v^* \rightarrow v]$, let  $$P_\gamma:Out(v_*)\rightarrow In(v)$$ be the diffeomorphism defined in Subsection \ref{ss: global map}. Define the map:
$$
H^\varepsilon :=  \Psi_\varepsilon \circ P_\gamma \circ  (\Psi_\varepsilon)^{-1}:  \Psi_\varepsilon(Out(v^*)) \rightarrow \Psi_\varepsilon(In(v)).
$$

 \begin{figure}[h]
 \centering
\includegraphics[width=13cm]{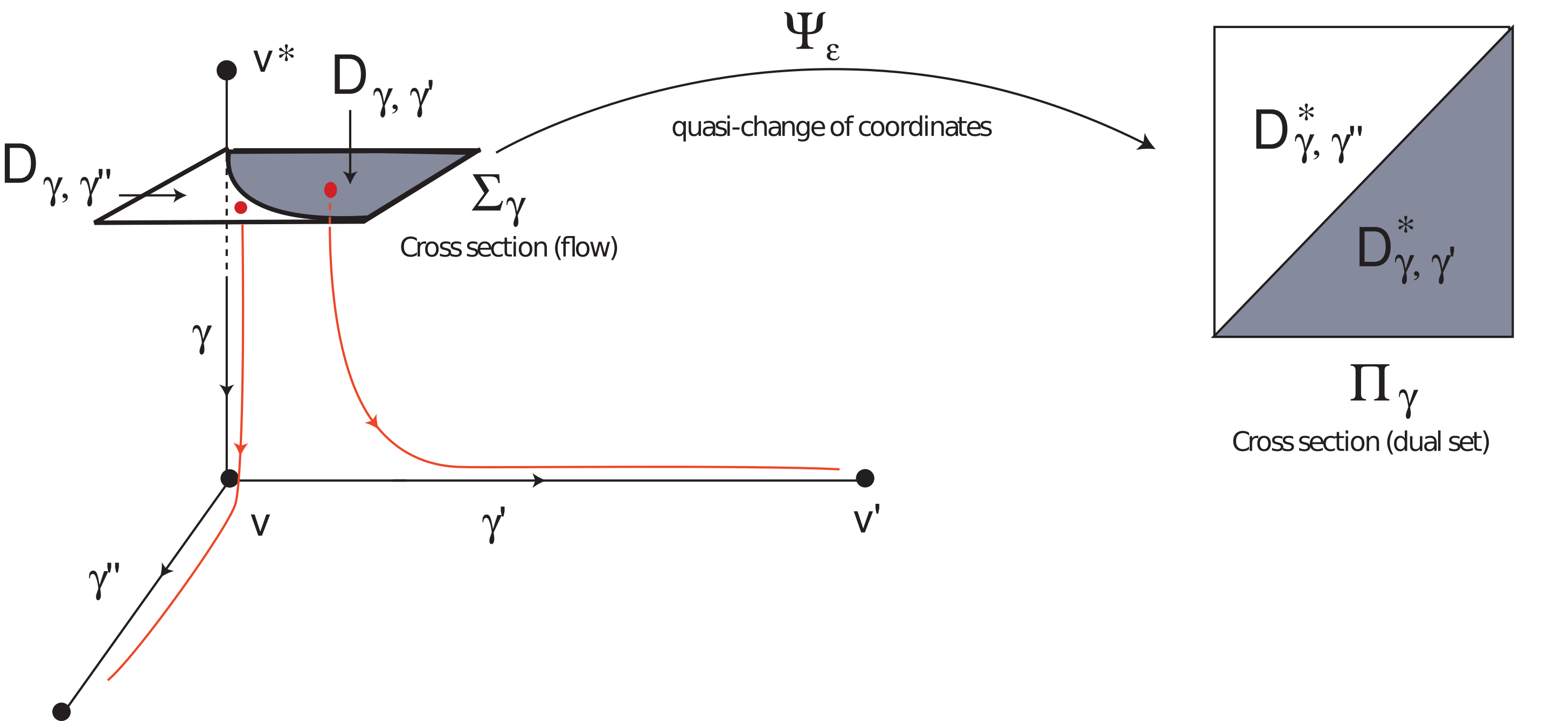}
\caption{\small Ilustration of $D_{\gamma, \gamma'}$ (left)  and its image under $\Psi_\varepsilon$ (right).  } 
\label{fig:Terminology3}
\end{figure}

The next result ensures that, although the original global map $P_\gamma$  is given by an invertible linear map (cf. Subsection \ref{ss: global map}), the map $H^\varepsilon $  
converges, in the $C^1$--topology, to the Identity map (denoted by \emph{Id}) as $\varepsilon \rightarrow 0$.

\begin{lem}\cite[Lemma 7.2]{alishah2019asymptotic} 
\label{global in the dual}
The following equality holds for $v, v^*\in \mathcal{V}$:
$$
 \lim_{\varepsilon\rightarrow 0} H^\varepsilon|_{\Psi_\varepsilon (Out(v^*))\cap \Pi_\gamma(\varepsilon)} =\text{Id }\,  |_{\Psi_\varepsilon (Out(v^*))\cap \Pi_\gamma(\varepsilon)}.
 $$

\end{lem}

Lemma \ref{global in the dual} says that, for any heteroclinic connection of the type $\gamma=[v^* \rightarrow v]$, we can
identify asymptotically the two sections $\Psi_\varepsilon (Out(v^*))$ and $\Psi_\varepsilon (In(v))$. We will refer  to the identified sections as the two-dimensional manifold $\Pi_\gamma(\varepsilon)$; it may be seen as $\Psi_\varepsilon(\Sigma_\gamma)$, where $\Sigma_\gamma$ is (any) cross section to $\gamma$, as depicted in Figure \ref{fig:Terminology3}.   

\bigbreak
Define $D_{\gamma, \gamma'}  $   the set of points in $\Sigma_\gamma $ that follows the connection $\gamma'=[v\rightarrow v']$ at a distance $\varepsilon>0$
and  set $$D^*_{\gamma, \gamma'}=\Psi_\varepsilon(D_{\gamma, \gamma'})\subset \Pi_\gamma(\varepsilon).$$

Let   $P_{\gamma, \gamma'}$ be the map that carries points from $D_{\gamma, \gamma'} \subset \Sigma_\gamma$ to $Out(v)\cap \gamma'$. For the admissible path   $\{\gamma, \gamma'\}$ defined as above, let  
$$
F_{\gamma, \gamma'}= \Psi_\varepsilon \circ P_{\gamma, \gamma'} \circ (\Psi_\varepsilon )^{-1}|_{D^*_{\gamma, \gamma'}}.
$$

For $\sigma_j\in \mathcal{F}_{v}$, denote by  $j_*$ the index of face  within $\mathcal{F}_v$  orthogonal to $\gamma'$.  
Consider the sector $\Pi_{\gamma, \gamma'}\subset \inter(\Pi_\gamma)$ defined as
\begin{equation}
\label{dominio2}
\Pi_{\gamma,\gamma'} := \left\{  y \in \inter(\Pi_\gamma): y_j>\frac{\chi_j^v}{\chi_{j_*}^{v}}y_{j_*}, \quad \forall j:\sigma_j\in \mathcal{F}_v, \quad \sigma_j\neq \sigma_{j_\ast}  \right\},
\end{equation}
containing all points in $\inter(\Pi_\gamma)$ whose image by $(\Psi_\varepsilon)^{-1}$ follow the admissible path $\{\gamma,\gamma'\}$ at a given positive (small) distance.

\begin{lem} 
 \label{lem: local dynamics}
 The following equality holds for the admissible path $\{\gamma, \gamma'\}$:
$$
\lim_{\varepsilon \rightarrow 0} F_{\gamma, \gamma'}= L_{\gamma, \gamma'}
$$
where 
$L_{\gamma, \gamma'} : \Pi_{\gamma, \gamma'}\to \Pi_{\gamma'}$ is the linear map defined by:
$$
L_{\gamma, \gamma'}(y) = \left( y_j - \frac{\chi_j^v}{\chi_{j_*}^v}y_{j_*}\right)_{\sigma_j \in \mathcal{F}}.
$$
\end{lem}

 \begin{figure}[h]
\includegraphics[width=13cm]{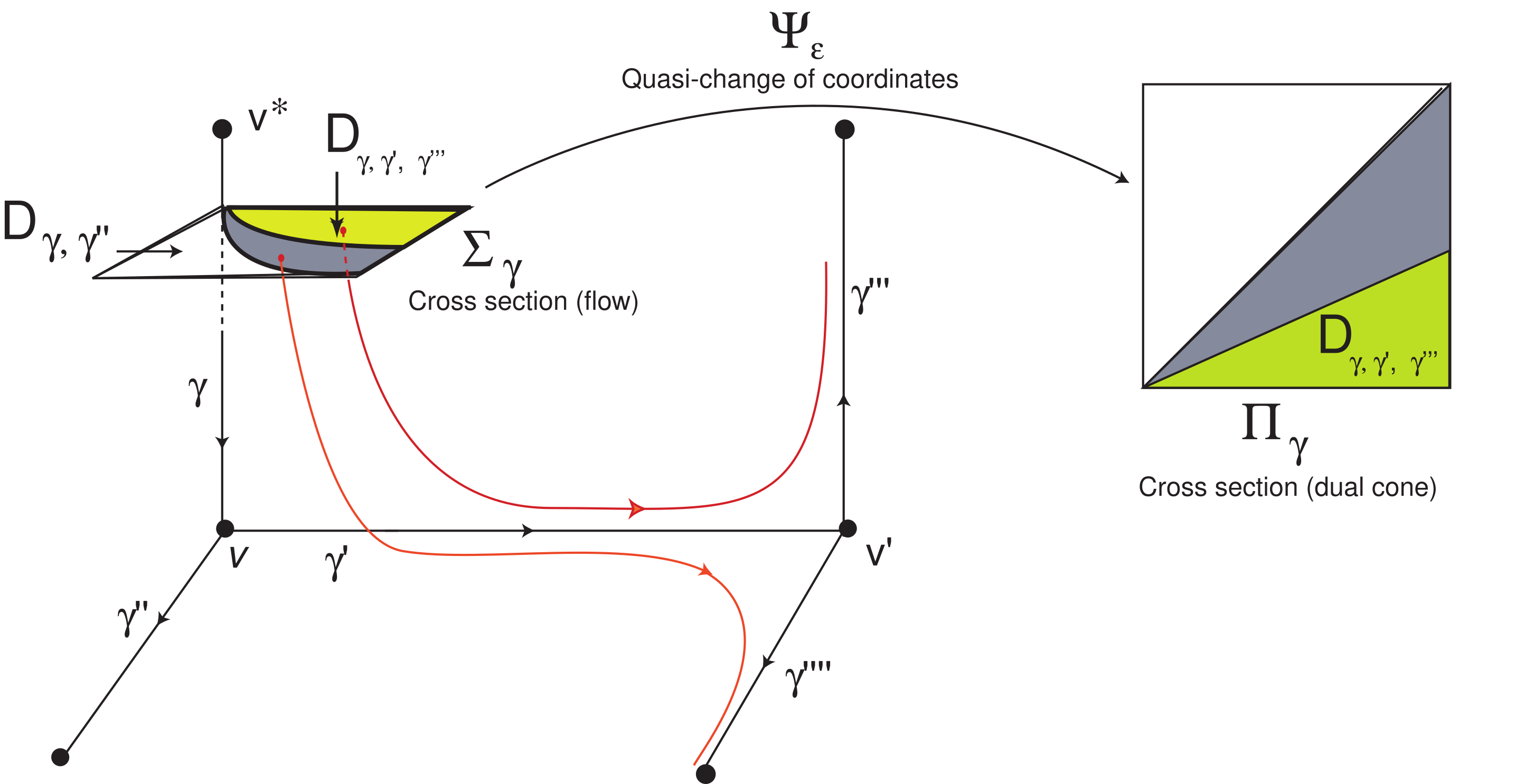}
\caption{\small Illustration of Lemma \ref{lem: local dynamics}, where $v$ is a switching node. Compare with Figure \ref{fig:Terminology3}, where $D_{\gamma, \gamma', \gamma'''}$ is a subset of $D_{\gamma, \gamma'}$.} 
\label{fig:Terminology4}
\end{figure}

\begin{proof}
 The proof of this result relies on the proof of Lemma \ref{prop:limit}.
We consider in $\Sigma_\gamma$ (cross section transverse to $\gamma$), the points  that follow the chain of heteroclinic connections
$$
\gamma=[v^* \to v], \quad \gamma'= [v\to v'].
$$
Observe that the equilibrium $v$ is a switching node of $\mathcal{H}$\footnote{If $v$ is not a switching node, the proof is much simpler. See the next ``Digestive Remark''.}. This means that $Df  (v) $ has two positive real eigenvalues, say $E_2,E_1$ where $E_2>E_1$, and one negative, say $-C$.

Let us consider a neighbourhood  $N_v$ 
and the coordinates $(x,y,z)$
in such a way that
$v\equiv (0,0,0)$, the axis $Ox$ is associated to the eigenvalue $E_1$,
the axis $Oy$ is associated to the eigenvalue $E_2$, and
the $Oz$ is associated to the eigenvalue $-C<0$.
Therefore, by \textbf{(TH)}, the system of ODEs that locally describes the vector field in  $N_v$,  is given by
\begin{equation} \label{eq:lineariz_1}
\left\{\begin{array}{l}
\dot{x}=E_1 x  \\ [1mm]
\dot{y}=E_2 y  \\ [1mm]
\dot{z}=-C z  
\end{array}
\right. , \quad \textrm{ where }  \,  E_2> E_1>0, \, \, C>0
\end{equation}
whose solution is
 \begin{equation} \label{eq:lineariz_1}
\left\{\begin{array}{l}
x(t)=x_0 e^{E_1 t}  \\ [1mm]
y(t)=y_0 e^{E_2 t}  \\ [1mm]
z(t)=z_0 e^{-C t}
\end{array}
\right. 
\end{equation}
and $(x_0, y_0, z_0)\in \RR_+^3$. The local map from the cross section  $In(v)=\{z=\varepsilon\}$ to the connected component of $Out(v)$ defined by $\{y=\varepsilon\}$ is given (in local coordinates $(x, y, \varepsilon)\equiv (x,y))$ by
$$
P_v (x,y)=\left( x_0 y_0^{-\frac{E_1}{E_2}}, y_0^{\frac{C}{E_2}} \right)
$$
and the associated time of flight is $$\frac{1}{E_2}\ln\left(\frac{\varepsilon}{y_0}\right) .$$The line defined by $x=1 \wedge y=1$ is the intersection of the two connected components of $Out(v)$. Noticing that $
x_0 y_0^{-\frac{E_1}{E_2}}>1 $ is equivalent to $ x_0 >y_0^{\frac{E_1}{E_2}},
$
one may define the region of points in $\{z=\varepsilon\}$ that follow the admissible path
$\{\gamma, \gamma'\}$ as
$$
y_{j}>\frac{E_1}{E_2}y_{j_{*}} \overset{(\S \ref{subsec:dual_cone} ) }{=} =\frac{\chi_j^v}{\chi_{j_*}^v} y_{j_*}
$$
and the result is proved.
\end{proof}

\subsection{Digestive remark}
For $v^*, v, v' \in \mathcal{V}$, we concentrate our attention in the following chain of heteroclinic connections:
\begin{equation}
\label{consecutive edges}
\gamma=[v^* \to v], \quad \gamma'= [v\to v']\quad \textrm{and}\quad  \gamma''= [v\to v'']
\end{equation}
where $v$ is a  \emph{switching node}.  Since $v$ is a switching node and  $Df_{\boldsymbol{\mu}} (v)$ has real eigenvalues,  up to a set of zero Lebesgue measure, the cross section  $\Sigma_\gamma$    is divided in two regions containing initial conditions that follow $ \gamma'$ and  $\gamma''$. These regions are disjoint cusps whose topological closure contains the origin. The map $\Psi_\varepsilon$ sends these cusps into triangles where the origin is one vertex (see also  \cite{rodrigues2017attractors}). 

Concatenating paths, the subset of $\Sigma_\gamma$ that realise an ``increased'' chain of heteroclinic connections give rise to a sequence of nested cusps containing the origin and then a sequence of nested triangles in the dual cone, as suggested by Figures \ref{fig:Terminology3} and \ref{fig:Terminology4}.

 If $v$ is not a switching node, then there are two incoming directions to $v$ and just one outcoming from $v$, which means that the inequaliy  of   \eqref{dominio2} does not impose any additional condition.

 \subsection{Heteroclinic cycle} 
 \label{ss:5.6}
For $m\in \NN$, given an admissible path of the type  $\xi = \{\gamma_0,\gamma_1,...,\gamma_m\}$,
with $v_0 = \alpha(\gamma_0)$ and $v_m = \alpha(\gamma_m)$,
the composition \\
$$
P_\xi := P_{\gamma_{m-1}, \gamma_m} \circ P_{\gamma_{m-1}} \circ   ...\circ P_{\gamma_{1}, \gamma_2}\circ
P_{\gamma_1} \circ P_{\gamma_{0}, \gamma_1}: D_0 \rightarrow Out(v_m)
$$
\medbreak
is the \textit{first return map to $Out(v_m)$} of solutions of  \eqref{eq:poly_rep_4} starting at $D_0\subset D_{\gamma_0, \gamma_1}$ and following $\xi$ at a distance $\varepsilon>0$. 
 It is the composition of local and global maps, when  well defined. The following result is a direct corollary of Lemma \ref{lem: local dynamics} and, roughly speaking,  asserts that the quasi-change of coordinates $\Psi_\varepsilon$ transforms the map $P_\xi$ into a piecewise linear map.

\begin{cor}
\label{lemma13}
For $m\in \NN$, given an admissible  path $\xi = \{\gamma_0,\gamma_1,...,\gamma_m\}$, let
$$
F_\xi= \Psi_{\varepsilon} \circ P_{\xi} \circ (\Psi_{\varepsilon} )^{-1}: \Pi_{\xi} \to \Pi_{\gamma_m},
$$
where
$$
\Pi_\xi:= \inter(\Pi_{\gamma_0}) \cap \bigcap_{j=1}^m \left( L_{\gamma_{j-1}, \gamma_j } \circ ...\circ L_{\gamma_0, \gamma_1} \right)^{-1}(\inter(\Pi_{\gamma_j})).
$$ 
Then
$$
\lim_{\varepsilon \rightarrow 0} F_{\xi}= L_{\gamma_{m-1}, \gamma_m}\circ ... \circ L_{\gamma_{0}, \gamma_1}=: \pi_{\xi}.
$$
\end{cor}

For every $y\in \Pi_\xi$, we have $\pi_\xi(y)\in  \inter(\Pi_{\gamma_m})$ and then there exists a solution of \eqref{eq:poly_rep_4} from $(\Psi_\varepsilon)^{-1}(y)$ to $(\Psi_\varepsilon)^{-1}(\pi_\xi(y))$ following the heteroclinic path $\xi$. 
The map $\pi_\xi= L_{\gamma_{m-1}, \gamma_m}\circ ... \circ L_{\gamma_{0}, \gamma_1}$ of Corollary \ref{lemma13}, designated by \textit{skeleton map along $\xi$}, is  an endomorphism in $\RR_+^6$ and induces  an invertible matrix 
\begin{equation}
\label{matrix1}
M_\xi= \left( \delta_{jk} - \frac{\chi_j^v}{\chi_{{j_*}}^v}\delta_{jk}\right)_{\sigma_j, \sigma_k \in \mathcal{F}},
\end{equation}
where $\delta$ represents the \emph{Kronecker delta operator}.
The matrix $M_\xi$  gives a  suitable  representation for computational purposes. From now on, recall that:
 \begin{eqnarray*}
\Pi_\xi &\mapsto & \text{subset of $\Pi_{\gamma_0}$ of initial conditions whose image under $\Psi_\varepsilon^{-1}$}  \\ 
 & &  \text{that follow the heteroclinic path $\xi$ at a distance $\varepsilon>0$;} \\ 
\pi_\xi &\mapsto & \text{linear map  from $\Pi_{\gamma_0}$ to $\Pi_{\gamma_m}$}. \\ 
 \end{eqnarray*}

\subsection{Dynamics of a linear operator}
\label{PF_review}
For the sake of completeness, we review the dynamics associated to a linear two-dimensional operator, which   follows from the \emph{Perron-Frobenius Theory} -- we address the reader to Chapter 1.9 of \cite{katok1997introduction} for more information on the subject. 
Suppose that $A$ is a linear map defined in  $\RR^2$ whose eigenvalues are real, different and positive, say $\lambda_1< \lambda_2\in \RR^+$ and with eigenspaces $E_1$ and $E_2$, respectively. Then:
\begin{lem}\label{trivial1}
If $v\in \RR^2\backslash E_1$,  then $\lim_{n\in \NN}\frac{A^n(v)}{\|A^n(v)\|} \in E_2$.
 \end{lem}
The \emph{Jordan decomposition Theorem} \cite{palis1982local, katok1997introduction} provides an unitary orthogonal basis of $\RR^2$ such that the matrix of $A$ with respect to that basis is diagonal. In this case, the basis consists of two non-zero unit vectors of $E_{1}$ and $E_{ 2}$, respectively. Then for $(v_1, v_2) \in \RR^2\backslash\{(0,0)\}$. we have:
$$
A^n
\left( \begin{array}{c}
v_1  \\
 v_2 \\
\end{array} \right)\,
=\left(
\begin{array}{cc}
\lambda_1^n & 0  \\
 0 & \lambda_2^n \\
\end{array}
\right)\,.
\left(
\begin{array}{c}
v_1  \\
 v_2 \\
\end{array}
\right)\,
= 
\left(
\begin{array}{c}
\lambda_1^n v_1  \\
 \lambda_2^n  v_2 \\
\end{array}
\right)\,.
$$
Since $\lambda_2>\lambda_1$, we get:
$$
\lim_{n\in \NN} \frac{\left(\lambda_1^n v_1, \lambda_2^n v_2\right) }{\sqrt{\lambda_1^{2n} v_1^2+ \lambda_{2}^{2n} v_2^2}}= (0,1) \in E_2,
$$
and Lemma \ref{trivial1} follows. 
 
\subsection{Structural set}
\label{ss: paths}
We now define the concept of \emph{structural set}, a definition emerging  from the Isospectral Theory \cite{bunimovich2012isospectral}.

\begin{defn}\label{def:structural_set}
A non-empty set of  heteroclinic connections $\mathcal{S}$  is said to be a \emph{structural set} for the heteroclinic network $\mathcal{H}$ if every heteroclinic cycle of $\mathcal{H}$ contains an edge of $\mathcal{S}$.
\end{defn}
In general, the structural set associated to a heteroclinic network is not unique, but the results do not depend  on this set of connections \cite{alishah2019asymptotic}. From now on, we ask that this set is \emph{minimal}. 
\bigbreak

\begin{defn}\label{def:branchs}
For $m\in \NN$, we say that the admissible heteroclinic path $\xi = \{\gamma_0,...,\gamma_m\}$ is a $\mathcal{S}$--\emph{branch} for the network $\mathcal{H}$  if:\medbreak
\begin{enumerate}
\item $\gamma_0$ and $\gamma_m$ belong to $\mathcal{S}$;\medbreak
\item $\gamma_j \notin \mathcal{S}$ for all $j\in \{1, ..., m-1\}$. 
\end{enumerate}
\end{defn}
We denote by $B_\mathcal{S}$ the set of all $\mathcal{S}$--branches.
\bigbreak
\begin{defn} 
Let $\mathcal{H}'$ be a   cycle of the heteroclinic network $\mathcal{H}$. We say that $\mathcal{H}'$ is \emph{elementary} if $\mathcal{H}'\cap \mathcal{S}$ contains just one element. Otherwise $\mathcal{H}'$ is \emph{non-elementary}. 
\end{defn}
If a cycle $\mathcal{H}'$ is non-elementary, then it is the concatenation of a finite number of branches of $\mathcal{S}$, say $\xi_0, \xi_1, ..., \xi_m$, $m\in \NN$; in this case we write $$\mathcal{H}'= \xi_0\oplus \xi_1 \oplus...\oplus \xi_m.$$

Our next goal is the formal definition of \textit{skeleton map} associated to a given structural set $\mathcal{S}$.  First, set:
$$
\Pi_\mathcal{S} := \bigcup _{\gamma \in \mathcal{S}} \Pi_\gamma
\quad\textrm{ and }\quad
D_\mathcal{S}^{\ast} := \bigcup _{\xi \in B_\mathcal{S}} \Pi_\xi.
$$
 
If $\xi$ is a $\mathcal{S}$-branch, as observed in expression \eqref{rank1}, the set $\Pi_\xi \subset\Rr_+^6$ is a two-dimensional submanifold of $\RR_+^2$ since four components of  $\Pi_\gamma$ are zero. This is why, from now on, this set will be seen as subsets of $\RR^2$. This fact will be used later at the Subsection  \ref{subsec:the_dual}. 
We are in the right moment to introduce the \textit{skeleton map} associated to $\mathcal{S}$ through   $\pi_\xi$ already defined in Corollary \ref{lemma13}.

\begin{defn}\label{def:skltn_flow_map_S}
Given a structural set $\mathcal{S}$ associated to $\mathcal{H}$, the map $$\pi_\mathcal{S} : D_\mathcal{S}^{\ast}\to \Pi_\mathcal{S}$$  given by 
$$
\pi_\mathcal{S}(y)=\pi_\xi(y),
$$
for $y\in\Pi_\xi $ and $ {\xi \in B_\mathcal{S}}$, is called the \textit{skeleton map} associated to $\mathcal{S}$.
\end{defn}

\begin{figure}[h]
\includegraphics[width=10cm]{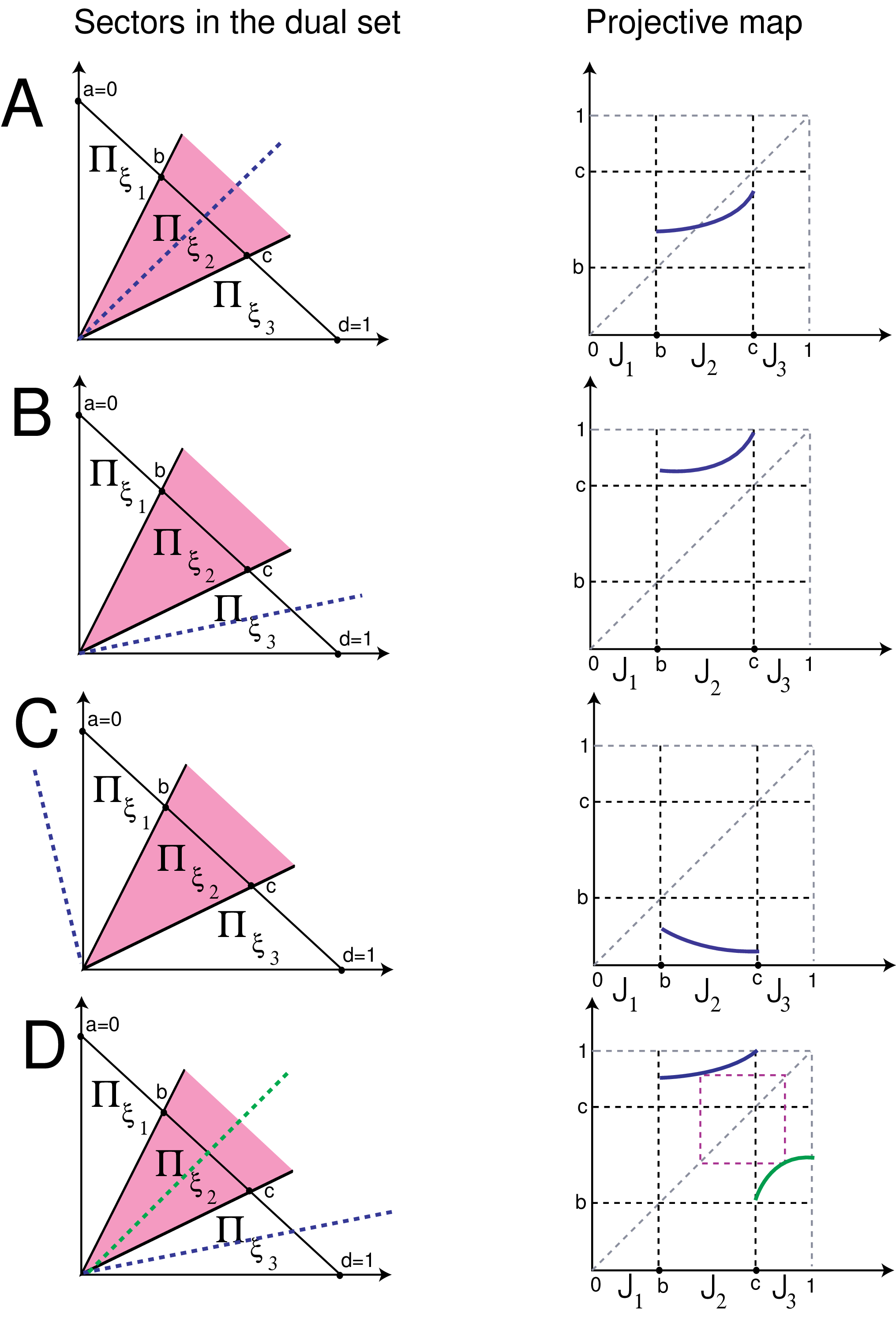}
\caption{\small{The dictionary between the dynamics of $\pi_\xi$ on the dual and the projective map  for an elementary cycle ~$\xi$.}}
 \label{fig:dual_projective1}
\end{figure}

\bigbreak

The following result says that Lebesgue almost all points in $\Pi_\mathcal{S}$ follow \emph{ad infinitum} a prescribed $\mathcal{S}$-branch (or an admissible concatenation of $\mathcal{S}$-branches).

\begin{prop}
\label{prop15}
If $\mathcal{H}$ is asymptotically stable, the set $D_\mathcal{S}^{\ast}$ has full Lebesgue measure in $\Pi_\mathcal{S}$.
\end{prop}

\begin{proof}
  Suppose that $\mathcal{H}$ is asymptotically stable. In particular, there are no more invariant and compact sets in $\Gamma_{(2,2,2)}$ in the neighbourhood of $\mathcal{H}$. Define $D^*_{\gamma_0, \gamma_1} = \bigcup_{\gamma_0, \gamma_1 \in \mathcal{E}} \Pi_{\gamma_0, \gamma_1} $ over any heteroclinic path of the type $\{\gamma_0, \gamma_1\}$.
  The set $D^*_{\gamma_0, \gamma_1} $ has full Lebesgue measure in $\Pi_{\gamma_0}$ because (\cite{alishah2019asymptotic}):
$$
\Pi_{\gamma_0} \backslash D^*_{\gamma_0, \gamma_1} \, \, \subset\, \,  \partial \Pi_{\gamma_0}
 \cup \left( \bigcup_{\gamma_0, \gamma_1 \in \mathcal{E}} L_{\gamma_0, \gamma_1}^{-1} (\partial \Pi_{\gamma_1}) \right) .
$$ 
Note that $L_{\gamma_0, \gamma_1}$ is a linear isomorphism   carrying sets with zero Lebesgue measure into sets with the same property. Consider now any heteroclinic path of the type $\{\gamma_0, \gamma_1, \gamma_2\}$. Using the same line of argument, we get :
\begin{eqnarray*}
\Pi_{\gamma_0} \backslash D^*_{\gamma_0, \gamma_1, \gamma_2} \, \, &\subset\,& \,  \partial \Pi_{\gamma_0}
 \cup \left( \bigcup_{\gamma_0, \gamma_1 \in \mathcal{E}} L_{\gamma_0, \gamma_1}^{-1} (\partial \Pi_{\gamma_1}) \right) \\
&& \cup \left( \bigcup_{\gamma_0, \gamma_1, \gamma_2 \in \mathcal{E}}  L_{\gamma_0, \gamma_1}^{-1}  \circ L_{\gamma_1, \gamma_2}^{-1} (\partial \Pi_{\gamma_2}) \right) ,
\end{eqnarray*}
and then $\Pi_{\gamma_0} \backslash D^*_{\gamma_0, \gamma_1, \gamma_2}$ has zero Lebesgue measure   in $\Pi_{\gamma_0}$. 
Continuing the procedure a countable number of times, we may conclude that $
  D^*_{\gamma_0, \gamma_1, \gamma_2, ...}
$
has also full measure since it is  a countable union of sets with full Lebesgue measure in $\Pi_{\gamma_0}$


\end{proof}
  
For $m\in \NN$, assume that  $\xi = \{\gamma_0,\gamma_1,...,\gamma_m\}$ is an elementary cycle with respect to a given structural set $\mathcal{S}$ and  the map $\pi_\xi$ of Corollary \ref{lemma13} has two different positive real eigenvalues. Given $\Pi_{\xi}$, we have three disjoint possibilities: \\
\begin{itemize}
\item the greatest  eigenvector of $\pi_\xi$ lies on the corresponding sector $\Pi_\xi$ (Case A of Figure \ref{fig:dual_projective1}); \\
\item the greatest  eigenvector of $\pi_\xi$  lies on another  sector of $\Pi_{\gamma_0}$ and then the asymptotic dynamics is computed using the matrix associated to the sector to where the eigenvector moves for (Cases B and D of Figure \ref{fig:dual_projective1});\\
\item the greatest eigenvector lies outside the first quadrant. In this case, this analysis is valid just to the moment where points hit  on the boundary  (Case C of Figure \ref{fig:dual_projective1}). Dynamics accumulates on the boundary. \\
\end{itemize}

If $\xi$ is a non-elementary cycle, then the same conclusions hold by concatenating a finite number of $\mathcal{S}$--branches, provided  the corresponding eigenvalues (for the composition of linear maps) are positive.

\subsection{Projective map}
\label{sec:projective_map}

Based on~\cite[Section 3.6]{peixe2015lotka} we define a \emph{projective map} on the dual cone and study their    periodic orbits,
from where we are able to deduce the asymptotic dynamics of \eqref{eq:poly_rep_4}. The following notation will be useful in the sequel to simplify the writing: \\
\begin{itemize}
\item $\overline{v}= \sum_{i=1}^6 v_j$ for $v=(v_1, ..., v_6)\in (\RR^+_0)^6$; \\
 \item $\Delta_\gamma:=\{\, u\in {\rm int}(\Pi_\gamma)\,\colon \, \overline{u} = 1\,\}$, for $\gamma\in \mathcal{S}$; \\

\item $\Delta_{\xi}:=\{\, u\in {\rm int}(\skDomPoin{\chi}{\xi})\,\colon \, \overline{u} = 1 \,\},$ for $\xi$ a $\mathcal{S}$-branch with $\inter (\Delta_\xi)\neq \emptyset$; \\
 \item  $\Delta_\mathcal{S}:= \cup_{\gamma\in \mathcal{S}} \Delta_\gamma$.\\
\end{itemize}

\begin{defn} 
\label{def: proj}
For a structural set $\mathcal{S}$ associated to the network $\mathcal{H}$ and $\xi=\{\gamma_0, ..., \gamma_m\}$ a $\mathcal{S}$-branch ($m\in \NN$) we define: \\
\begin{enumerate}
\item \emph{the  projective map along   $\xi$} as $\skProjPoin{\chi}{\xi}:\skProjDom{}{\xi} \subset \skProjDom{}{\gamma_0}\to \skProjDom{}{\gamma_m} $ given by:
\begin{equation*}
  \skProjPoin{}{\xi}(u) =   \skPoin{}{\xi}(u) /\overline{ \skPoin{}{\xi}(u) };
  \end{equation*}

\item \emph{the projective $\mathcal{S}$-map} $\skProjPoin{\chi}{\mathcal{S}}: D_\mathcal{S}^*   \to \Delta_\mathcal{S} $ given by\footnote{The domain of $\skProjPoin{\chi}{\mathcal{S}}$ has been defined in Proposition \ref{prop15}.}:
  $$\skProjPoin{\chi}{\mathcal{S}}(u)= \skProjPoin{}{\xi}(u).$$ 
\end{enumerate}
\end{defn}

\begin{defn}
\label{def_x}
A point
$u\in D_\mathcal{S}^*$ such that 
$u=(\skProjPoin{\chi}{\mathcal{S}})^n(u)$ for the some $n\geq 1$, is called a  {\em $n$-periodic point}\index{periodic point} of $\skProjPoin{\chi}{\mathcal{S}}$.
\end{defn}

Throughout this article, we assume that the period  of Definition \ref{def_x}  is minimum.
For $n \in \mathbb{N}$, if $u\in \skProjDom{}{\mathcal{S}}$ is a $n$-periodic point of $\skProjPoin{\chi}{\mathcal{S}}$,
let us denote by $\xi_k$ the unique $\mathcal{S}$-branch such that
$(\skProjPoin{\chi}{\mathcal{S}})^j(u)\in  \skProjDom{}{\xi_k}$
for all $j,k=0, 1,\ldots, n-1$ and $u\in \skProjDom{}{\xi_j}$. Concatenating these branches, we obtain the cycle  $$\Theta= \xi_{k_0}  \oplus \xi_{k_1} \oplus \ldots \oplus \xi_{k_{n-1}}.$$
We refer to this cycle $\Theta$ as the {\em itinerary}\index{itinerary} of
the periodic point $u$.
\bigbreak

\begin{defn}\label{eigenvectors}
\label{perron_eig}
Let $u\in D_\mathcal{S}^*$  be a periodic point of $\skProjPoin{\chi}{\mathcal{S}}$ whose
itinerary is the cycle $\Theta$. We say that: \\
\begin{enumerate}
\item $u$ is an
{\em eigenvector}\index{projective Poincar\'e map!eigenvector} of $\skProjPoin{\chi}{\mathcal{S}}$ if there is $\lambda\in \RR\backslash\{0\}$ such that $\skProjPoin{\chi}{\mathcal{S}}(u)=\lambda\,u$. The number  $\lambda=\lambda(u)>0$ is the {\em Perron eigenvalue} of $u$. 
\item the \emph{saddle-value} of $u$, denoted by $\sigma(u)$, is the maximum ratio  $\frac{{|\lambda'| }}{\lambda  }$  where $\lambda '$ ranges over all
non-zero eigenvalues of $D \skProjPoin{\chi}{\mathcal{S}}$ different from $\lambda$.  \\
\end{enumerate}
 \end{defn}
 
   The next proposition follows straightforwardly: \\

\begin{prop}\label{maximum_ratio_prop}
Let $u$ be a periodic point of $\skProjPoin{\chi}{\mathcal{S}}$ with
itinerary   $\Theta$.  \medbreak
\begin{itemize}
\item[(a)] If $\sigma(u)<1$ then  $u$ is an attracting periodic point of  $\skProjPoin{\chi}{\mathcal{S}}$;
\item[(b)] If $\sigma(u)>1$ then  $u$ is a repelling periodic point of  $\skProjPoin{\chi}{\mathcal{S}}$. \\
\end{itemize}
\end{prop}

\begin{proof}
We prove item (a). Let $u$ be a $n$-periodic point of $\skProjPoin{\chi}{\mathcal{S}}$ with itinerary $\Theta$ and such that  $\sigma(u)<1$. Let $\overline{u}$ be the corresponding vector in $\Pi_\xi \subset \Pi_\mathcal{S}$, where $\xi$ is either a $\mathcal{S}$-branch or a concatenation of $\mathcal{S}$-branches associated to $\mathcal{H}$, depending on whether the itinerary $\Theta$ is elementary or not.
 Since $\sigma(u)<1$, it means that the other eigenvalue of  $\pi_\xi$  is less  than $\lambda$. The result follows by Lemma \ref{trivial1} which says that initial conditions are attracted to the eigendirection associated to the greatest eigenvalue.  
 The proof of (b) is analogous.
\end{proof}

In order to study the  projective map $\skProjPoin{\chi}{\mathcal{S}}: D_\mathcal{S}^* \to \Delta_\mathcal{S}$, 
we identify $\Delta_{\xi_k}$ with $J_k$, where $k$ is over the number the $\mathcal{S}$-branches. 
 With these identifications, we define a map $\varphi : [0,m]\to[0,m]$,  where $m:=\# \mathcal{S}$. This map describes the dynamics of the projective map $\skProjPoin{\chi}{\mathcal{S}}$. As an abuse of language, we also call this map as the \emph{projective map}. 

\begin{table}[htb]

\tiny
\begin{center}
\begin{tabular}{|l|l|l|}
\toprule
Projective map $\varphi$ &   Sector in $\Pi_\mathcal{S}$ \qquad  \qquad &  {Phase space} \qquad \\
\bottomrule
&&\\
Stable fixed point for $\pi_{\xi}$  & Stable eigendirection in $\Pi_\xi$  &  Stable elementary cycle  \\  &&\\
 \hline \hline 
 &&\\
Unstable fixed point for $\pi_{\xi}$  & Unstable eigendirection in $\Pi_\xi$  &  Unstable elementary cycle  \\  &&\\
 \hline \hline 
 &&\\

Stable fixed point for $\pi_{\xi_1\oplus\xi_2}$  & Stable eigendirection in $\Pi_{\xi_1\oplus\xi_2}$  &  Stable cycle  \\  && (concatenation of two branches) \\&&\\
 \hline \hline  
 &&\\
Unstable fixed point for $\pi_{\xi_1\oplus\xi_2}$  &  Unstable eigendirection in $\Pi_{\xi_1\oplus\xi_2}$  &  Unstable cycle  \\  && (concatenation of two branches) \\&&\\
 \hline \hline 
 &&\\  
No fixed point in $\Delta_\xi$  &  Strongest eigendirection in $\Pi_\xi$ &  Initial conditions are repelled \\
  &  lies outside  $\Pi_\xi$ &  \\  &&  \\
\bottomrule
\end{tabular}
\end{center}
\bigskip
\caption{\small{The dictionary between the dynamics of the the projective map, the dual cone and the phase space, for $\xi\in B_\mathcal{S}$ (elementary cycle) and $\xi_1, \xi_2\in B_\mathcal{S}$ (non-elementary cycles).}}\label{notationC}
\end{table}

\begin{rem}
\label{invariant manifold}
The existence of an unstable invariant line within a sector of $\Pi_\mathcal{S}$ has two implications in terms of dynamics: first, the associated cycle is unstable; secondly,  there is an invariant compact manifold of dimension two in the phase space accumulating on the corresponding cycle. 
This will be used in Corollary \ref{Lema19} to show where the manifold $\overline{\mathcal{M}_{\boldsymbol{\mu}}}$ ``glues''.
\end{rem}

\section{Computer aided analysis of the projective map}
\label{sec:analysis_on_dual}
 In this section, we put together the  established theory  to study the stability of the  heteroclinic cycles of $\mathcal{H}$ listed in Lemma \ref{Lem:het_cycles}. All the results rely on  system \eqref{eq:poly_rep_4}.

\subsection{Procedure}
We give a description of our method, locating its theoretical background in the previous section. Our starting point is the heteroclinic network $\mathcal{H}$ given in Lemma \ref{Lem:het_cycles} formed by 6 cycles, and the vector field $f_{\boldsymbol{\mu}}$ \eqref{eq:poly_rep_4} defined in an interval $\mathcal{I}$ where the interior equilibrium $\mathcal{O}_{\boldsymbol{\mu}}$ exists.  \\
\begin{enumerate}
	\item Compute the character map $\chi^v$ of $f_{\boldsymbol{\mu}}$ and draw its flowing-edge graph (Definition \ref{def:character_1});  \\
	\item Find a structural set $\mathcal{S}$ associated to $\mathcal{H}$  and determine all associated $\mathcal{S}$-branches (Definition \ref{def:structural_set});\\
	\item Write explicitly  the skeleton map $\pi_\mathcal{S}$  associated to all possible $\mathcal{S}$-branches $\xi$ with matrix $M_\xi$ (see Definition \ref{def:skltn_flow_map_S} and Expression \eqref{matrix1}). Note that $M_\xi$ just depends on the eigenvalues of $f_{\boldsymbol{\mu}}$ at the equilibria and the architecture of $\mathcal{H}$;	\\

	\item For the periodic points of the skeleton map, define all heteroclinic cycles $\mathcal{H}'$ (given by possible concatenation of  branches) and compute the eigenvalues and eigenvectors of $M_{\mathcal{H}'}$. Every matrix $M_{\mathcal{H}'}$ is a two-dimensional projection of $\RR^6$ and has exactly 4 zero eigenvalues. This is why  this set will be seen as subsets of $\RR^2$;  \\ 
	
	\item Identify the eigenvectors associated to to the greatest eigenvalues and, according to their location in the dual cone,  use Lemma ~\ref{trivial1}   to determine its stability;  \\
	
\item Intersect the eigenvectors (associated to the greatest eigenvalues) with the hyperplane $\overline{u}=1$ (Subsection \ref{sec:projective_map});  \\

\item Define the projective map $\skProjPoin{\chi}{\mathcal{S}}$  for the corresponding $\mathcal{S}$-branches and analyse their periodic points. Fixed points of $\skProjPoin{\chi}{\mathcal{S}}$ correspond to eigendirections of the corresponding matrices $M_\xi$. By computing the Perron eigenvalue associated to each periodic point (Definition \ref{perron_eig}), we determine its stability by applying Proposition \ref{maximum_ratio_prop}.\\
\end{enumerate}
 
Our route in this Section is to  pass from (1), (2) to the projective map defined in (6) to classify the stability of a given subcycle of $\mathcal{H}$. This is the main novelty of our article.

\subsection{Structural set}

We see how the analysis on the dual allows us to draw conclusions about the stability of heteroclinic cycles in $[0,1]^3$.
For $\boldsymbol{\mu} \in \mathcal{I}$, all twelve edges of $[0,1]^3$  correspond to heteroclinic connections and will be called by $\gamma_1,\ldots, \gamma_{12}$,  according to Table~\ref{ex:edges} and Figure \ref{fig:het_cycles3}.  
\begin{figure}[h]
\includegraphics[width=7cm]{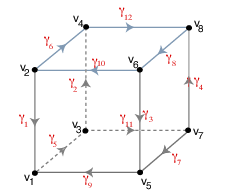}
\caption{\small{Terminology for the twelve different paths on $\mathcal{H}$ displayed in Table \ref{ex:edges}.}}
 \label{fig:het_cycles3}
\end{figure}

\begin{table}[h]
\centering
\begin{tabular}[c]{cccc}
\\
\hline
\\[-3mm]
$\gamma_1=[v_2 \rightarrow v_1]$ & $\gamma_2=[ v_3 \rightarrow  v_4]$ & $\gamma_3=[v_6 \rightarrow v_5]$ & $\gamma_4=[v_7\rightarrow v_8]$ \\
$\gamma_5=[ v_1 \rightarrow v_3] $ & $\gamma_6=[v_2\rightarrow v_4]$ & $\gamma_7=[v_7 \rightarrow v_5]$ & $\gamma_8=[ v_8 \rightarrow v_6]$ \\
$\gamma_9=[ v_5 \rightarrow v_1]$ & $\gamma_{10}=[ v_6 \rightarrow v_2]$ & $\gamma_{11}=[ v_3 \rightarrow v_7]$ & $\gamma_{12}=[ v_4 \rightarrow v_8]$ \vspace{1mm} \\
\hline
\end{tabular}
\vspace{.2cm}
\caption{\footnotesize{Edge labels.}}\label{ex:edges}
\end{table}

Looking at Figure~\ref{fig:het_cycles3}. we can see that
$$
\mathcal{S}=\{ \, \gamma_5=[ v_1 \rightarrow v_3]\,; \,\,\, \gamma_8=[v_8 \rightarrow v_6] \, \}
$$
is a  \emph{structural set} for the heteroclinic network $\mathcal{H}$ in $[0,1]^3$ (Definition~\ref{def:structural_set}),
whose $\mathcal{S}$-branches (Definition~\ref{def:branchs}) are displayed in Table~\ref{branch_table}.
We can see also that:\\
\begin{itemize}
	\item there is only one path, $\xi_1$, that starts and ends at $\gamma_5$; \\
	\item there are two paths, $\xi_2$ and $\xi_3$, starting at $\gamma_5$ and ending at $\gamma_8$;\\
	\item there are two paths, $\xi_4$ and $\xi_5$, starting at $\gamma_8$ and ending at $\gamma_5$;\\
	\item there is only one path, $\xi_6$, that starts and ends at $\gamma_8$.\\
\end{itemize}

Define $\Pi_\mathcal{S}=\Pi_{\gamma_5}\cup \Pi_{\gamma_8}$ where (Figure \ref{fig:het_cycles4}) $$\Pi_{\gamma_5}=\skDomPoin{\chi}{\xi_1}\cup \skDomPoin{\chi}{\xi_2}\cup \skDomPoin{\chi}{\xi_3} \qquad \text{and} \qquad \Pi_{\gamma_8}=\skDomPoin{\chi}{\xi_4}\cup \skDomPoin{\chi}{\xi_5}\cup \skDomPoin{\chi}{\xi_6}.$$

\begin{figure}[h]
\includegraphics[width=13.8cm]{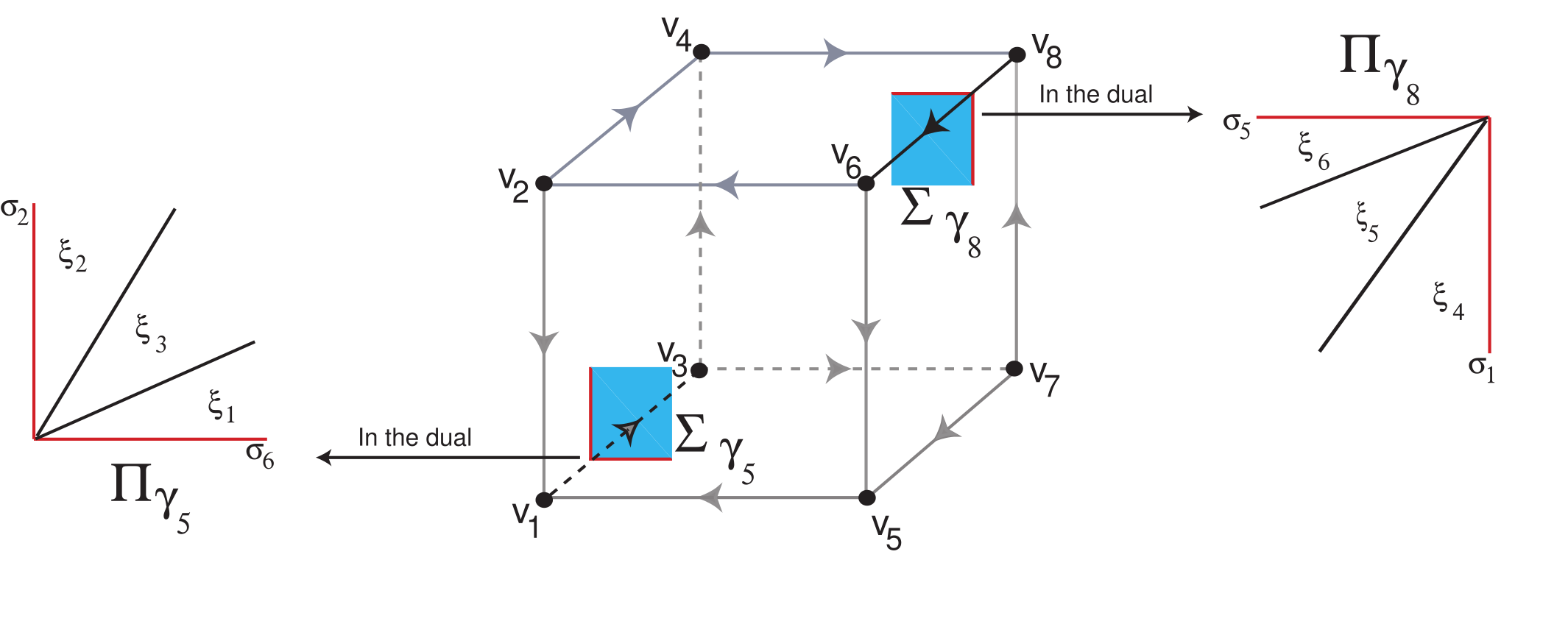}
\caption{\small{ Illustration of the cross sections  $\Sigma_{\gamma_5}$ and $\Sigma_{\gamma_8}$  (in the phase space), and $\Pi_{\gamma_5}$ and $\Pi_{\gamma_8}$ (in the dual set). The letters $\xi$ and $\gamma$ are associated to $\mathcal{S}$--branches and heteroclinic connections, respectively.}}
 \label{fig:het_cycles4}
\end{figure}

\begin{table}[h]
\centering
\begin{tabular}{c|cc}
From\textbackslash To & $\gamma_5$ & $\gamma_8$ \vspace{1mm} \\
\hline 
\\[-3mm]
$\gamma_5$ & $\xi_1$ & \; $\xi_2$,  $\xi_3$ \; \\[2mm]
$\gamma_8$ & \; $\xi_4$,  $\xi_5$ \; & $\xi_6$ \vspace{1mm} \\
\hline
\end{tabular}

\vspace{-5mm}
{\begin{align*}
\\
\hline
\\[-3mm]
\xi_1 &=\{\gamma_5,\gamma_{11},\gamma_7,\gamma_{9},\gamma_5\} 	& \xi_2 &=\{\gamma_5,\gamma_2,\gamma_{12},\gamma_8\}		& \xi_3 &=\{\gamma_5,\gamma_{11},\gamma_4,\gamma_8\} \\[1mm]
\xi_4 &=\{\gamma_8,\gamma_{10},\gamma_{1},\gamma_5\} 					& \xi_5 &=\{\gamma_8,\gamma_{3},\gamma_9,\gamma_5\}		& \xi_6 &=\{\gamma_8,\gamma_{10},\gamma_6,\gamma_{12},\gamma_8\} \vspace{1mm} \\[1mm]
\hline
\end{align*}}

\vspace{-6mm}
\caption{\footnotesize{$\mathcal{S}$-branches associated to $\mathcal{H}$.}} \label{branch_table}
\end{table}

\subsection{Dynamics on the projective map}\label{subsec:the_dual}
We  consider now the skeleton map  $\skPoin{\chi}{\mathcal{S}}:D_\mathcal{S}^{\ast} \to\Pi_{\mathcal{S}}$  whose domain is depicted in Figure~\ref{2D_Cones_paths}.

Because the remaining coordinates vanish, we consider the coordinates
$(u_2,u_6)$ on $\Pi_{\gamma_5}$ and $(u_1,u_5)$ on $\Pi_{\gamma_8}$.
Table~\ref{branch:table_paths} provides the matrix representation and the corresponding defining
conditions for all the branches of the skeleton map  $\skPoin{\chi}{\mathcal{S}},$ with respect to
the previous coordinates.
As already referred at the end of Subsection \ref{ss: paths}, in all  domains $\Pi_{\xi_j}$, the inequalities $u_1\geq 0$,  $u_2\geq 0$,  $u_5\geq 0$
and  $u_6\geq 0$ are implicit.

\begin{figure}[h]
\includegraphics[width=13cm]{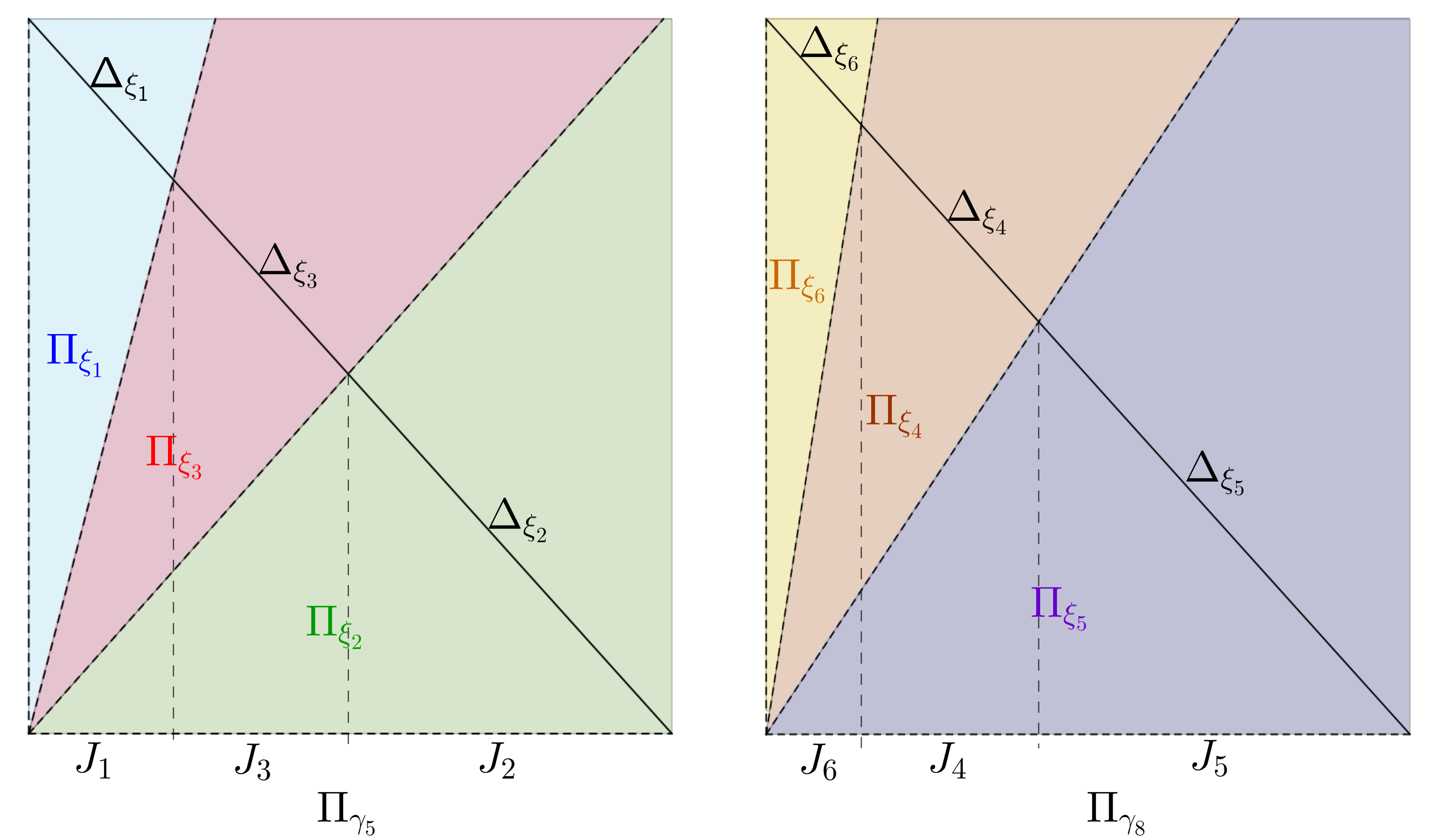}
\vspace{-.7cm}
\caption{\footnotesize{The domains $\skDomPoin{\chi}{\xi_1}$, $\skDomPoin{\chi}{\xi_2}$, $\skDomPoin{\chi}{\xi_3} \subset \Pi_{\gamma_5}$ (left), and $\skDomPoin{\chi}{\xi_4}$, $\skDomPoin{\chi}{\xi_5}$, $\skDomPoin{\chi}{\xi_6} \subset \Pi_{\gamma_8}$ (right), of the skeleton flow  map  $\skPoin{\chi}{\mathcal{S}}:\skDomPoin{\chi}{\mathcal{S}} \to\Pi_{\mathcal{S}}$, with $\boldsymbol{\mu}=\frac{850}{11}$.
Moreover, the domains $\Delta_{\xi_1}$, $\Delta_{\xi_2}$, $\Delta_{\xi_3}\subset \Delta_{\gamma_5}$ (left), and $\Delta_{\xi_4}$, $\Delta_{\xi_5}$, $\Delta_{\xi_6}\subseteq \Delta_{\gamma_8}$ (right), of the projective $\mathcal{S}$-map $\skProjPoin{\chi}{\mathcal{S}}:\Delta_\mathcal{S}\to \Delta_\mathcal{S} $.
}}
 \label{2D_Cones_paths}
\end{figure}

To represent the  projective map $\skProjPoin{\chi}{\mathcal{S}}: D_\mathcal{S}^{\ast} \to \Delta_\mathcal{S}$ (Definition \ref{def: proj}), 
we identify $\Delta_{\xi_k}$ with $J_k$, where $k=1,2,3$, and $1+\Delta_{\xi_\ell}$ with $J_\ell$ and  $\ell=4,5,6$. 
Hence, we are identifying
$\Delta_{\gamma_5}$ with $[0,1]$, 
$\Delta_{\gamma_8}$ with $[1,2]$ and $\Delta_\mathcal{S}$ with $[0,2]$.
With these identifications, we define the map $\varphi_{\boldsymbol{\mu}}:[0,2]\to[0,2]$ given in~\eqref{proj:branches}. See its graph in Figure~\ref{Iterated_Projective} for two different values of $\boldsymbol{\mu}$.

\begin{small}
\begin{equation} 
\varphi_{\boldsymbol{\mu}} (x) = \left\{\begin{array}{ll}

\frac{588 (\boldsymbol{\mu} -176) x}{79 (65 \boldsymbol{\mu} -6834) x-518 (\boldsymbol{\mu} -18)} \,,
& x \in \left[ 0 , \frac{74}{329} \right[ = J_1  \\ \\[-2mm]

\frac{2 (74 (\mu -78)+(131 \mu -55298) x)}{74 (\mu +34)+(131 \mu -92146) x} \,,
& x \in \left] \frac{74}{329} , \frac{74}{149} \right] = J_3  \\ \\[-2mm]

\frac{324 \mu +(226 \mu -160251) x-24674}{162 \mu +(113 \mu -141380) x+2380} \,,
& x \in \left] \frac{74}{149} , 1 \right[ = J_2  \\ \\[-2mm]

\frac{-3476 \mu +(2761 \mu -412581) x+419016}{-2453 \mu +158 (11 \mu -1530) x+248175} \,,
&  x \in \left[ 1 , \frac{4137+22\boldsymbol{\mu}}{4236+11\boldsymbol{\mu}} \right[ = J_6 \\ \\[-2mm]

-\frac{4 (-26 \mu +(13 \mu -2772) x+2889)}{-94 \mu +(47 \mu +49212) x-48789} \,,
& x \in \left] \frac{4137+22\boldsymbol{\mu}}{4236+11\boldsymbol{\mu}} , \frac{75+2\boldsymbol{\mu}}{84+\boldsymbol{\mu}} \right] = J_4  \\ \\[-2mm]

-\frac{42 (-6 \mu +(3 \mu -530) x+626)}{(\mu +96270) x-2 (\mu +52374)} \,,
& x \in \left] \frac{75+2\boldsymbol{\mu}}{84+\boldsymbol{\mu}} , 2 \right[ = J_5

\end{array}
\right. .
\label{proj:branches}
\end{equation}
\end{small}

The points $0$, $1$ and $2$ correspond to the invariant boundary lines of the domains $\skDomPoin{\chi}{\gamma_5}$ and $\skDomPoin{\chi}{\gamma_8}$. They are fixed points for the projective map $\varphi_{\boldsymbol{\mu}}$, associated to initial conditions lying on the cube's boundary.

\begin{figure}[h]
    \centering
	\includegraphics[width=13cm]{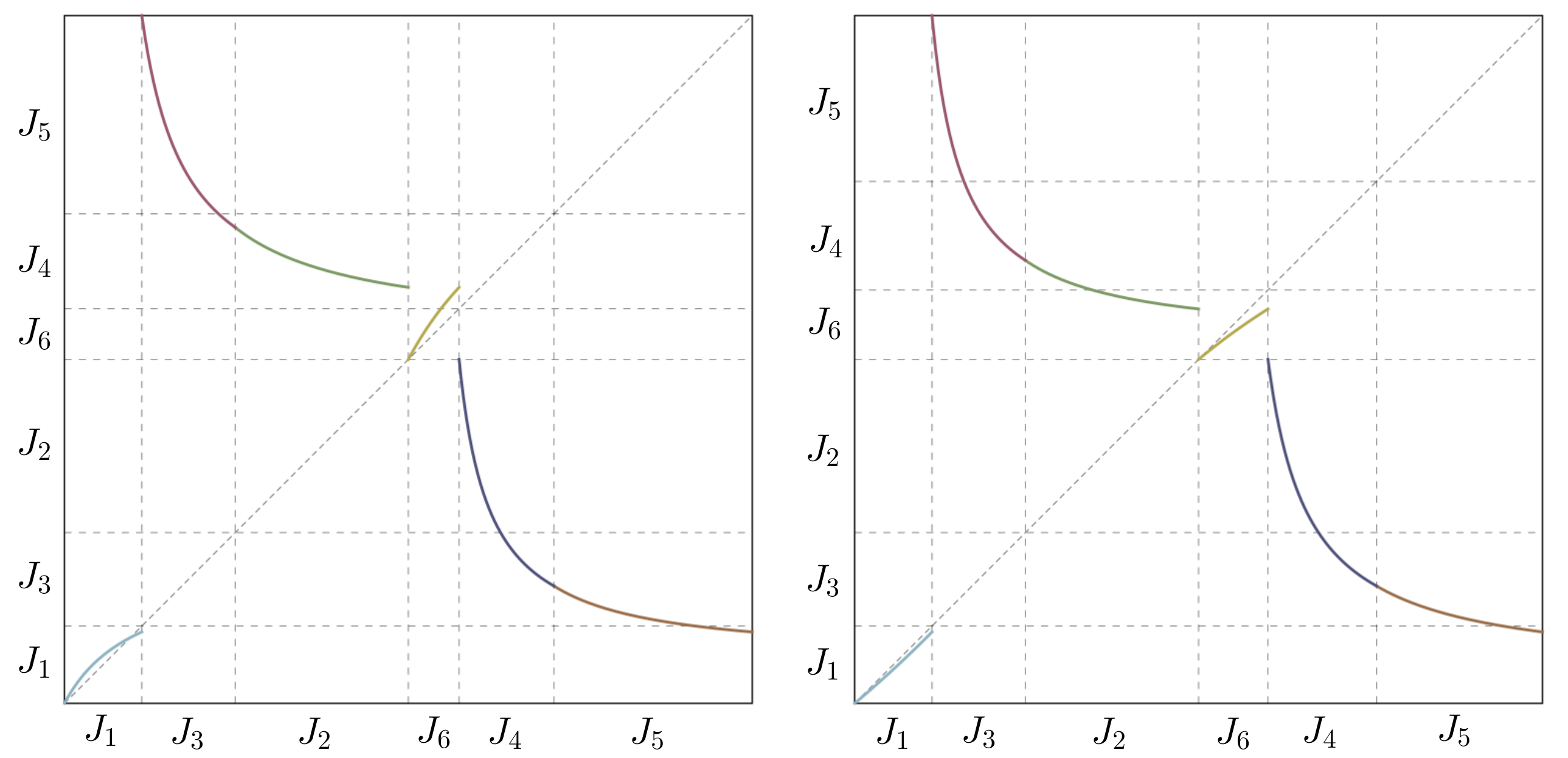}
\caption{\footnotesize{The projective map  $\varphi_{\boldsymbol{\mu}}:[0,2]\to[0,2]$ with
$\boldsymbol{\mu}=\frac{850}{11}$ (left) and $\boldsymbol{\mu}=\frac{544}{5}$ (right),
and the corresponding domains $J_k$, for $k=1, \dots, 6$.}}
\label{Iterated_Projective}
\end{figure}

\begin{prop}\label{prop:projective_paths}
For system~\eqref{eq:poly_rep_4}\footnote{Observe that systems  \eqref{eq:poly_rep_4} and \eqref{eq:poly_rep_5} are equivalent.}, there exist $\boldsymbol{\mu}_3$, $\boldsymbol{\mu}_4$, $\boldsymbol{\mu}_5\in \mathcal{I}_1$ such that: \\
\begin{enumerate}
	\item[(a)] for $\boldsymbol{\mu}\in\left] \frac{850}{11}, \boldsymbol{\mu}_3 \right[$, the projective map $\varphi_{\boldsymbol{\mu}}$ has a unique globally attracting  fixed point in $\inter\left( J_1 \right)$;\\
	\item[(b)] for $\boldsymbol{\mu}\in\left] \boldsymbol{\mu}_3, \boldsymbol{\mu}_4 \right[$, the projective map $\varphi_{\boldsymbol{\mu}}$ has:\\
	\begin{itemize}
		\item two attracting  fixed points, one in $\inter\left( J_1 \right)$ and another in $\inter\left( J_6 \right)$; 
		\item a repelling periodic point of period two in $\inter\left( J_2 \right)$ such that its image by $\varphi_{\boldsymbol{\mu}}$ is in $\inter\left( J_4 \right)$ (cf. case $\boldsymbol{\mu}=96$ in Table \ref{tbl:projective_mu}). \\
	\end{itemize}
	\item[(c)] for $\boldsymbol{\mu}\in\left] \boldsymbol{\mu}_4, \boldsymbol{\mu}_5 \right[$, the projective map $\varphi_{\boldsymbol{\mu}}$ has: \\
	\begin{itemize}
		\item two attracting  fixed points, one in $\inter\left( J_1 \right)$ and another in $\inter\left( J_6 \right)$; 
		\item a repelling periodic point of period two in $\inter\left( J_3 \right)$ such that its image by $\varphi_{\boldsymbol{\mu}}$ is in $\inter\left( J_4 \right)$ (cf. case $\boldsymbol{\mu}=99$ in Table \ref{tbl:projective_mu}). \\
	\end{itemize}
	\item[(d)] for $\boldsymbol{\mu}\in\left] \boldsymbol{\mu}_5, 102 \right[$, the projective map $\varphi_{\boldsymbol{\mu}}$ have: \\
	\begin{itemize}
		\item two attracting  fixed points, one in $\inter\left( J_1 \right)$ and another in $\inter\left( J_6 \right)$; 
		\item a repelling periodic point of period two in $\inter\left( J_3 \right)$ such that its image by $\varphi_{\boldsymbol{\mu}}$ is in $\inter\left( J_5 \right)$ (cf. case $\boldsymbol{\mu}=101$ in Table \ref{tbl:projective_mu}). \\
	\end{itemize}
	\item[(e)] for $\boldsymbol{\mu}\in \mathcal{I}_2 \cup \mathcal{I}_3$, the projective map $\varphi_{\boldsymbol{\mu}}$  has a repelling periodic point of period two in $\inter\left( J_3 \right)$ such that its image by $\varphi_{\boldsymbol{\mu}}$ is in $\inter\left( J_5 \right)$ (cf. case $\boldsymbol{\mu}=103$ in Table \ref{tbl:projective_mu}).
\end{enumerate}
\end{prop}

\begin{proof}
The eigenvector of $M_{\xi_1}$ that depends on $\boldsymbol{\mu}$,
$$
v_1\equiv \left\{\frac{14 (\boldsymbol{\mu} -102)}{51 (\boldsymbol{\mu} -106)},1\right\}
$$
lies in the interior of $\Pi_{\xi_1}$ if and only if
$\boldsymbol{\mu}\in\mathcal{I}_1$.
For these values of the parameter, the eigenvector is the one associated to the greatest eigenvalue of $M_{\xi_1}$.
Hence,  by Proposition \ref{maximum_ratio_prop} the point $\frac{v_1}{\overline{v_1}}\in\inter\left( \Delta_{\xi_1} \right)$ corresponds to the attracting  fixed point
$$
x_1:=\frac{14 \boldsymbol{\mu} -1428}{65 \boldsymbol{\mu} -6834}\in \inter\left( J_1 \right),
$$
of $\varphi_{\boldsymbol{\mu}}$ for $\boldsymbol{\mu}\in\mathcal{I}_1$.
Moreover, for $\boldsymbol{\mu}\in\left] \frac{850}{11}, \boldsymbol{\mu}_3 \right[$, where $\boldsymbol{\mu}_3=94$, this is the unique periodic point of $\varphi_{\boldsymbol{\mu}}$ and hence 
$\frac{v_1}{\overline{v_1}}\in\Delta_{\xi_1}$ corresponds to the unique globally attracting  fixed point $x_1$ of $\varphi_{\boldsymbol{\mu}}$. This concludes the proof of $(a)$.

We can analogously see that the eigenvector of $M_{\xi_6}$ that depends on $\boldsymbol{\mu}$,
$$
v_6\equiv \left\{\frac{11 (102-\boldsymbol{\mu})}{408},1\right\}
$$
lies in the interior of $\Pi_{\xi_6}$ if and only if
$\boldsymbol{\mu}\in\left] \boldsymbol{\mu}_3, 102 \right[$.
For these values of the parameter,  this eigenvector is the one   associated to the greatest eigenvalue of $M_{\xi_6}$.
By Proposition \ref{maximum_ratio_prop}, the point $\frac{v_6}{\overline{v_6}}\in\inter\left( \Delta_{\xi_6} \right)$ corresponds to the attracting  fixed point
$$
x_2:=\frac{22 \boldsymbol{\mu} -2652}{11 \boldsymbol{\mu} -1530}\in \inter\left( J_6 \right),
$$
of $\varphi_{\boldsymbol{\mu}}$ for $\boldsymbol{\mu}\in\left] \boldsymbol{\mu}_3, 102 \right[$. Furthermore, we can see that: \\
\begin{enumerate}
	\item for $\boldsymbol{\mu}\in\left] \boldsymbol{\mu}_3, \boldsymbol{\mu}_4 \right[$, where $\boldsymbol{\mu}_4=\frac{85251}{869}$, the point
	\begin{scriptsize}$$
	x_3:=3 \sqrt{\frac{320140324 \boldsymbol{\mu}^2-60787796412 \boldsymbol{\mu} +2893236225489}{(71941 \boldsymbol{\mu} -6256224)^2}}-\frac{5 (7486 \boldsymbol{\mu} -751455)}{71941 \boldsymbol{\mu} -6256224} \in \inter\left( J_2 \right),
	$$\end{scriptsize}
	is a repelling periodic point of period two,	such that $\varphi_{\boldsymbol{\mu}}\left( x_3 \right)\in J_4$;\\ 
	
	\item for $\boldsymbol{\mu}\in\left] \boldsymbol{\mu}_4, \boldsymbol{\mu}_5 \right[$, where $\boldsymbol{\mu}_5=\frac{85234}{849}$, the point
	\begin{tiny}$$
	x_4:=3 \sqrt{\frac{2615265201481 \boldsymbol{\mu}^2-504216904560828 \boldsymbol{\mu} +24312026567983716}{(9965897 \boldsymbol{\mu} -946939158)^2}}-\frac{5 (533489 \boldsymbol{\mu} -52940718)}{9965897 \boldsymbol{\mu}
   -946939158} \in \inter\left( J_3 \right),
	$$\end{tiny}
	is a repelling periodic point of period two,	such that $\varphi_{\boldsymbol{\mu}}\left( x_4 \right)\in J_4$;\\
	
	\item for $\boldsymbol{\mu}\in\left] \boldsymbol{\mu}_5, \frac{544}{5} \right[$, the point
	\begin{scriptsize}$$
	x_5:=\frac{1}{8} \sqrt{\frac{3386009761 \mu ^2-644714033868 \mu +30745583285796}{(9157 \mu -787158)^2}}-\frac{5 (8467 \mu -845394)}{8 (9157 \mu -787158)} \in \inter\left( J_3 \right),
	$$\end{scriptsize}
	is a repelling periodic point of period two,	such that $\varphi_{\boldsymbol{\mu}}\left( x_5 \right)\in J_5$;\\
\end{enumerate}
which concludes the proof of $(b)$, $(c)$, $(d)$, and $(e)$.

\end{proof}

\begin{table}[htb]
\tiny
\begin{center}
\begin{tabular}{|c|c|c|}
\toprule \\[-2mm]
\quad $\boldsymbol{\mu}=96$ \qquad & \qquad \includegraphics[width=4cm]{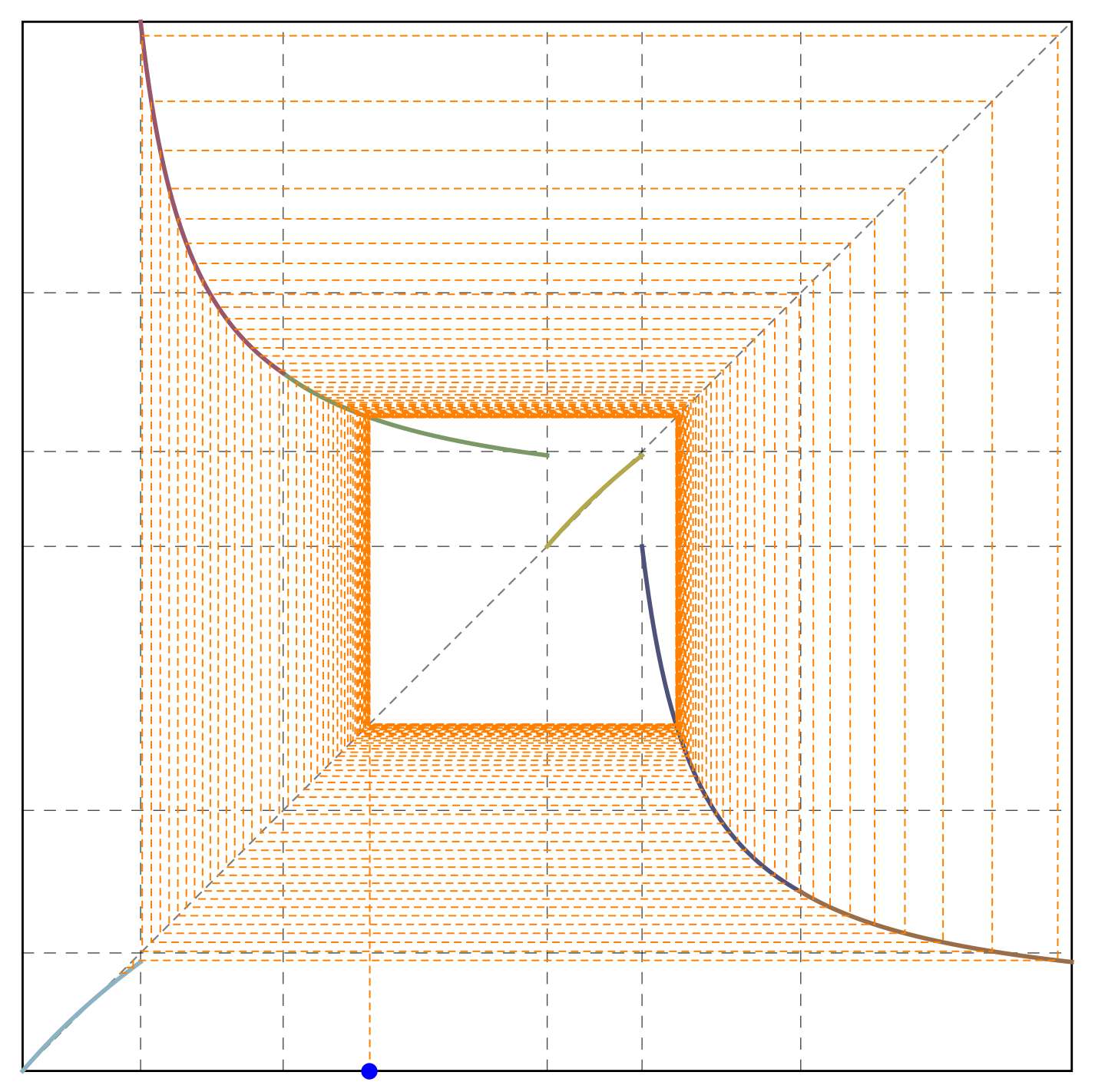} \qquad & \qquad \includegraphics[width=4cm]{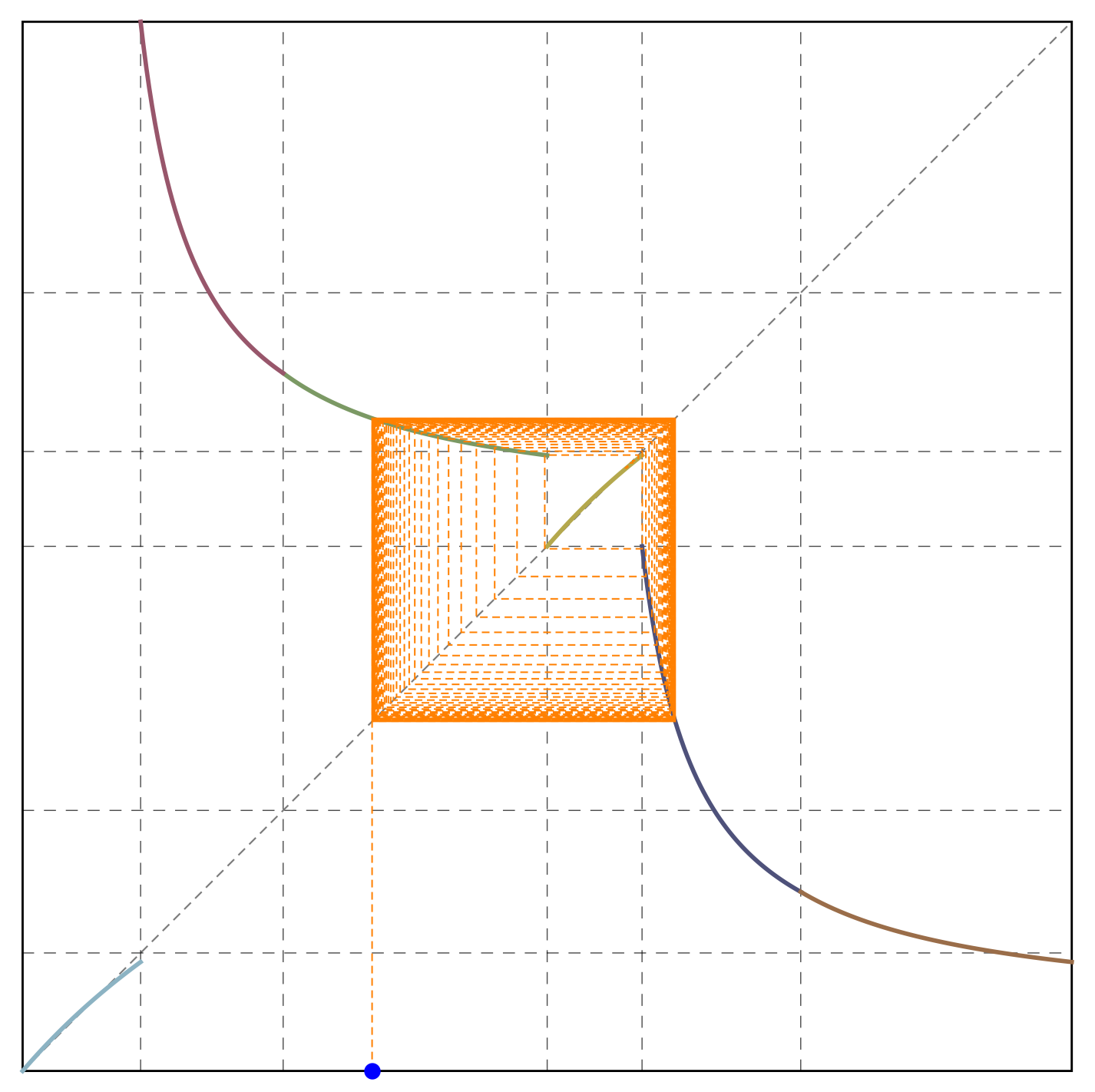} \qquad \\[2mm] \midrule \\[-2mm]

\quad $\boldsymbol{\mu}=99$ \qquad & \qquad \includegraphics[width=4cm]{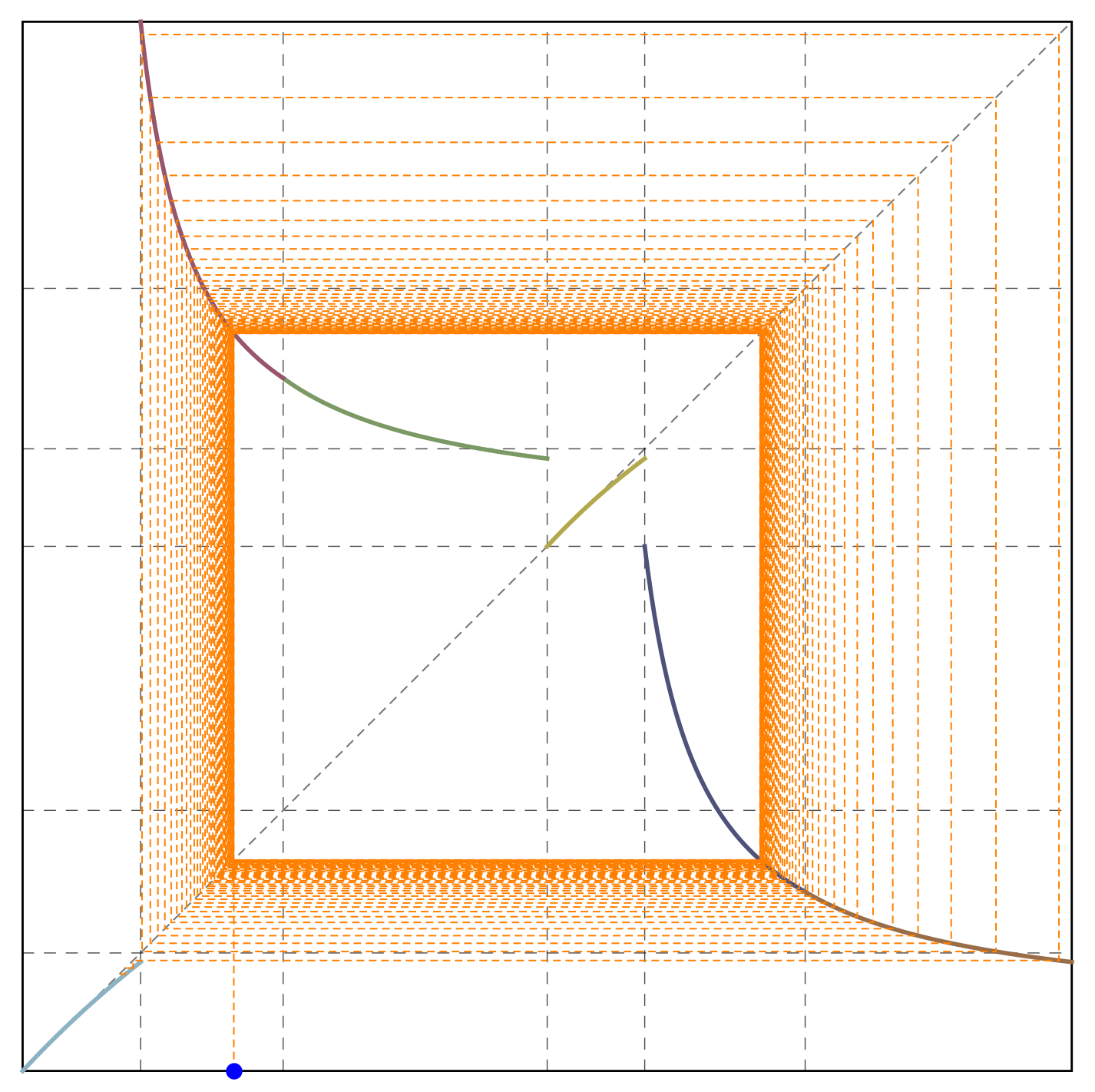} \qquad & \qquad \includegraphics[width=4cm]{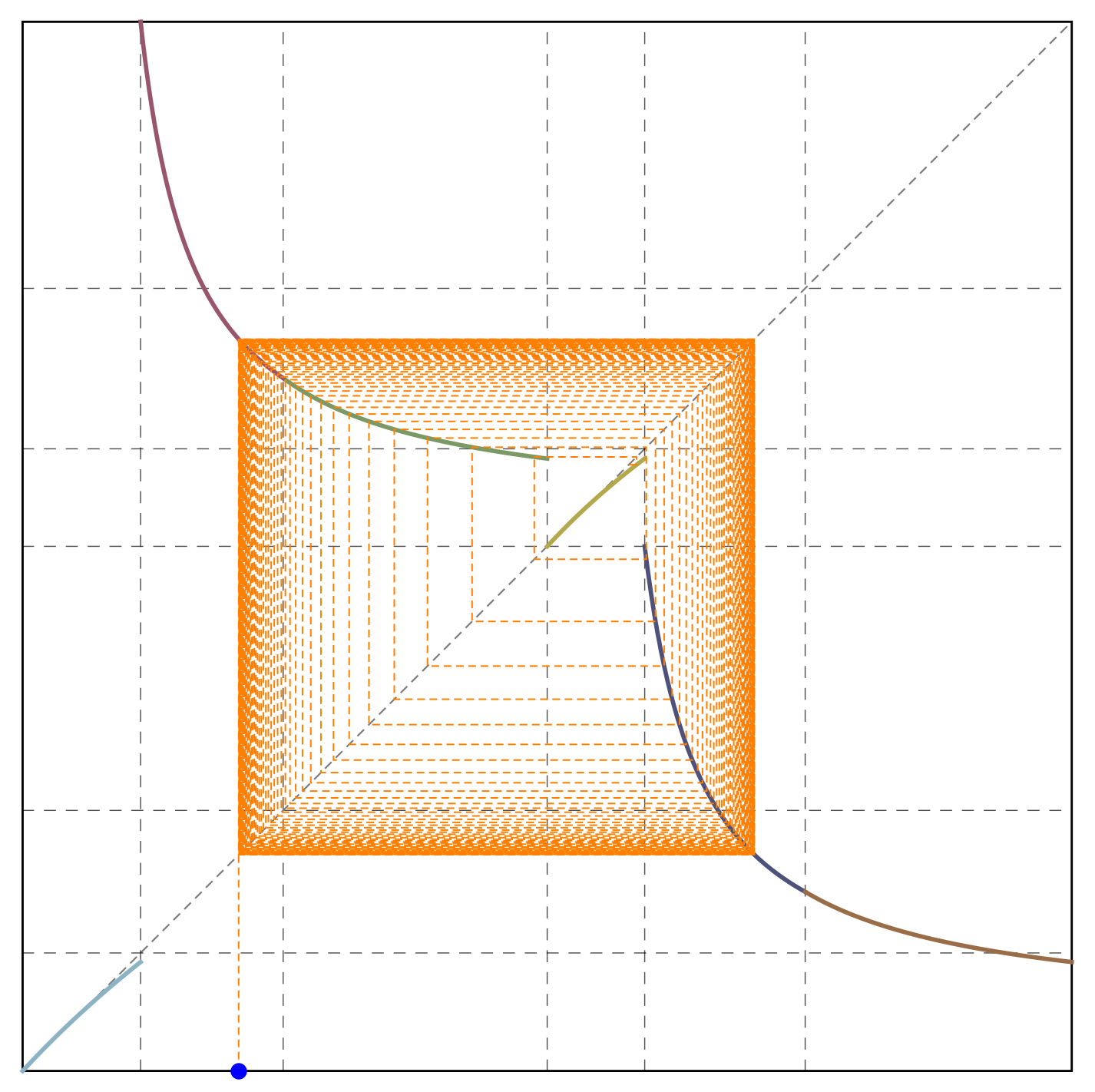} \qquad \\[2mm] \midrule \\[-2mm]

\quad $\boldsymbol{\mu}=101$ \qquad & \qquad \includegraphics[width=4cm]{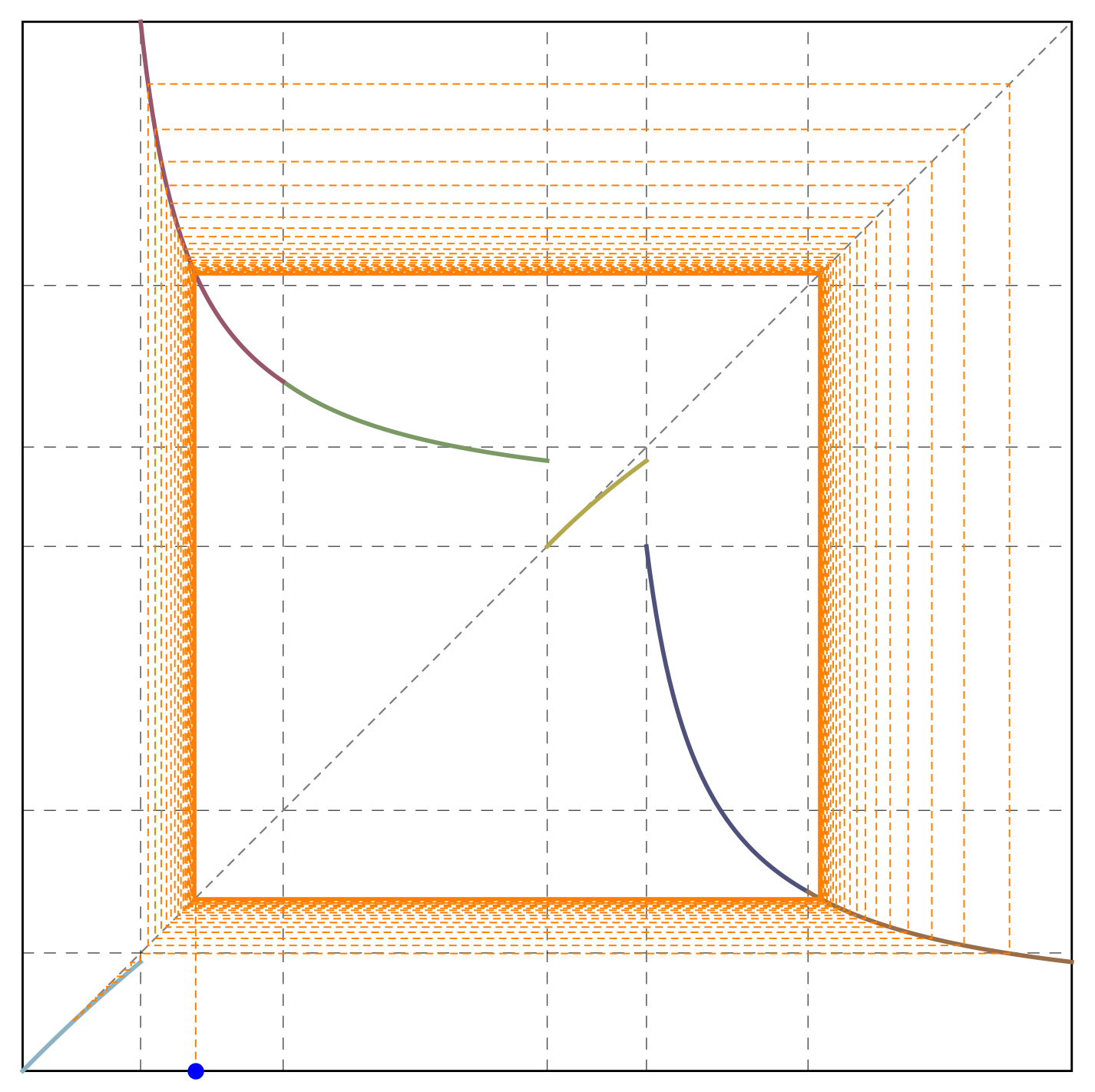} \qquad & \qquad \includegraphics[width=4cm]{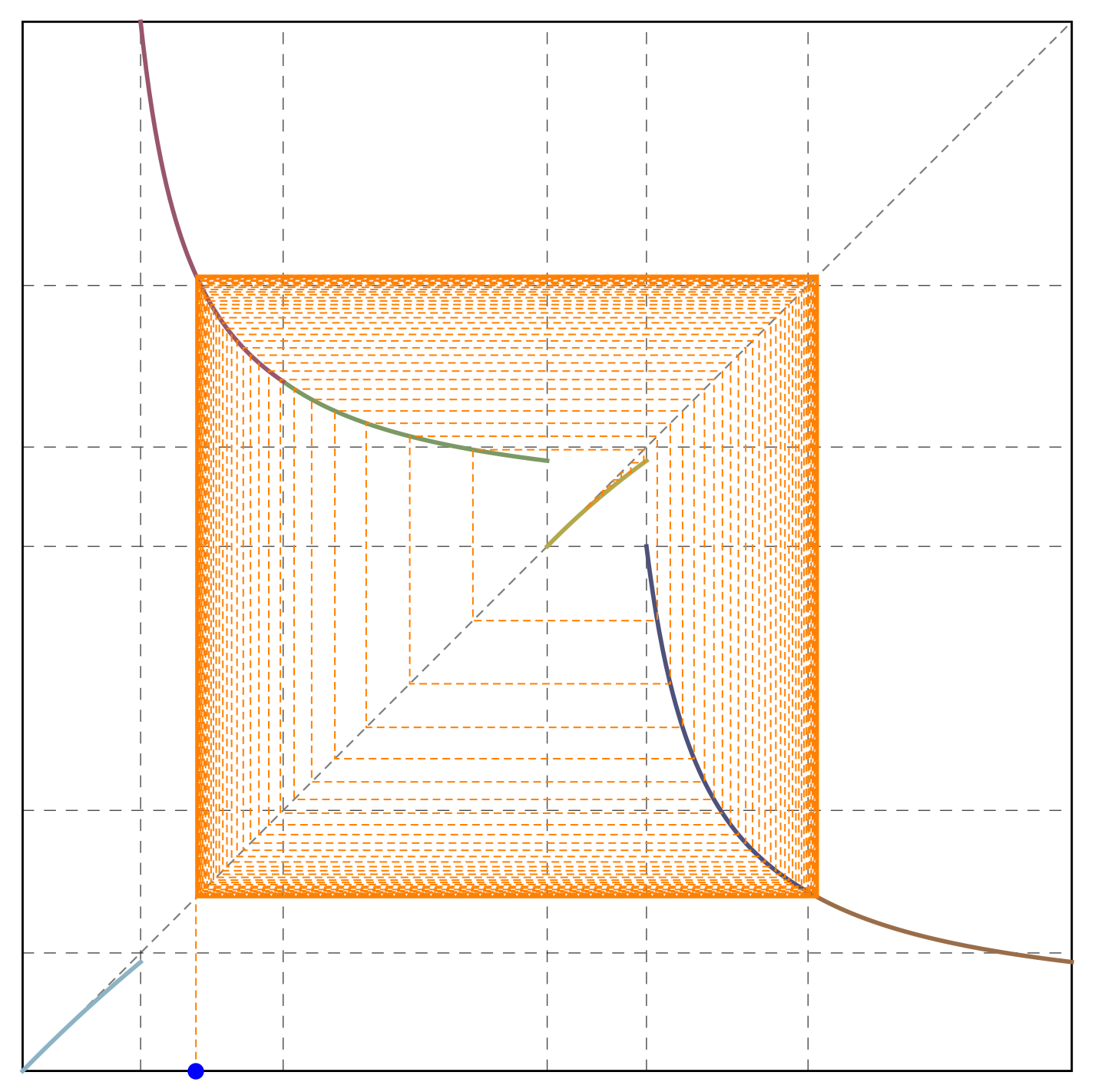} \qquad \\[2mm] \midrule \\[-2mm]

\quad $\boldsymbol{\mu}=103$ \qquad & \qquad \includegraphics[width=4cm]{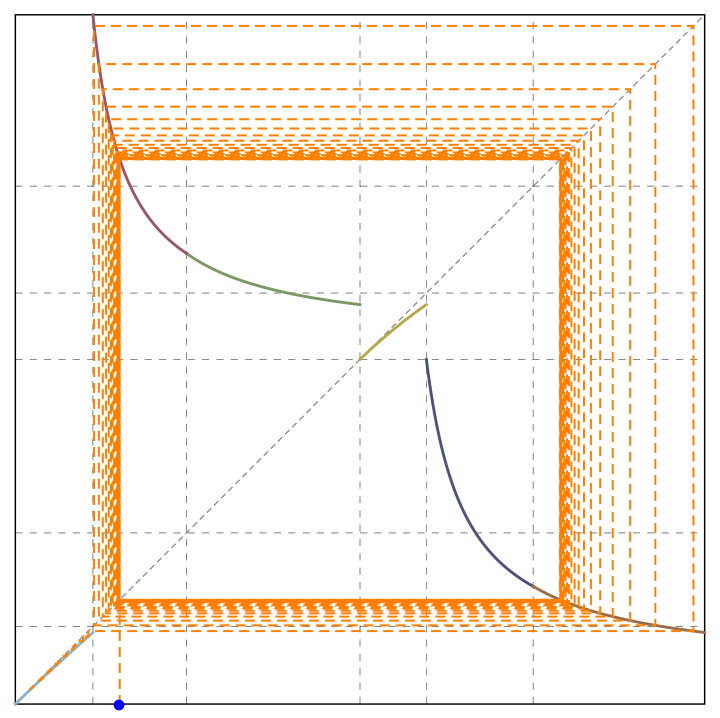} \qquad & \qquad \includegraphics[width=4cm]{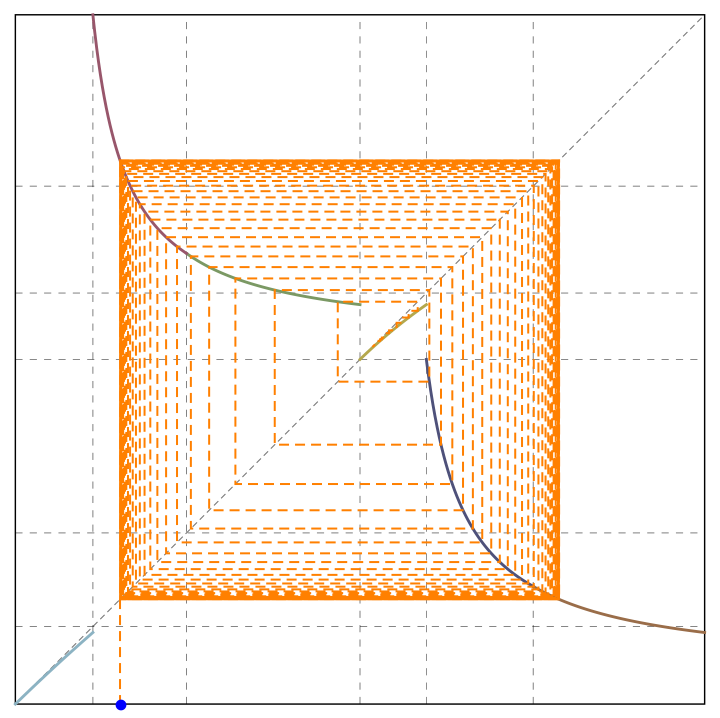} \qquad \\[2mm]
\bottomrule

\end{tabular}
\end{center}
\bigskip
\caption{\small The plot of the projective map  $\varphi_{\boldsymbol{\mu}}:[0,2]\to[0,2]$, where the blue dot in $\inter\left( J_k \right)$ is the initial condition, whith $k=2,3$ for $\boldsymbol{\mu}=96$, and $\boldsymbol{\mu}=99, 101, 103$, respectively, and the orange lines follows its 100 iterates by $\varphi_{\boldsymbol{\mu}}$.} \label{tbl:projective_mu}
\end{table}

\subsection{Stability of the heteroclinic cycles}
\mbox\\
First of all, observe that 
\begin{eqnarray*}
\mathcal{H}_1&=&\xi_6\\
\mathcal{H}_2&=&\xi_2 \oplus \xi_4\\
\mathcal{H}_3&=&\xi_2 \oplus \xi_5 \\
\mathcal{H}_4&=&\xi_3 \oplus \xi_4\\
\mathcal{H}_5&=&\xi_3 \oplus \xi_5 \\
\mathcal{H}_6&=&\xi_1
\end{eqnarray*}
where the symbol $\oplus$ means the \emph{concatenation} between the admissible paths. 
The entries of the matrix associated to an elementary cycle are all positive while those associated to a non-elementary cycle may be negative. Nevertheless those that correspond to the matrix of the concatenated path are positive. In particular, the  Perron-Frobenius theory may be applied\footnote{Note that the theory revisited in Subsection \ref{PF_review} (in particular, Lemma \ref{trivial1})  is valid for different positive real eigenvalues.}.

\begin{figure}[h]
    \centering
	\includegraphics[width=13cm]{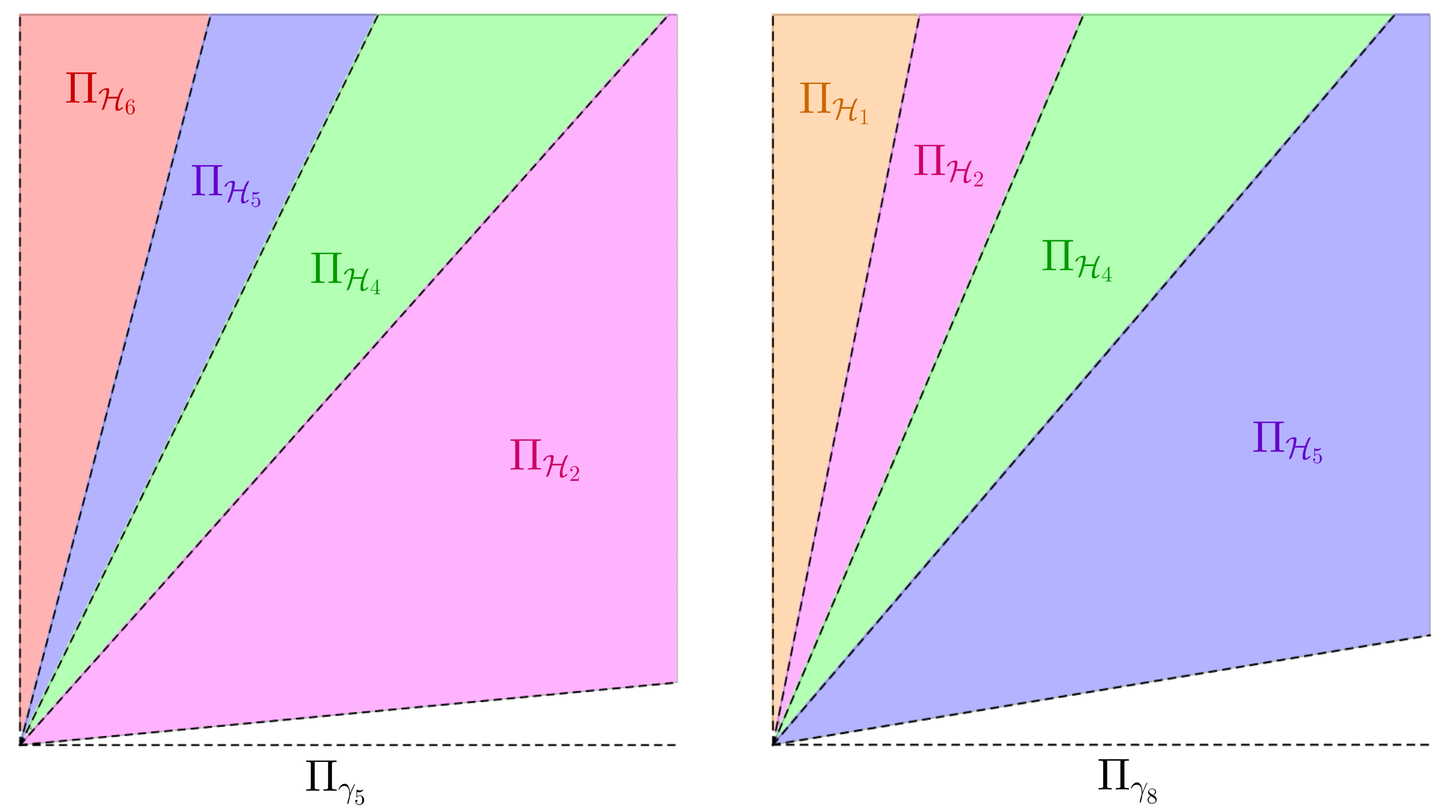}
\caption{\footnotesize{The domains $\skDomPoin{\chi}{\mathcal{H}_2}$, $\skDomPoin{\chi}{\mathcal{H}_4}$, $\skDomPoin{\chi}{\mathcal{H}_5}$, $\skDomPoin{\chi}{\mathcal{H}_6} \subset \Pi_{\gamma_5}$ (left), and $\skDomPoin{\chi}{\mathcal{H}_1}$, $\skDomPoin{\chi}{\mathcal{H}_2}$, $\skDomPoin{\chi}{\mathcal{H}_4}$, $\skDomPoin{\chi}{\mathcal{H}_5} \subset \Pi_{\gamma_8}$ (right), of the skeleton map  $\skPoin{\chi}{\mathcal{S}}:\skDomPoin{\chi}{\mathcal{S}} \to\Pi_{\mathcal{S}}$, with $\boldsymbol{\mu}=97$.}}
 \label{2D_Cones_cycles}
\end{figure}


\begin{cor}\label{prop:projective_cycles}
For system~\eqref{eq:poly_rep_5} there exist $\boldsymbol{\mu}_3$, $\boldsymbol{\mu}_4$, $\boldsymbol{\mu}_5\in \mathcal{I}_1$ such that: \\
\begin{enumerate}
	\item[(a)] for $\boldsymbol{\mu}\in\left] \frac{850}{11}, \boldsymbol{\mu}_3 \right[$
	(cf. case $\boldsymbol{\mu}=90$ in Table \ref{tbl:vectors_and_orbits_mu}): \\
	\begin{itemize}
		\item $\overline{\mathcal{M}_{\boldsymbol{\mu}}}  \cap \partial [0,1]^3 = \mathcal{H}_6$; 
		\item  the cycle $\mathcal{H}_6$ is globally asymptotically stable in the interior of the cube. \\
	\end{itemize}
	\item[(b)] for $\boldsymbol{\mu}\in\left] \boldsymbol{\mu}_3, \boldsymbol{\mu}_4 \right[$, $\overline{\mathcal{M}_{\boldsymbol{\mu}}} \cap \partial [0,1]^3 = \mathcal{H}_2$  (cf. case $\boldsymbol{\mu}=96$ in Table \ref{tbl:vectors_and_orbits_mu}); \\
		\item[(c)] for $\boldsymbol{\mu}\in\left] \boldsymbol{\mu}_4, \boldsymbol{\mu}_5 \right[$, $\overline{\mathcal{M}_{\boldsymbol{\mu}}} \cap \partial [0,1]^3 = \mathcal{H}_4$  (cf. case $\boldsymbol{\mu}=99$ in Table \ref{tbl:vectors_and_orbits_mu});  \\
		\item[(d)] for $\boldsymbol{\mu}\in\left] \boldsymbol{\mu}_5, 102 \right[$, $\overline{\mathcal{M}_{\boldsymbol{\mu}}} \cap \partial [0,1]^3 = \mathcal{H}_5$  (cf. case $\boldsymbol{\mu}=101$ in Table \ref{tbl:vectors_and_orbits_mu}).\\
\end{enumerate}
Moreover, in  Cases $(b), (c)$, and $(d)$, $\overline{\mathcal{M}_{\boldsymbol{\mu}}}$ divides the interior of the phase space in two regions, $\mathcal{U}_1$ and $\mathcal{U}_2$, such that for any initial condition in $\mathcal{U}_1$, its $\omega$-limit is the cycle $\mathcal{H}_1$, and for any initial condition in $\mathcal{U}_2$, its $\omega$-limit is the cycle $\mathcal{H}_6$.
\end{cor}

\begin{proof}
The proof follows from the analysis of the projective map performed in Proposition \ref{prop:projective_paths} and the theory developed in Section \ref{sec:asym_dyn}.

To conclude about the cycles stability we look at the eigenvectors of the matrices for each cycle. For example, in case $\boldsymbol{\mu}\in \left] \boldsymbol{\mu}_5, 102 \right[$ we can see that the eigenvector of $M_{\mathcal{H}_1}$ that belongs to the interior of the sector $\Pi_{\mathcal{H}_1}$ is the  greatest eigenvalue, and the same happens for  $M_{\mathcal{H}_6}$. The eigenvector of $M_{\mathcal{H}_5}$ that belongs to the interior of the sector $\Pi_{\mathcal{H}_5}$ is the  smallest eigenvalue (see case $\boldsymbol{\mu}=101$ in Table \ref{tbl:vectors_and_orbits_mu}). For the other cases, the analysis is analogous.

The existence of an unstable periodic point for the projective map $\varphi_{\boldsymbol{\mu}}$ implies that there exists an invariant line for the corresponding dual cone $\Pi_\mathcal{S}$. Since  the flow of system ~\eqref{eq:poly_rep_4} may be seen as the the lift of the first return map to $\Pi_\mathcal{S}$ it implies that there exists a  two-dimensional invariant manifold repelling all trajectories nearby. By Fact \ref{fact:no_more_invariants}, there are no more invariant sets besides $\mathcal{M}_{\boldsymbol{\mu}}$. Therefore, this invariant line should correspond to the cycle within $\mathcal{H}$ containing the $\omega$-limit of all points of  $ {\mathcal{M}_{\boldsymbol{\mu}}}$.
\end{proof}

\begin{table}[htb]
\tiny
\begin{center}
\begin{tabular}{|c|c|c|c|}
\toprule
$\boldsymbol{\mu}$  &  Eigenvectors of $M_{\mathcal{H}}$ in $\Pi_{\gamma_5}$   & Eigenvectors of $M_{\mathcal{H}}$ in $\Pi_{\gamma_8}$   &   Phase space $[0,1]^3$ \\[1mm]
\toprule \\[-2mm]
$\boldsymbol{\mu}=90$ &
\includegraphics[width=4.5cm]{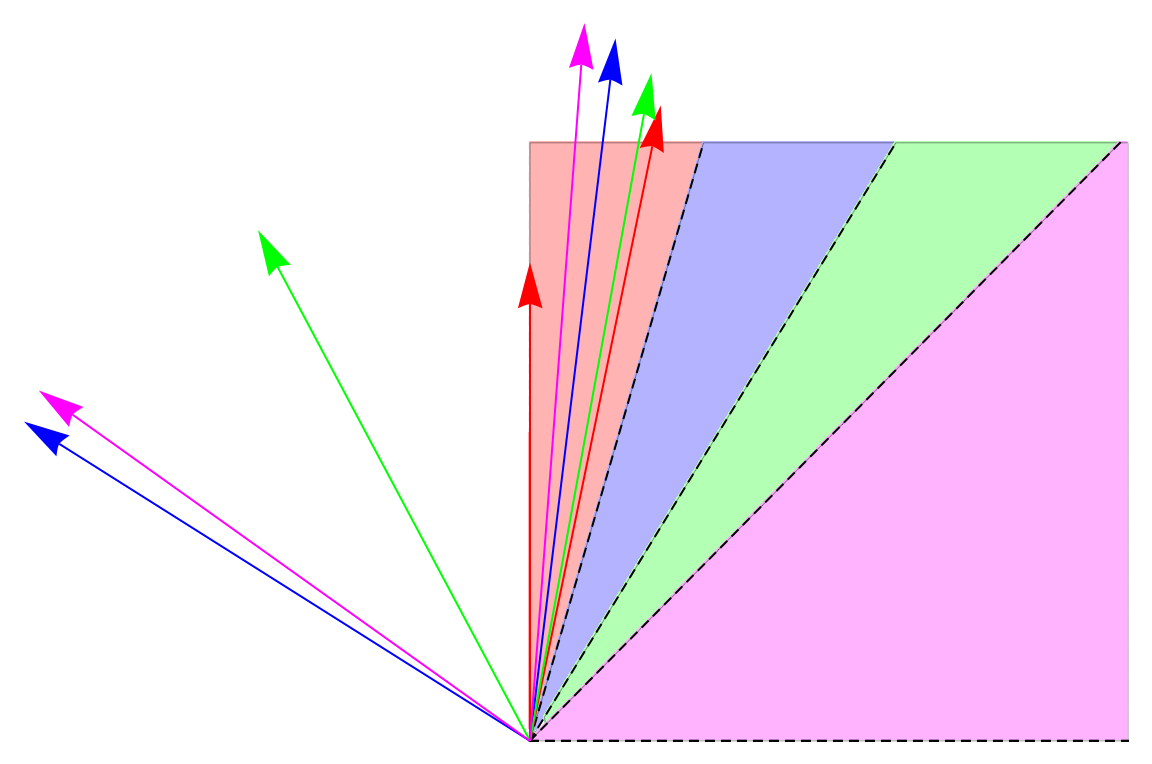} &
\includegraphics[width=4.5cm]{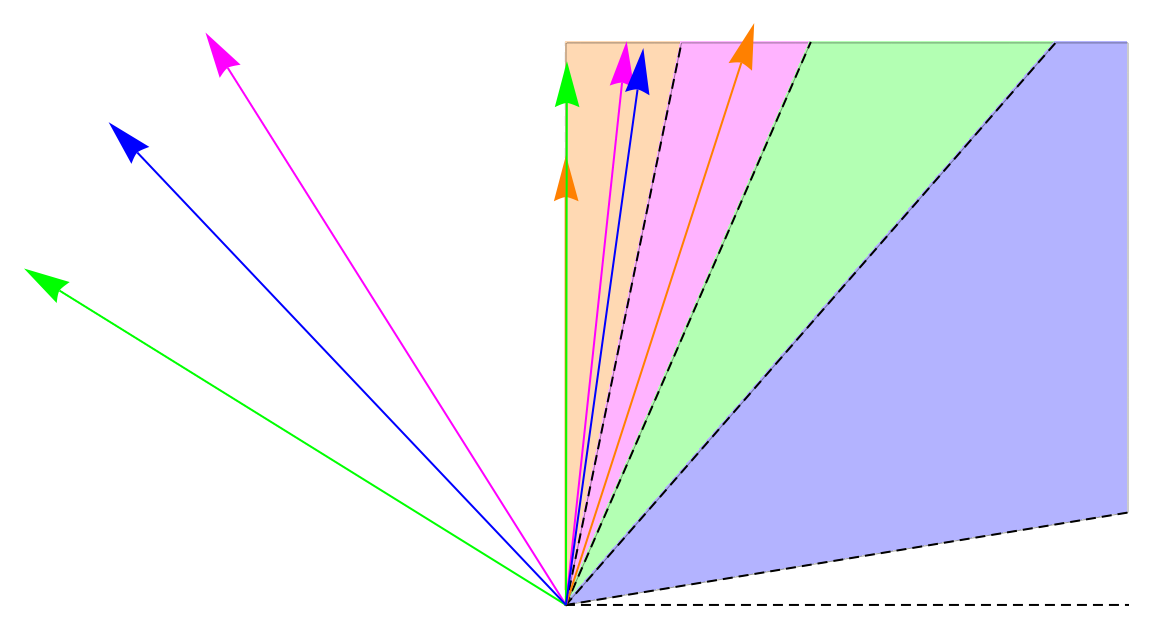} &
\includegraphics[width=3.2cm]{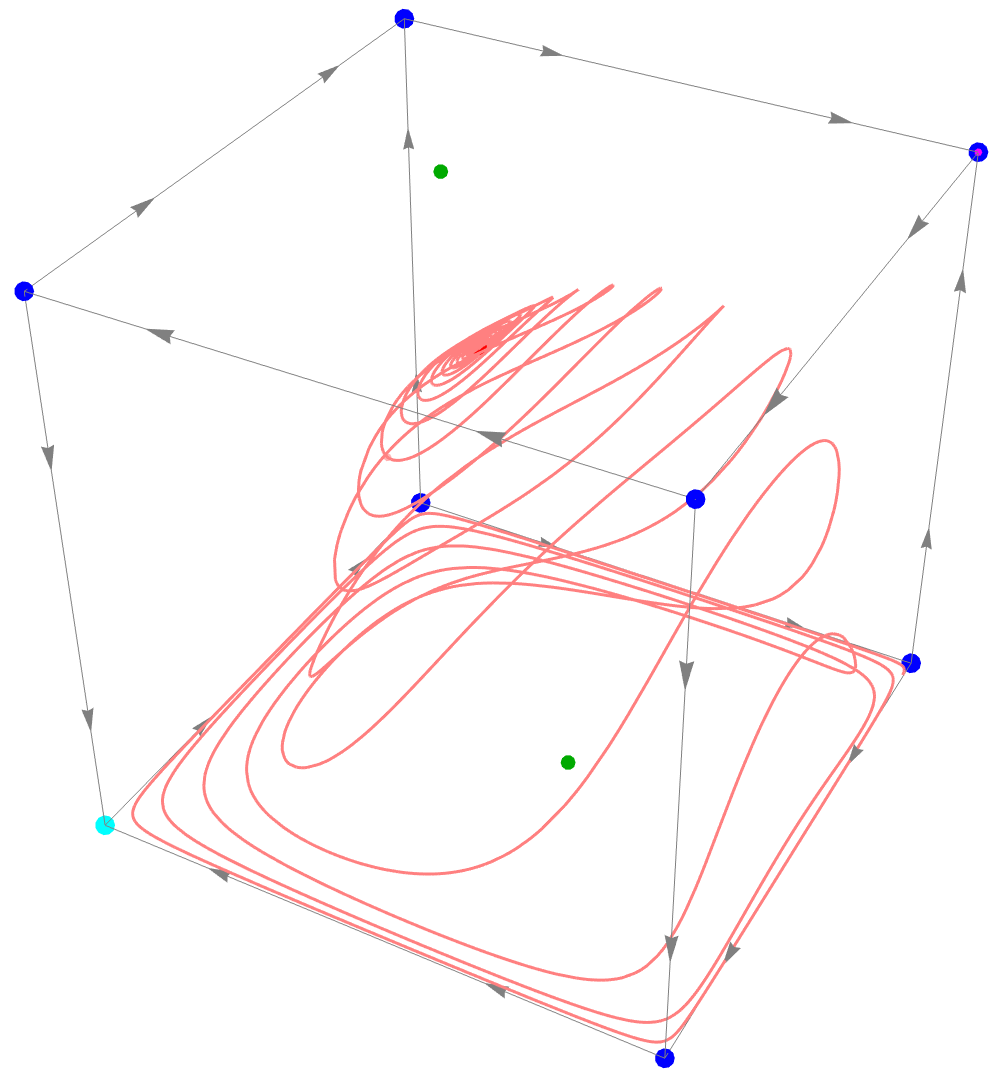}
\\[2mm] \midrule \\[-2mm]

$\boldsymbol{\mu}=96$ &
\includegraphics[width=5cm]{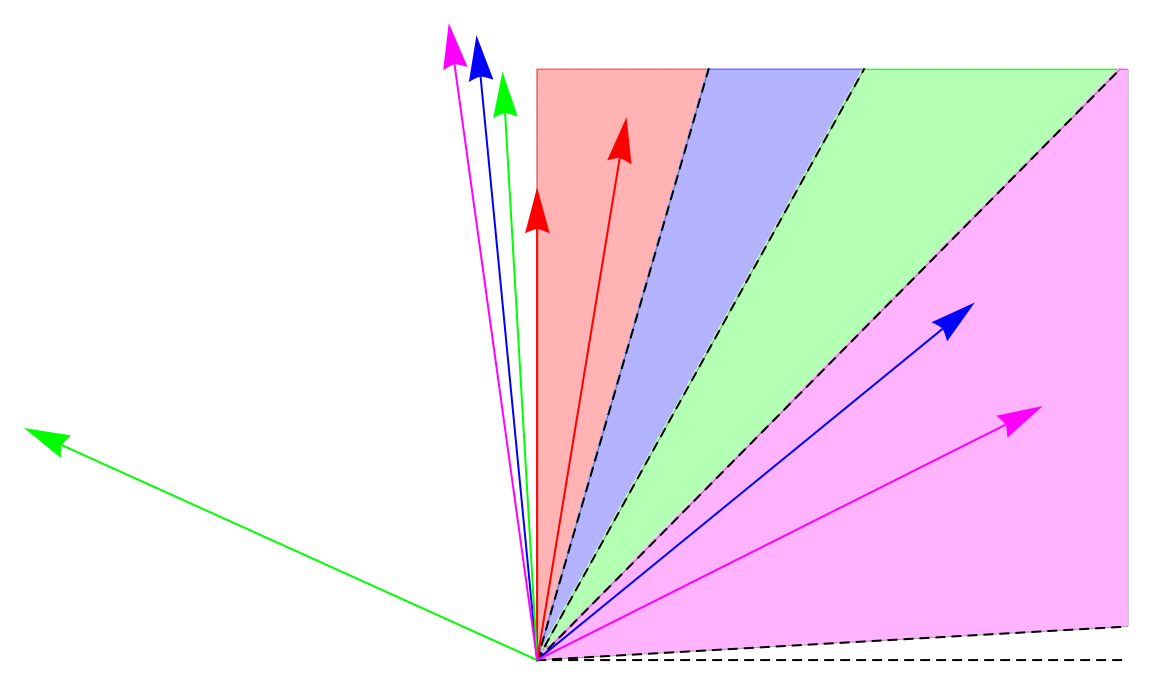} &
\includegraphics[width=3.5cm]{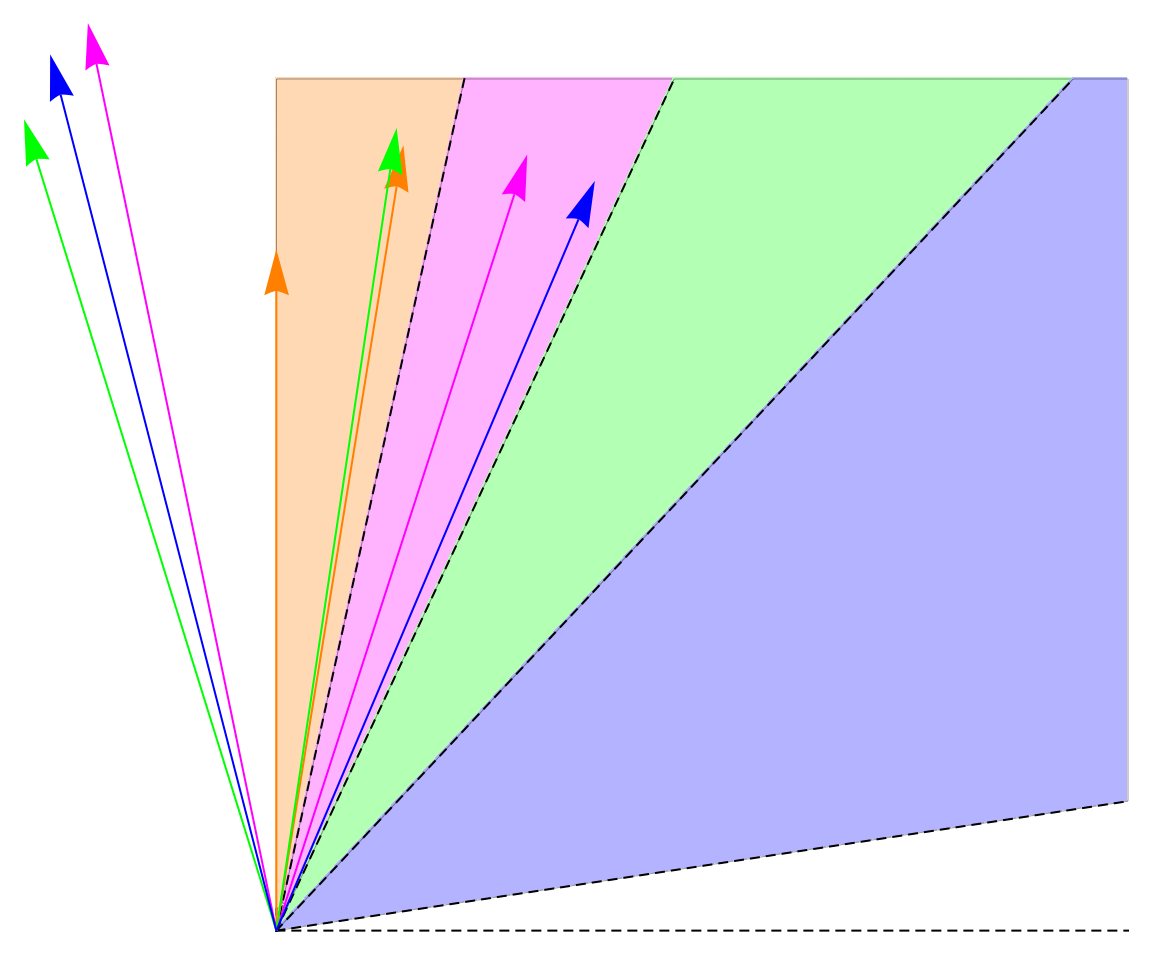} &
\includegraphics[width=3.2cm]{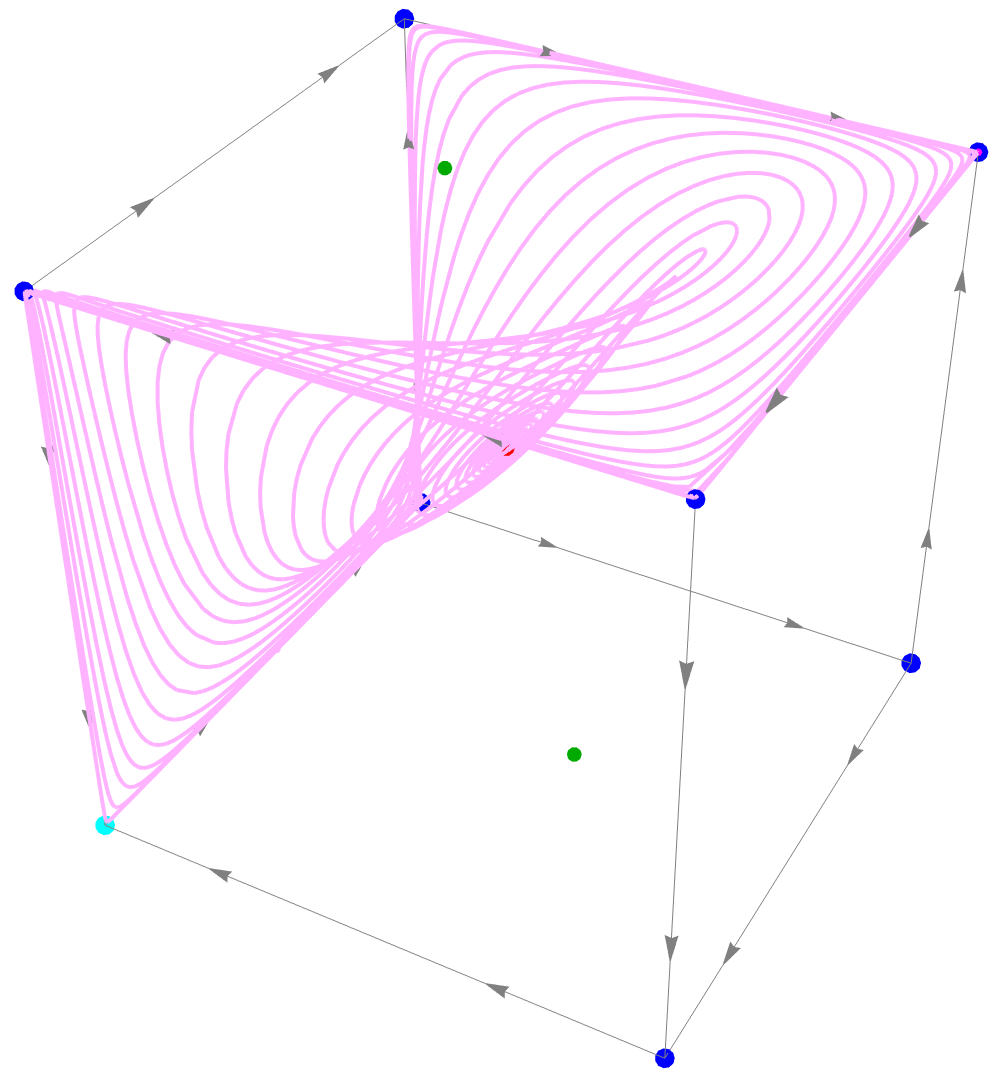}
\\[2mm] \midrule \\[-2mm]

$\boldsymbol{\mu}=99$ &
\includegraphics[width=3.5cm]{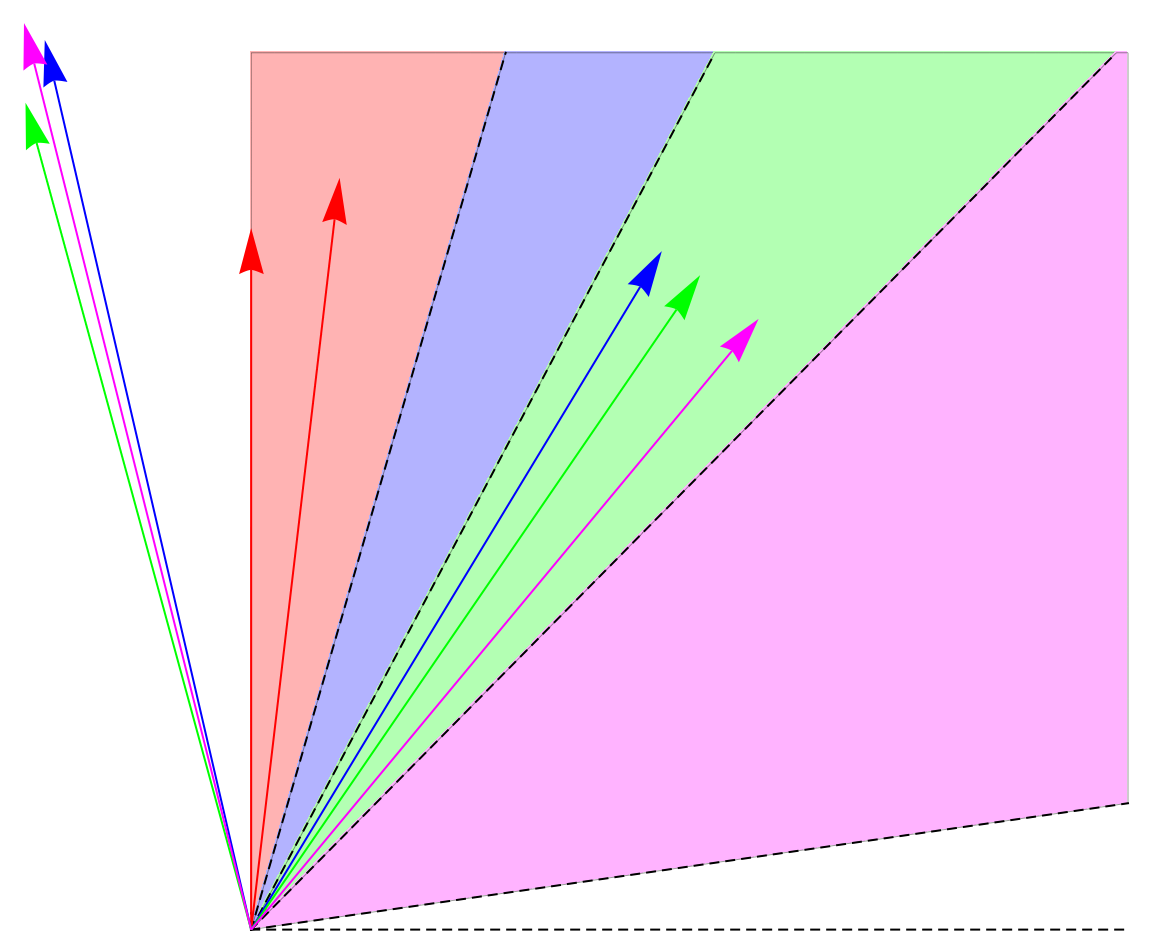} &
\includegraphics[width=3.1cm]{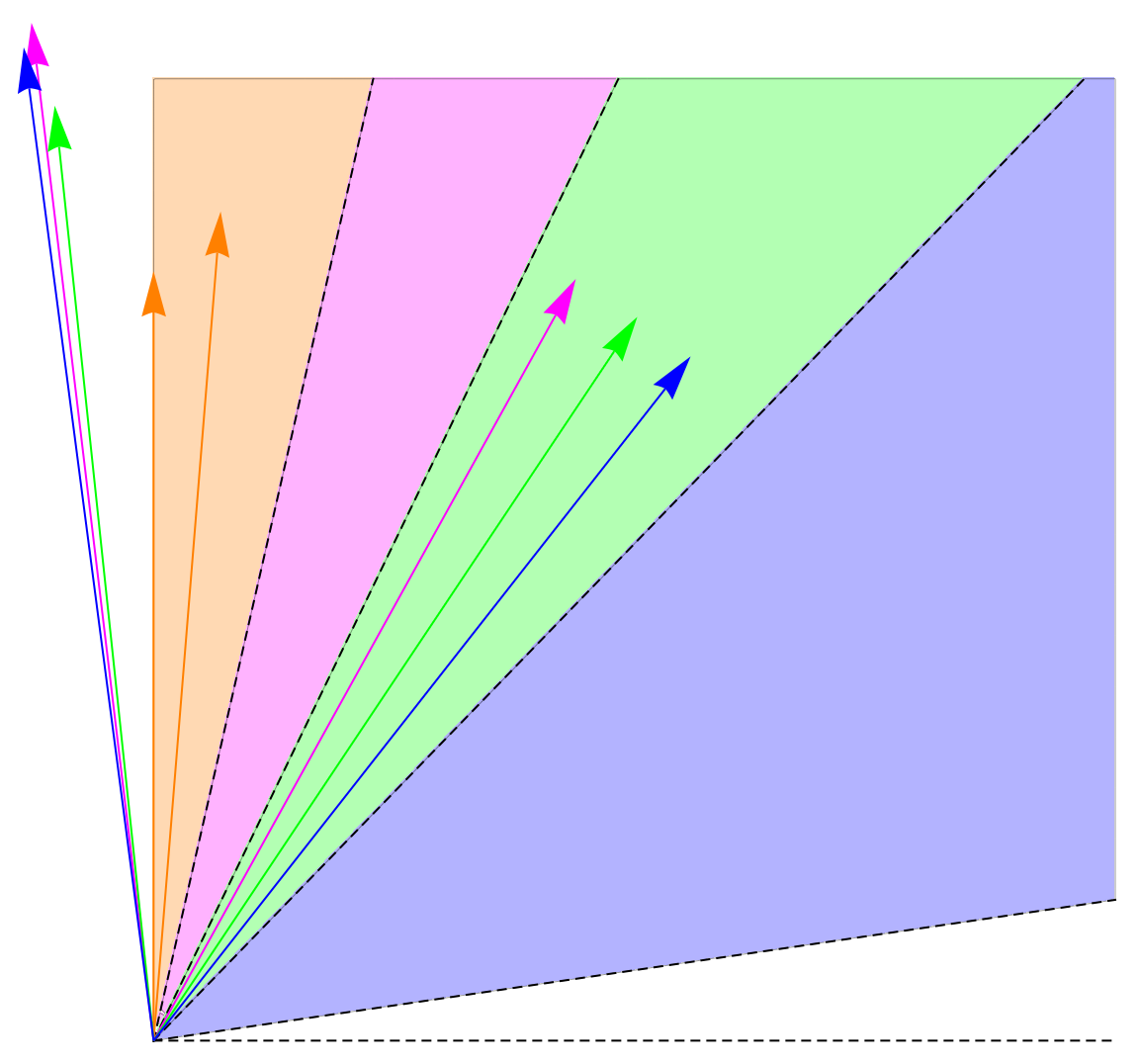} &
\includegraphics[width=3.2cm]{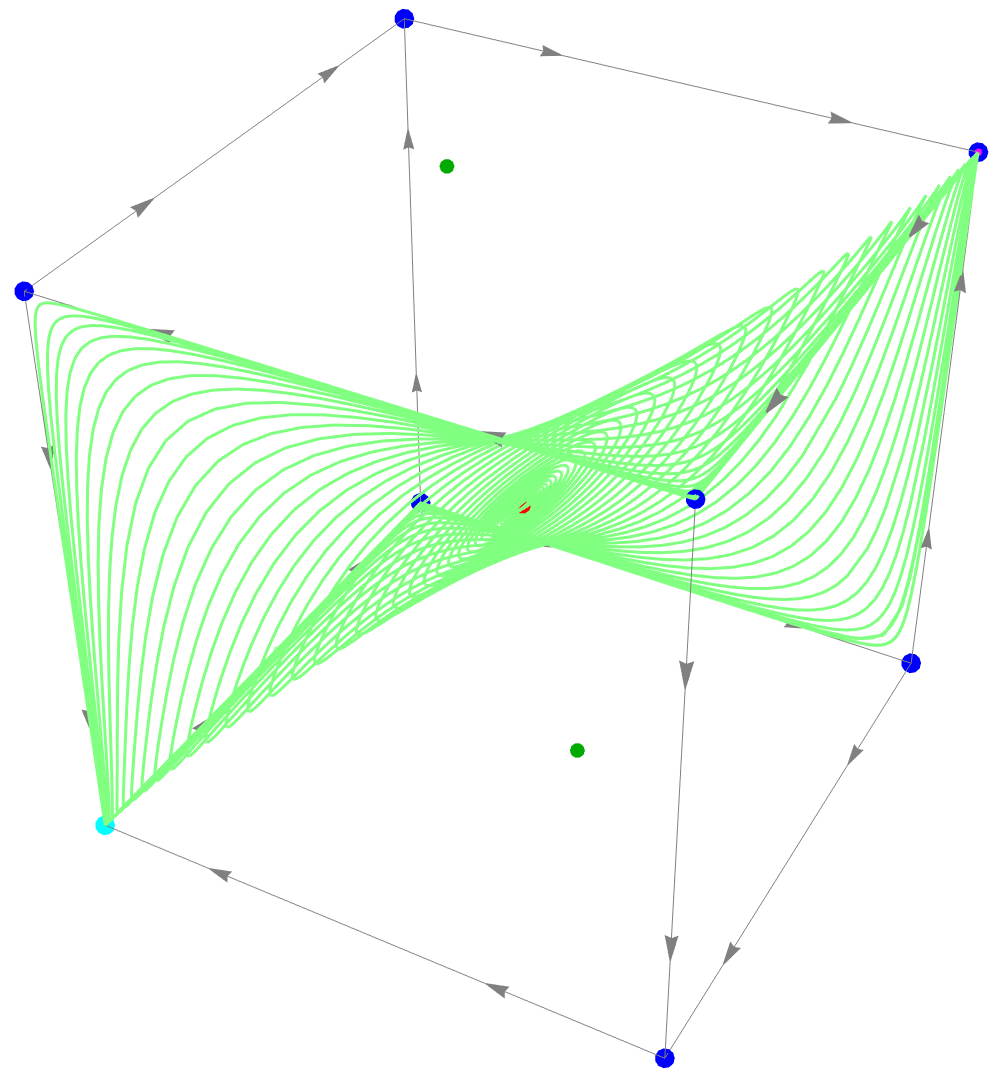}
\\[2mm] \midrule \\[-2mm]

$\boldsymbol{\mu}=101$ &
\includegraphics[width=3.5cm]{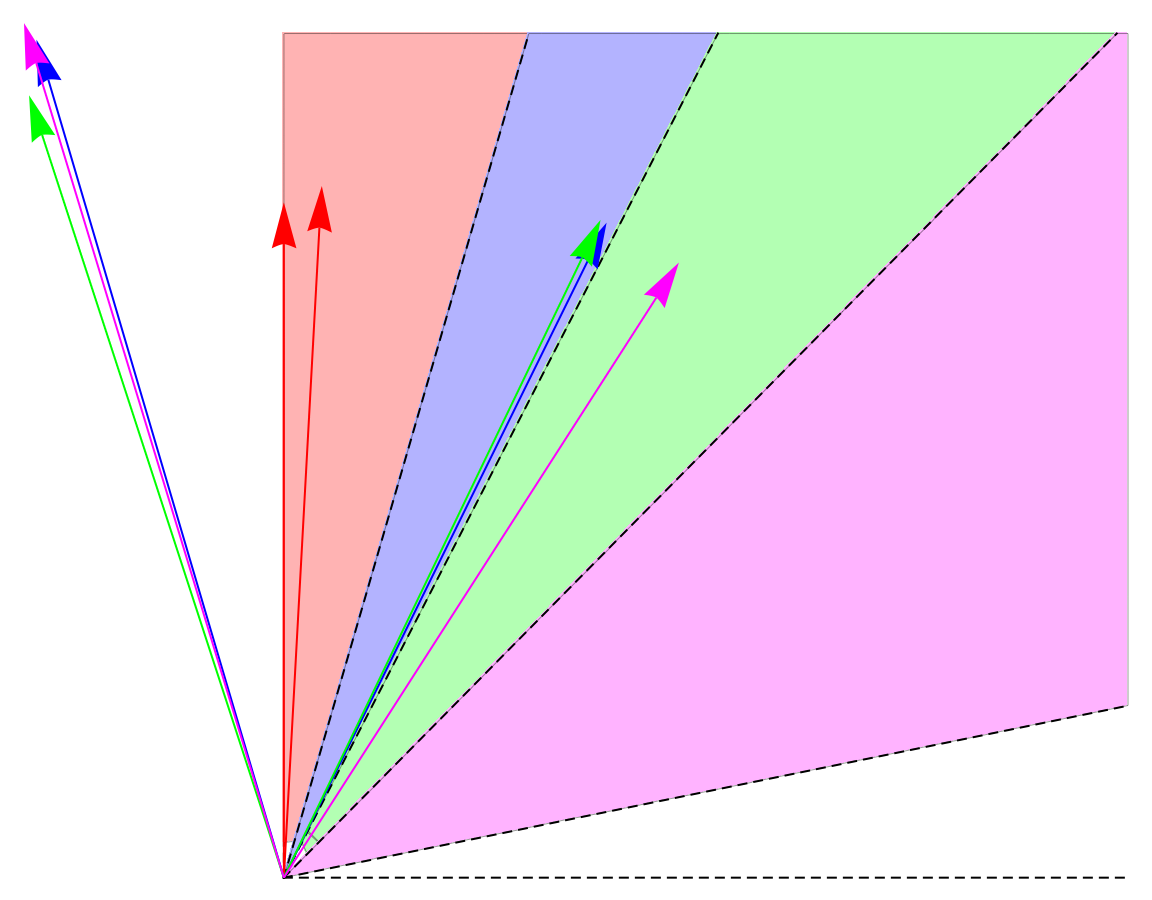} &
\includegraphics[width=3.cm]{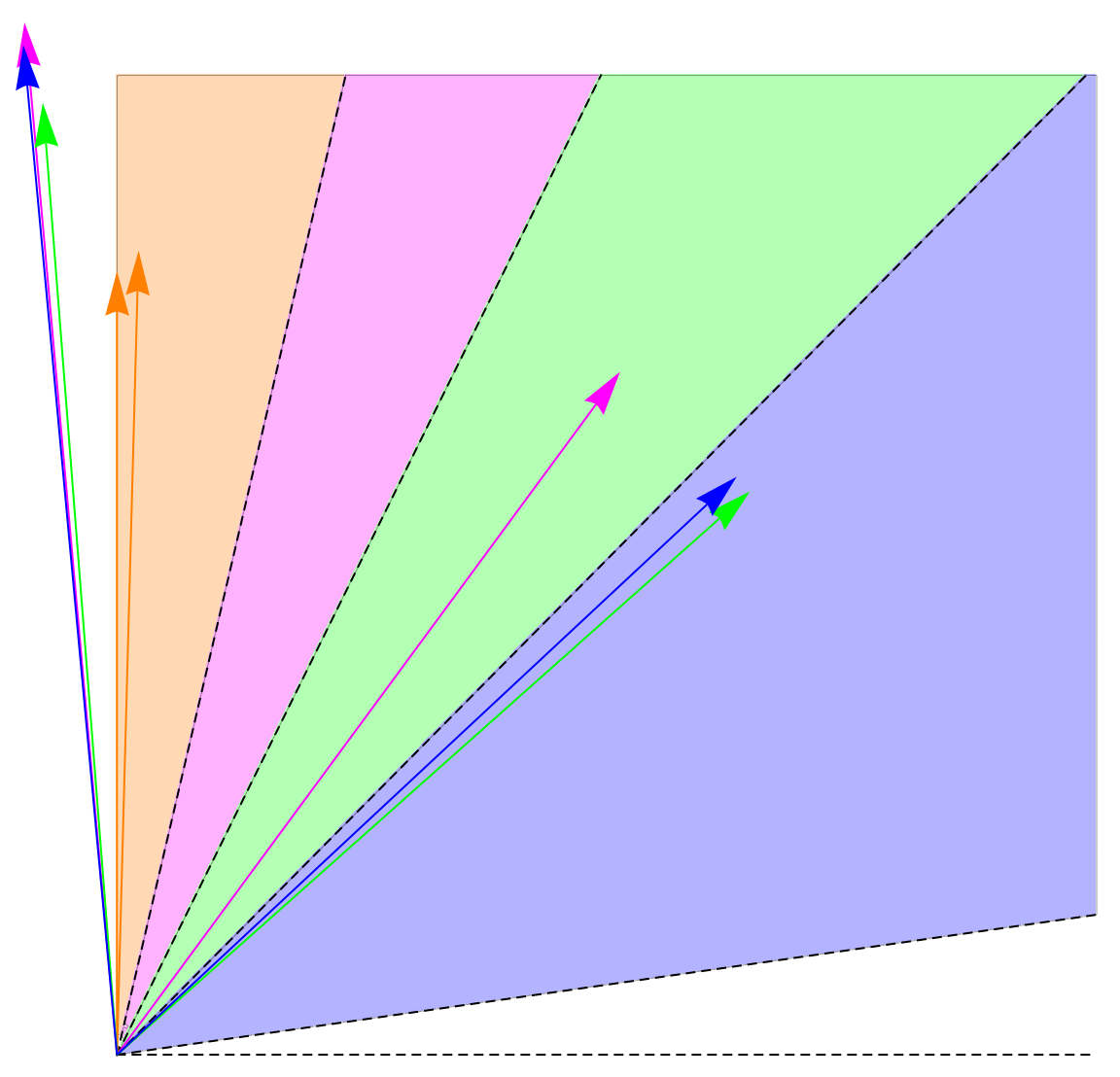} &
\includegraphics[width=3.2cm]{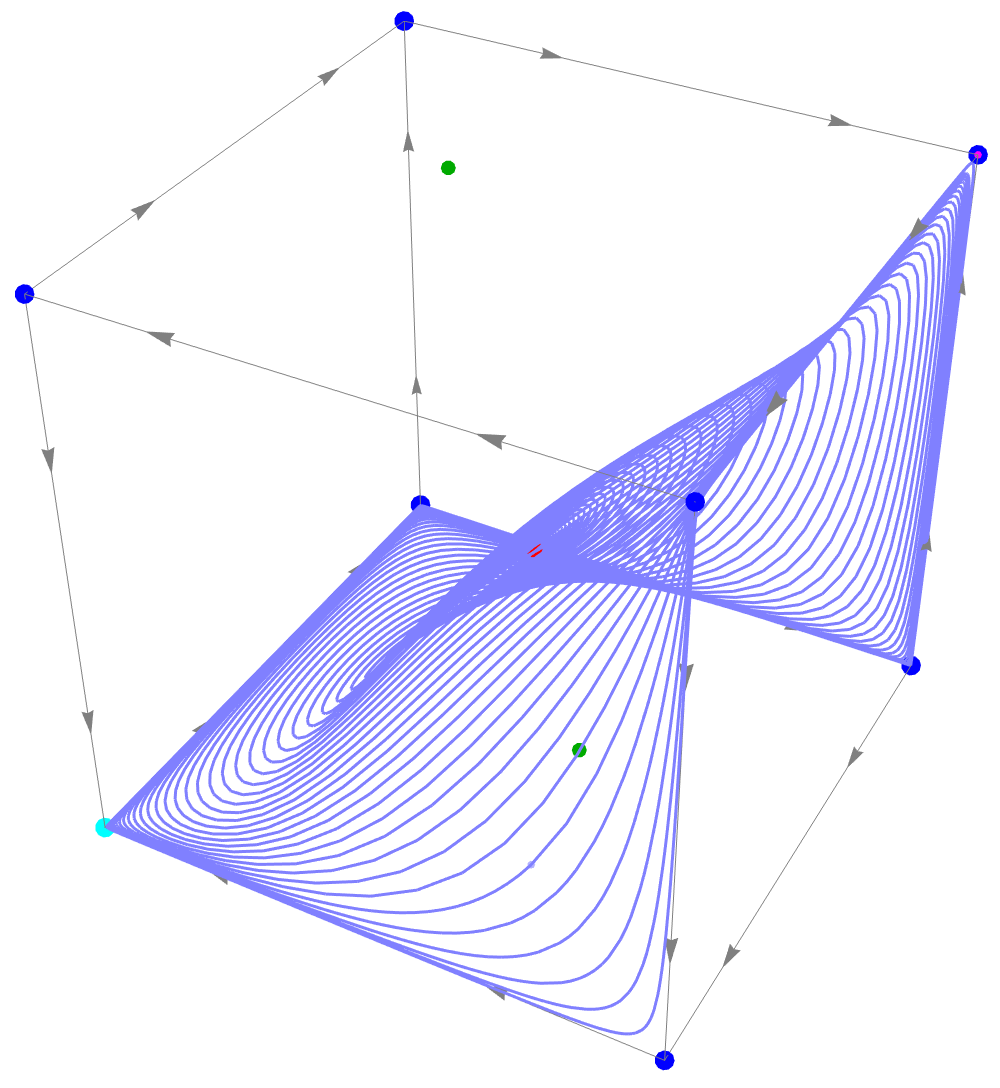}
\\[2mm]
\bottomrule

\end{tabular}
\end{center}
\bigskip
\caption{\small The eigenvectors of $M_{\mathcal{H}}$ and the corresponding sectors $\Pi_{\mathcal{H}}$ in $\Pi_{\gamma_5}$ and $\Pi_{\gamma_8}$ (where for each cycle $\mathcal{H}$, the color of the eigenvectors of $M_{\mathcal{H}}$ is the same of the corresponding sector $\Pi_{\mathcal{H}}$),
and the plot of an orbit of system \eqref{eq:poly_rep_5} whith initial condition in the interior of the phase space near $\mathcal{M}_{\boldsymbol{\mu}}$, for different values of $\boldsymbol{\mu}\in \mathcal{I}_1$.}
\label{tbl:vectors_and_orbits_mu}
\end{table}

\begin{cor}
\label{Lema19}
For $\boldsymbol{\mu} \in \mathcal{I}_2 \cup  \mathcal{I}_3$, the following assertions hold:\\
\begin{enumerate}
	\item the set $\overline{\mathcal{M}_{\boldsymbol{\mu}}}$ accumulates on $\mathcal{H}_5$;\\
	\item the set $\overline{\mathcal{M}_{\boldsymbol{\mu}}}$ divides $[0,1]^3$ in two connected components, each one containing either $B_1$ or $B_2$;\\
	\item for $ z\in \inter \left( [0,1]^3 \right)\setminus \mathcal{M}_{\boldsymbol{\mu}}$, $\omega(z)$ is either $\{B_1\}$ or $\{B_2\}$, according to the connected component  where $z$ lies.
\end{enumerate}
\end{cor}

The proof of Corollary \ref{Lema19} runs along the same arguments of Corollary \ref{prop:projective_cycles}. We have $\mathcal{L} (\inter \left( [0,1]^3 \right)\setminus \mathcal{M}_{\boldsymbol{\mu}})=\{B_1, B_2\}$ because   $\mathcal{M}_{\boldsymbol{\mu}} \cup \mathcal{H}$ is repelling and there are no more compact invariant sets candidates for $\omega$-limit sets.

\begin{table}[htb]
\tiny
\begin{center}
\begin{tabular}{|c|c|c|c|}
\toprule
{Parameter Interval}  &  $\mathcal{L}(\mathcal{B}(\mathcal{H}))$   & \quad\quad  Basin of attraction of $\mathcal{L}(\mathcal{B}(\mathcal{H}))$ \qquad\qquad   &   $\overline{\mathcal{M}_{\boldsymbol{\mu}}}$ ``glues'' at    \\[1mm]
\toprule \\[-2mm]

$[850/11, \mu_3[ $ &  $\mathcal{H}_6$ & $]0,1[^3\backslash \mathcal{M}_{\boldsymbol{\mu}}$  & $\mathcal{H}_6$   \\[3mm]
\midrule \\[-1mm]

$[\mu_3, \mu_4[ $ &  $\mathcal{H}_1\cup \mathcal{H}_6$ & Each CC of  $]0,1[^3\backslash \mathcal{M}_{\boldsymbol{\mu}}$ accumulates either on  $\mathcal{H}_1$ or $\mathcal{H}_6$ & $\mathcal{H}_2$   \\[3mm]
\midrule \\[-1mm]

$[\mu_4, \mu_5[ $ &  $\mathcal{H}_1\cup \mathcal{H}_6$ &  Each CC of  $]0,1[^3\backslash \mathcal{M}_{\boldsymbol{\mu}}$ accumulates either on  $\mathcal{H}_1$ or $\mathcal{H}_6$  & $\mathcal{H}_4$   \\[3mm]
\midrule \\[-1mm]

$[\mu_5, 102[ $ &  $\mathcal{H}_1\cup \mathcal{H}_6$ &  Each CC of  $]0,1[^3\backslash \mathcal{M}_{\boldsymbol{\mu}}$ accumulates either on  $\mathcal{H}_1$ or $\mathcal{H}_6$  & $\mathcal{H}_5$   \\[3mm]
\midrule \\[-1mm]

$\mathcal{I}_2 \cup \mathcal{I}_3 $ &  $\{B_1, B_2\}$ &Each CC of  $]0,1[^3\backslash \mathcal{M}_{\boldsymbol{\mu}}$ accumulates either on  $B_1$ or $B_2$  & $\mathcal{H}_5$   \\[3mm]

\bottomrule

\end{tabular}
\end{center}
\bigskip
\caption{\small Summary of concluding results for system \eqref{eq:poly_rep_5}, where $\boldsymbol{\mu}_3=94$, $\boldsymbol{\mu}_4=\frac{85251}{869}$ and $\boldsymbol{\mu}_5=\frac{85234}{849}$. CC: connected component.}\label{tbl:notationA}
\end{table}

\section{Implemented software code}\label{sec:software_code}
We provide in \url{https://www.iseg.ulisboa.pt/aquila/homepage/telmop/investigacao/flows-on-polytopes---mathematica-code}  the Mathematica code we developed to explore the dynamics of polymatrix replicators for low dimensional polytopes (First author's personal webpage).   

\section{Discussion}\label{sec:discussion}
 In the present article, by using the theory introduced in \cite{alishah2019asymptotic}, we develop a method to study the asymptotic dynamics near an attracting heteroclinic network $ \mathcal{H}$ formed by six one-dimensional cycles involving hyperbolic equilibria (lying on the boundary of a cube). We have described a general way to compute the likely limit set associated to the basin of attraction of the network. 
 
Our study  contributes to a deeper understanding of the results obtained in  \cite{alishah2015hamiltonian,alishah2015conservative, afraimovich2016two}, where numerical simulations evidenced the visibility of two cycles. 
In our model (defined by systems \eqref{eq:poly_rep_4} or \eqref{eq:poly_rep_5}), the parameter $\boldsymbol{\mu}$ represents the average payoffs in the context of EGT.
 We concluded that whenever the parameter $\boldsymbol{\mu}$ lies on $\mathcal{I}=\left[\frac{850}{11},\frac{544}{5} \right]$, the associated dynamics is non-chaotic and a given set of strategies  dominates. Our results extend to models other than the Lotka-Volterra that preserve the invariance of coordinate lines and hyperplanes.

Our method has simililarities with the \emph{transitions matrices technique} used by Krupa and Melbourne \cite{krupa1995asymptotic} and Castro and Garrido-da-Silva \cite{garrido2019stability}. The main advantage of our method is twofold. First, the dynamics in a given cross section may be seen as a piecewise linear map  where the classical Perron-Frobenius theory of linear operators may be easily used. The analysis is computationally much more amenable than the classical method. 

Secondly, the reduction to a one-dimensional projective map  allows us to construct a bridge between its  periodic points and the existence of heteroclinic cycles (in  the flow), as well as their stability.   In contrast to the findings of \cite{podvigina2020asymptotic}, we do not need the assumption that the network is \emph{clean}. 

 Our class of examples is related to the dynamical systems represented by ODEs that support the dynamics of the Rock-Scissors-Paper-Lizard-Spock game \cite{postlethwaite2021stability} and Lotka-Volterra systems constructed using the methods of \cite{field2020lectures, ashwin2013designing}. See also \cite{aguiar2011there, podvigina2020asymptotic}. Although there are similarities between system \eqref{eq:poly_rep_5} and Equation (6)  of \cite{peixe2021persistent}, the associated dynamics are very different. While, in the latter case, the existence of  chaos is a persistent phenomenon, in the first case, the dynamics   exhibits zero topological entropy.

\subsection*{Classical method: an overview}
The classical method to analyse the stability of cycles and networks  is based in the following procedure:  assuming a non-resonance condition on the spectrum of the linearization of the vector field at the equilibria, we approximate the behaviour of nearby trajectories by composing local and global maps.  For compact networks, global maps are linear whose coefficients  are bounded from above. The   estimates for local maps near the saddles involve exponents of eigenvalues ratios.

A network is \emph{stable} if certain products of the exponents appearing in the expression of the first return map to a cross section are larger than one.  
According to their role on the network,  the eigenvalues  can be classified as \emph{radial, contracting, expanding and transverse} (see \cite{podvigina2015simple, podvigina2017asymptotic}) .  The estimates for local maps depend on the local structure of  the network near the equilibria.  
In the presence of symmetry (or other constraints), the application of the method is slightly different since the fixed-point subspaces may be seen as borders that cannot be crossed.

\subsection*{Our technique: a summary}
Looking to a heteroclinic network $\mathcal{H}$  (formed by one-dimen\-sional connections) on a manifold with boundary, we consider the set $\mathcal{S}$, called \emph{structural set}, consisting of heteroclinic connections such that every cycle of the  network contains at least one connection in $\mathcal{S}$. Given a structural set $\mathcal{S}$, we denote by $\Sigma$   the union of cross sections to $\mathcal{S}$, one at each heteroclinic connection in $\mathcal{S}$. The flow   induces a Poincar\'e return map, say  $P_\mathcal{S}$, to $\Sigma$, designated as the $\mathcal{S}$-Poincar\'e map   associated to each possible itinerary that starts and ends at $\mathcal{S}$. This map captures well the global dynamics near $\mathcal{H}$.
 

Using the  quasi-change of coordinates of Section \ref{subsec:dual_cone}, at the level of the dual cone,  we obtain a return map $\pi_\mathcal{S}$  well defined on the union  of the corresponding  sections (up to a set with zero Lebesgue zero),  denoted by $\Pi_\mathcal{S}$. 

After making explicit the piecewise linear skeleton map $\pi_\mathcal{S}: D_\mathcal{S}^{\ast}  \rightarrow \Pi_\mathcal{S}$ of Proposition \ref{prop15}, we use an algorithm to compute the associated matrix $M_\xi$ for each $\xi \in B_\mathcal{S}$, as well the inequalities defining the domain of each sector $\Pi_\xi$.   Using the asymptotic of linear maps, all solutions approach the eigendirection associated to the greatest eigenvalue, as a consequence of the Perron-Frobenius Theory.  

The   map $\pi_\mathcal{S}$ carries the asymptotic behaviour of  {$P_\mathcal{S}$} along the different paths in the sense that after a rescaling change of coordinates $\Psi_\varepsilon$,  $\pi_\mathcal{S}$ is the   limit of $\Psi_\varepsilon \circ {P_\mathcal{S}} \circ (\Psi_\varepsilon)^{-1}$ as $\varepsilon$ tends to $0^+$ (in the $C^1$--topology).  

Because the   map $\pi_\mathcal{S}$ is easily computable, we can run an algorithm to find the $\pi_\mathcal{S}$-invariant linear algebra structures, provided their eigenvalues are two different positive real numbers. If these structures are invariant under small non-linear perturbations, they will persist as invariant geometric structures for  {$P_\mathcal{S}$}, and hence for the flow. Under the assumption that there are no compact invariant sets in the interior of the cube, we also make use of this stability principle to prove the existence of normally hyperbolic  manifolds for heteroclinic cycles satisfying some appropriate conditions.

The intersection of each iterate of $\pi_\mathcal{S}$ with the line $\overline{u}=1$ generates the \emph{projective map} $\hat\pi_\mathcal{S}$. The saddle-value given by the ratio between the eigenvalues of $D \pi_\mathcal{S}$ at the corresponding fixed point determine its stability (that is associated to a given cycle).

The connection between the stability of periodic points for the projective map and the stability of
the original heteroclinic cycles is summarized in Table \ref{notationC}.

 Results of Subsection \ref{ss:skeleton} provides the most important breakthrough in the study of  stability for networks on Lotka-Volterra systems because the local and global maps are  stated according to the architecture of the network. They depend on the coordinates of the system allowing a systematic study of all subcycles of $\mathcal{H}$.
 This technique  may be generalized for other vector fields defined on a manifold isomorphic  to $[0,1]^n$, $n\in \NN$, containing a heteroclinic network on the boundary.

\subsection*{Future work}
The natural continuation work of this article is the application of our method in higher dimensions. The most intriguing question is to know how switching properties of the network may be realized in ``switching properties'' of the projective map. Another question is the relation between the $\boldsymbol{\mu}_3, \boldsymbol{\mu}_4, \boldsymbol{\mu}_5$ value with some linear combination of the eigenvalues of $Df_{\boldsymbol{\mu}}$  at the equilibria. These questions are deferred for future work.

\section*{Acknowledgements}

The authors are grateful to Pedro Duarte for the suggestion of the projective map defined in Section \ref{sec:projective_map} during the first author's PhD period.

The first author was supported by
the Project CEMAPRE/REM -- UIDB /05069/2020 financed by FCT/MCTES through national funds.
The second author was partially supported by CMUP (UID/MAT/00144/2019), which is funded by FCT with national (MCTES) and European structural funds through the programs FEDER, under the partnership agreement PT2020. He also acknowledges financial support from Program INVESTIGADOR FCT (IF/ 0107/ 2015).


\bibliographystyle{unsrt}
\bibliography{Bibfile}

\begin{thebibliography}{10}

\bibitem{field1991stationary}
Mike Field and James~W Swift.
\newblock Stationary bifurcation to limit cycles and heteroclinic cycles.
\newblock {\em Nonlinearity}, 4(4):1001, 1991.

\bibitem{podvigina2011local}
Olga Podvigina and Peter Ashwin.
\newblock On local attraction properties and a stability index for heteroclinic
  connections.
\newblock {\em Nonlinearity}, 24(3):887, 2011.

\bibitem{krupa1995asymptotic}
Martin Krupa and Ian Melbourne.
\newblock Asymptotic stability of heteroclinic cycles in systems with symmetry.
\newblock {\em Ergodic Theory and Dynamical Systems}, 15(1):121--147, 1995.

\bibitem{rodrigues2013persistent}
Alexandre~AP Rodrigues.
\newblock Persistent switching near a heteroclinic model for the geodynamo
  problem.
\newblock {\em Chaos, Solitons \& Fractals}, 47:73--86, 2013.

\bibitem{field2020lectures}
Michael~J Field.
\newblock {\em Lectures on bifurcations, dynamics and symmetry}.
\newblock CRC Press, 2020.

\bibitem{hofbauer1987permanence}
Josef Hofbauer and Karl Sigmund.
\newblock Permanence for replicator equations.
\newblock In {\em Dynamical systems}, pages 70--91. Springer, 1987.

\bibitem{hofbauer1998evolutionary}
Josef Hofbauer, Karl Sigmund, et~al.
\newblock {\em Evolutionary games and population dynamics}.
\newblock Cambridge university press, 1998.

\bibitem{gaunersdorfer1995fictitious}
Andrea Gaunersdorfer and Josef Hofbauer.
\newblock Fictitious play, shapley polygons, and the replicator equation.
\newblock {\em Games and Economic Behavior}, 11(2):279--303, 1995.

\bibitem{labouriau2017takens}
Isabel~S Labouriau and Alexandre~AP Rodrigues.
\newblock On takens’ last problem: tangencies and time averages near
  heteroclinic networks.
\newblock {\em Nonlinearity}, 30(5):1876, 2017.

\bibitem{podvigina2015simple}
Olga Podvigina and Pascal Chossat.
\newblock Simple heteroclinic cycles in.
\newblock {\em Nonlinearity}, 28(4):901, 2015.

\bibitem{podvigina2020asymptotic}
Olga Podvigina, Sofia~BSD Castro, and Isabel~S Labouriau.
\newblock Asymptotic stability of robust heteroclinic networks.
\newblock {\em Nonlinearity}, 33(4):1757, 2020.

\bibitem{melbourne1991example}
Ian Melbourne.
\newblock An example of a nonasymptotically stable attractor.
\newblock {\em Nonlinearity}, 4(3):835, 1991.

\bibitem{podvigina2012stability}
Olga Podvigina.
\newblock Stability and bifurcations of heteroclinic cycles of type z.
\newblock {\em Nonlinearity}, 25(6):1887, 2012.

\bibitem{podvigina2017asymptotic}
Olga Podvigina and Pascal Chossat.
\newblock Asymptotic stability of pseudo-simple heteroclinic cycles in
  $\mathbb{R}^{4}$.
\newblock {\em Journal of Nonlinear Science}, 27(1):343--375, 2017.

\bibitem{garrido2019stability}
Liliana Garrido-da Silva and Sofia~BSD Castro.
\newblock Stability of quasi-simple heteroclinic cycles.
\newblock {\em Dynamical Systems}, 34(1):14--39, 2019.

\bibitem{lohse2015unstable}
Alexander Lohse.
\newblock Unstable attractors: existence and stability indices.
\newblock {\em Dynamical Systems}, 30(3):324--332, 2015.

\bibitem{smith1973logic}
John~Maynard Smith and George~Robert Price.
\newblock The logic of animal conflict.
\newblock {\em Nature}, 246(5427):15--18, 1973.

\bibitem{peixe2021persistent}
Telmo Peixe and Alexandre~A Rodrigues.
\newblock Persistent strange attractors in 3d polymatrix replicators.
\newblock {\em arXiv preprint arXiv:2103.11242}, 2021.

\bibitem{alishah2019asymptotic}
Hassan~Najafi Alishah, Pedro Duarte, and Telmo Peixe.
\newblock Asymptotic {P}oincar{\'e} maps along the edges of polytopes.
\newblock {\em Nonlinearity}, 33(1):469, 2019.

\bibitem{alishah2015hamiltonian}
Hassan~Najafi Alishah and Pedro Duarte.
\newblock Hamiltonian evolutionary games.
\newblock {\em Journal of Dynamics \& Games}, 2(1):33, 2015.

\bibitem{alishah2015conservative}
Hassan~Najafi Alishah, Pedro Duarte, and Telmo Peixe.
\newblock Conservative and dissipative polymatrix replicators.
\newblock {\em Journal of Dynamics \& Games}, 2(2):157, 2015.

\bibitem{peixe2019permanence}
Telmo Peixe.
\newblock Permanence in polymatrix replicators.
\newblock {\em Journal of Dynamics \& Games}, page~0, 2019.

\bibitem{guckenheimer2013nonlinear}
John Guckenheimer and Philip Holmes.
\newblock {\em Nonlinear oscillations, dynamical systems, and bifurcations of
  vector fields}, volume~42.
\newblock Springer Science \& Business Media, 2013.

\bibitem{podvigina2019stability}
Olga Podvigina, Sofia~BSD Castro, and Isabel~S Labouriau.
\newblock Stability of a heteroclinic network and its cycles: a case study from
  boussinesq convection.
\newblock {\em Dynamical Systems}, 34(1):157--193, 2019.

\bibitem{milnor1985concept}
John Milnor.
\newblock On the concept of attractor.
\newblock In {\em The theory of chaotic attractors}, pages 243--264. Springer,
  1985.

\bibitem{castro2010heteroclinic}
Sofia~BSD Castro, Isabel~S Labouriau, and Olga Podvigina.
\newblock A heteroclinic network in mode interaction with symmetry.
\newblock {\em Dynamical Systems}, 25(3):359--396, 2010.

\bibitem{ruelle2014elements}
David Ruelle.
\newblock {\em Elements of differentiable dynamics and bifurcation theory}.
\newblock Elsevier, 2014.

\bibitem{palis1982local}
Jacob Palis and Welington de~Melo.
\newblock Local stability.
\newblock In {\em Geometric Theory of Dynamical Systems}, pages 39--90.
  Springer, 1982.

\bibitem{rodrigues2017attractors}
Alexandre~AP Rodrigues.
\newblock Attractors in complex networks.
\newblock {\em Chaos: An Interdisciplinary Journal of Nonlinear Science},
  27(10):103105, 2017.

\bibitem{katok1997introduction}
Anatole Katok and Boris Hasselblatt.
\newblock {\em Introduction to the modern theory of dynamical systems}.
\newblock Number~54. Cambridge university press, 1997.

\bibitem{bunimovich2012isospectral}
L.~A. Bunimovich and B.~Z. Webb.
\newblock Isospectral compression and other useful isospectral transformations
  of dynamical networks.
\newblock {\em Chaos: An Interdisciplinary Journal of Nonlinear Science},
  22(3):--, 2012.

\bibitem{peixe2015lotka}
Telmo Peixe.
\newblock {\em Lotka-{V}olterra {S}ystems and {P}olymatrix {R}eplicators}.
\newblock ProQuest LLC, Ann Arbor, MI, 2015.
\newblock Thesis (Ph.D.)--Universidade de Lisboa (Portugal).

\bibitem{afraimovich2016two}
Valentin~S Afraimovich, Gregory Moses, and Todd Young.
\newblock Two-dimensional heteroclinic attractor in the generalized
  lotka--volterra system.
\newblock {\em Nonlinearity}, 29(5):1645, 2016.

\bibitem{postlethwaite2021stability}
Claire~M Postlethwaite and Alastair~M Rucklidge.
\newblock Stability of cycling behaviour near a heteroclinic network model of
  rock-paper-scissors-lizard-spock, 2021.

\bibitem{ashwin2013designing}
Peter Ashwin and Claire Postlethwaite.
\newblock On designing heteroclinic networks from graphs.
\newblock {\em Physica D: Nonlinear Phenomena}, 265:26--39, 2013.

\bibitem{aguiar2011there}
Manuela~AD Aguiar.
\newblock Is there switching for replicator dynamics and bimatrix games?
\newblock {\em Physica D: Nonlinear Phenomena}, 240(18):1475--1488, 2011.

\end{thebibliography}
\newpage
\appendix

\section{Tables}
\label{Appendix}
\begin{table}[h]

\makebox[1 \textwidth][c]{       
\resizebox{1.3 \textwidth}{!}{

\begin{tabular}{|c|c|c|c|c|}

\toprule
$\quad\xi\quad$ & Defining equations of $\Pi_{\xi}$  &  $M_{\xi}$, the matrix of $\pi_\xi$ &  Eigenvalues of $M_{\xi}$ &  Eigenvectors of $M_{\xi}$ \vspace{.1cm} \\
\toprule \\[-3mm]

$\xi_1$ & $255 u_2-74 u_6 < 0$  &    
$\left(
\begin{array}{cc}
\frac{42 (\boldsymbol{\mu} -176)}{37 (18-\boldsymbol{\mu} )} & 0 \\[2mm]
\frac{4029 (\boldsymbol{\mu} -106)}{518 (18-\boldsymbol{\mu} )} & 1 \\
\end{array}
\right)\,$
 &  $\left\{\frac{42 (176-\boldsymbol{\mu})}{37 (\boldsymbol{\mu} -18)},1\right\}$
 &  $\left\{ \left\{\frac{14 (\boldsymbol{\mu} -102)}{51 (\boldsymbol{\mu} -106)},1\right\} , \{0,1\} \right\}$
\\ \\[-3mm]
\midrule \\[-3mm]

$\xi_2$ & $75 u_2-74 u_6 > 0$  &    
$\left(
\begin{array}{cc}
\frac{11 (167-\boldsymbol{\mu} )}{1495} & \frac{162 (167-\boldsymbol{\mu} )}{37375} \\[2mm]
 \frac{3723}{1495} & -\frac{29434}{37375} \\
\end{array}
\right)\,$
 &  $\left\{ \frac{a_2-b_2}{74750}, \frac{a_2+b_2}{74750} \right\}$
 &  $\left\{ \left\{ \frac{c_2-b_2}{186150}, 1 \right\} , \left\{ \frac{c_2+b_2}{186150}, 1 \right\} \right\}$
\\ \\[-3mm]
\midrule \\[-3mm]

$\xi_3$ & $\left\{
\begin{array}{l}
75 u_2-74 u_6 < 0 \\[2mm]
 255 u_2-74 u_6 > 0
\end{array}\right.$  &    
$\left(
\begin{array}{cc}
\frac{5 (6502-41 \boldsymbol{\mu} )}{23828} & \frac{190-\boldsymbol{\mu} }{322} \\[2mm]
 \frac{2040}{851} & -\frac{16}{23} \\
\end{array}
\right)\,$
 &  $\left\{ \frac{a_3-b_3}{47656}, \frac{a_3+b_3}{47656} \right\}$
 &  $\left\{ \left\{ \frac{c_3-b_3}{114240}, 1 \right\} , \left\{ \frac{c_3+b_3}{114240}, 1 \right\} \right\}$
\\ \\[-3mm]
\midrule \\[-3mm]

$\xi_4$ & $\left\{
\begin{array}{l}
93 u_1 - (\boldsymbol{\mu} -9) u_5 < 0 \\[2mm]
4335 u_1 - 11 (\boldsymbol{\mu} -9) u_5 > 0
\end{array}\right.$  &    
$\left(
\begin{array}{cc}
 \frac{531}{8 (\boldsymbol{\mu} -9)} & \frac{13}{40} \\[2mm]
 \frac{7803}{32 (\boldsymbol{\mu} -9)} & -\frac{99}{160} \\
\end{array}
\right)\,$
 &  $\left\{ \frac{3(a_4-\sqrt{3}b_4)}{320(\boldsymbol{\mu}-9)}, \frac{3(a_4+\sqrt{3}b_4)}{320(\boldsymbol{\mu}-9)} \right\}$
 &  $\left\{ \left\{ \frac{c_4-\sqrt{3}b_4}{26010}, 1 \right\} , \left\{ \frac{c_4+\sqrt{3}b_4}{26010}, 1 \right\} \right\}$
\\ \\[-3mm]
\midrule \\[-3mm]

$\xi_5$ & $93 u_1 - (\boldsymbol{\mu} -9) u_5 > 0$  &    
$\left(
\begin{array}{cc}
 \frac{294}{5 (\boldsymbol{\mu} -18)} & \frac{63 (\boldsymbol{\mu} -32)}{155 (\boldsymbol{\mu} -18)} \\[2mm]
 \frac{1122}{5 (\boldsymbol{\mu} -18)} & \frac{127 \boldsymbol{\mu} +4446}{5580-310 \boldsymbol{\mu} } \\
\end{array}
\right)\,$
 &  $\left\{ \frac{a_5-b_5}{620(\boldsymbol{\mu}-18)}, \frac{a_5+b_5}{620(\boldsymbol{\mu}-18)} \right\}$
 &  $\left\{ \left\{ \frac{c_5-b_5}{139128}, 1 \right\} , \left\{ \frac{c_5+b_5}{139128}, 1 \right\} \right\}$
\\ \\[-3mm]
\midrule \\[-3mm]

$\xi_6$ & $ 4335 u_1 - 11 (\boldsymbol{\mu} -9) u_5 < 0$  &
$\left(
\begin{array}{cc}
\frac{93 (167-\boldsymbol{\mu} )}{65 (\boldsymbol{\mu} -9)} & 0 \\[2mm]
\frac{64464}{715 (\boldsymbol{\mu} -9)} & 1 \\
\end{array}
\right)\,$
 &  $\left\{\frac{93 (167-\boldsymbol{\mu})}{65 (\boldsymbol{\mu} -9)},1\right\}$
 &  $\left\{ \left\{\frac{11 (102-\boldsymbol{\mu})}{408},1\right\} , \{0,1\} \right\}$ \\
\bottomrule
\end{tabular}

} 
} 

\vspace{.1cm}
\begin{align*}
a_2 & = 16491-275\boldsymbol{\mu}, &
b_2 & = \sqrt{75625\boldsymbol{\mu}^2-101760050\boldsymbol{\mu}+15751183081}, &
c_2 & = 75359-275\boldsymbol{\mu}, \\
a_3 & = 15934-205\boldsymbol{\mu}, &
b_3 & = \sqrt{42025\boldsymbol{\mu}^2-37032780\boldsymbol{\mu}+5621864196}, &
c_3 & = 49086-205\boldsymbol{\mu}, \\
a_4 & = 3837-33\boldsymbol{\mu}, &
b_4 & = \sqrt{363\boldsymbol{\mu}^2+371906\boldsymbol{\mu}+800643}, &
c_4 & = 3243+33\boldsymbol{\mu}, \\
a_5 & = 13782-127\boldsymbol{\mu}, &
b_5 & = \sqrt{16129\boldsymbol{\mu}^2+40819452\boldsymbol{\mu}-607817916}, &
c_5 & = 22674-127\boldsymbol{\mu}.
\end{align*}

\vspace{.1cm}
\caption{\footnotesize{Branches of $\skPoin{\chi}{\mathcal{S}}$: defining equations of $\Pi_{\xi}$, the matrix of $\pi_\xi$, and their eigenvalues and eigenvectors, for each $\xi\in\{ \xi_1, \xi_2, \dots, \xi_6\}$.}}
\label{branch:table_paths}
\end{table}

\newpage
\begin{landscape}
\begin{table}[h]

\makebox[1 \textwidth][c]{       
\resizebox{1.8 \textwidth}{!}{

\begin{tabular}{|c|c|c|c|c|}

\toprule
$\quad\mathcal{H}\quad$ & Defining equations of $\Pi_{\mathcal{H}}\subset\Pi_{\gamma_5}$  &  $M_{\mathcal{H}}$, the matrix of $\pi_\mathcal{H}$ &  Eigenvalues of $M_{\mathcal{H}}$ &  Eigenvectors of $M_{\mathcal{H}}$ \vspace{.1cm} \\
\toprule \\[-3mm]

$\mathcal{H}_2$ & $\left\{
\begin{array}{l}
75 u_2-74 u_6 > 0 \\[2mm]
1107975 (\boldsymbol{\mu}-94) u_2+2 (94624\boldsymbol{\mu}-28591281) u_6 < 0
\end{array}\right.$  &    
$\left(
\begin{array}{cc}
 \frac{3 (3199 \boldsymbol{\mu} +740274)}{29900 (\boldsymbol{\mu} -9)} & -\frac{79 (2572 \boldsymbol{\mu} -238203)}{373750 (\boldsymbol{\mu} -9)} \\[2mm]
 -\frac{398871 (\boldsymbol{\mu} -94)}{119600 (\boldsymbol{\mu} -9)} & -\frac{9 (94624 \boldsymbol{\mu} -28591281)}{1495000 (\boldsymbol{\mu} -9)} \\
\end{array}
\right)$
 & $\left\{ \frac{3(a_2-79 b_2)}{2990000(\boldsymbol{\mu}-9)}, \frac{3(a_2+79 b_2)}{2990000(\boldsymbol{\mu}-9)} \right\}$
 & $\left\{ \left\{ \frac{c_2-b_2}{42075(\boldsymbol{\mu}-94)}, 1 \right\} , \left\{ \frac{c_2+b_2}{42075(\boldsymbol{\mu}-94)}, 1 \right\} \right\}$
\\ \\[-3mm]
\midrule \\[-3mm]

$\mathcal{H}_4$ & $\left\{
\begin{array}{l}
75 u_2-74 u_6 < 0 \\[2mm]
15 (235834-5079\boldsymbol{\mu}) u_2+74 (15654+131\boldsymbol{\mu}) u_6 < 0
\end{array}\right.$  &    
$\left(
\begin{array}{cc}
 \frac{3 (13219 \boldsymbol{\mu} +5308734)}{190624 (\boldsymbol{\mu} -9)} & \frac{530658-5567 \boldsymbol{\mu} }{12880 (\boldsymbol{\mu} -9)} \\[2mm]
 -\frac{459 (5949 \boldsymbol{\mu} -574846)}{762496 (\boldsymbol{\mu} -9)} & -\frac{9 (1871 \boldsymbol{\mu} -801474)}{51520 (\boldsymbol{\mu} -9)} \\
\end{array}
\right)$
 & $\left\{ \frac{3(a_4-b_4)}{3812480(\boldsymbol{\mu}-9)}, \frac{3(a_4+b_4)}{3812480(\boldsymbol{\mu}-9)} \right\}$
 & $\left\{ \left\{ \frac{c_4-b_4}{765( 5949 \boldsymbol{\mu}-574846 )}, 1 \right\} , \left\{ \frac{c_4+b_4}{765( 5949 \boldsymbol{\mu}-574846 )}, 1 \right\} \right\}$
\\ \\[-3mm]
\midrule \\[-3mm]

$\mathcal{H}_5$ & $\left\{
\begin{array}{l}
255 u_2-74 u_6 > 0 \\[2mm]
15 (235834-5079\boldsymbol{\mu}) u_2+74 (15654+131\boldsymbol{\mu}) u_6 > 0
\end{array}\right.$  &    
$\left(
\begin{array}{cc}
 \frac{21 (1177 \boldsymbol{\mu} +123226)}{52762 (\boldsymbol{\mu} -18)} & -\frac{1659 (\boldsymbol{\mu} -94)}{3565 (\boldsymbol{\mu} -18)} \\[2mm]
 -\frac{4029 (267 \boldsymbol{\mu} -24914)}{369334 (\boldsymbol{\mu} -18)} & -\frac{19 (541 \boldsymbol{\mu} -187014)}{24955 (\boldsymbol{\mu} -18)} \\
\end{array}
\right)$
 & $\left\{ \frac{a_5-79b_5}{3693340(\boldsymbol{\mu}-18)}, \frac{a_5+79b_5}{3693340(\boldsymbol{\mu}-18)} \right\}$
 & $\left\{ \left\{ \frac{c_5+b_5}{510(267\boldsymbol{\mu}-24914)}, 1 \right\}, \left\{ \frac{c_5-b_5}{510(267\boldsymbol{\mu}-24914)}, 1 \right\} \right\}$
\\ \\[-3mm]
\midrule \\[-3mm]

$\mathcal{H}_6$ & $255 u_2-74 u_6 < 0$  &
$\left(
\begin{array}{cc}
 -\frac{42 (\boldsymbol{\mu} -176)}{37 (\boldsymbol{\mu} -18)} & 0 \\[2mm]
 -\frac{4029 (\boldsymbol{\mu} -106)}{518 (\boldsymbol{\mu} -18)} & 1 \\
\end{array}
\right)$
 &  $\left\{\frac{42 (176-\boldsymbol{\mu})}{37 (\boldsymbol{\mu} -18)},1\right\}$
 &  $\left\{ \left\{\frac{14 (\boldsymbol{\mu} -102)}{51 (\boldsymbol{\mu} -106)},1\right\} , \{0,1\} \right\}$\\
\bottomrule
\end{tabular}

} 
} 

\vspace{.1cm}
\begin{align*}
a_2 & = 122787543-123 922 \boldsymbol{\mu}, &
b_2 & = \sqrt{320 140 324 \boldsymbol{\mu}^2-60 787 796 412 \boldsymbol{\mu}+2 893 236 225 489}, &
c_2 & = 617 217-5618\boldsymbol{\mu}, \\
a_4 & = 142 050 954-75 491 \boldsymbol{\mu}, &
b_4 & = \sqrt{2 615 265 201 481 \boldsymbol{\mu}^2-504 216 904 560 828 \boldsymbol{\mu}+24 312 026 567 983 716}, &
c_4 & = 35 876 274-339 871\boldsymbol{\mu}, \\
a_5 & = 353 512 794 + 104 449 \boldsymbol{\mu}, &
b_5 & = \sqrt{3 386 009 761 \boldsymbol{\mu}^2-644 714 033 868 \boldsymbol{\mu}+30 745 583 285 796}, &
c_5 & = 2 181 906-20 579 \boldsymbol{\mu}.
\end{align*}

\vspace{.1cm}
\caption{\footnotesize{Cycles in $\Pi_{\gamma_5}$: defining equations of $\Pi_{\mathcal{H}}\subset\Pi_{\gamma_5}$, the matrix of $\pi_\mathcal{H}$, and their eigenvalues and eigenvectors, for each $\mathcal{H}\in\{\mathcal{H}_2, \mathcal{H}_4, \mathcal{H}_5, \mathcal{H}_6\}$ in $\Pi_{\gamma_5}$.}}
\label{branch:table_cycles_gamma_5}
\end{table}

\begin{table}[h]

\makebox[1 \textwidth][c]{       
\resizebox{1.8 \textwidth}{!}{

\begin{tabular}{|c|c|c|c|c|}

\toprule
$\quad\mathcal{H}\quad$ & Defining equations of $\Pi_{\mathcal{H}}\subset\Pi_{\gamma_8}$  &   $M_{\mathcal{H}}$, the matrix of $\pi_\mathcal{H}$ &  Eigenvalues of $M_{\mathcal{H}}$ &  Eigenvectors of $M_{\mathcal{H}}$ \vspace{.1cm} \\
\toprule \\[-3mm]

$\mathcal{H}_1$ & $4335 u_1-11 (\boldsymbol{\mu}-9) u_5 < 0$ &    
$\left(
\begin{array}{cc}
 -\frac{93 (\boldsymbol{\mu} -167)}{65 (\boldsymbol{\mu} -9)} & 0 \\[2mm]
 \frac{64464}{715 (\boldsymbol{\mu} -9)} & 1 \\
\end{array}
\right)$
 &  $\left\{\frac{93 (167-\boldsymbol{\mu})}{65 (\boldsymbol{\mu} -9)},1\right\}$
 &  $\left\{ \left\{\frac{11 (102-\boldsymbol{\mu})}{408},1\right\} , \{0,1\} \right\}$
\\ \\[-3mm]
\midrule \\[-3mm]

$\mathcal{H}_2$ & $\left\{
\begin{array}{l}
4335 u_1-11 (\boldsymbol{\mu}-9) u_5 > 0 \\[2mm]
348435 u_1-1871 (\boldsymbol{\mu}-9) u_5 < 0
\end{array}\right.$  &    
$\left(
\begin{array}{cc}
 -\frac{924093 (\boldsymbol{\mu} -167)}{598000 (\boldsymbol{\mu} -9)} & \frac{869 (\boldsymbol{\mu} -167)}{2990000} \\[2mm]
 -\frac{15991101}{598000 (\boldsymbol{\mu} -9)} & \frac{3876933}{2990000} \\
\end{array}
\right)$
 & $\left\{ \frac{3(a_2-79 b_2)}{2990000(\boldsymbol{\mu}-9)}, \frac{3(a_2+79 b_2)}{2990000(\boldsymbol{\mu}-9)} \right\}$
 & $\left\{ \left\{ \frac{c_2-b_2}{337 365}, 1 \right\} , \left\{ \frac{c_2+b_2}{337 365}, 1 \right\} \right\}$
\\ \\[-3mm]
\midrule \\[-3mm]

$\mathcal{H}_4$ & $\left\{
\begin{array}{l}
93 u_1- (\boldsymbol{\mu}-9) u_5 < 0 \\[2mm]
348435 u_1-1871 (\boldsymbol{\mu}-9) u_5 > 0
\end{array}\right.$  &    
$\left(
\begin{array}{cc}
 -\frac{9 (56269 \boldsymbol{\mu} -9931190)}{381248 (\boldsymbol{\mu} -9)} & \frac{149290-1667 \boldsymbol{\mu} }{1906240} \\[2mm]
 -\frac{17901}{1702 (\boldsymbol{\mu} -9)} & \frac{10293}{8510} \\
\end{array}
\right)$
 & $\left\{ \frac{3(a_4-b_4)}{3812480(\boldsymbol{\mu}-9)}, \frac{3(a_4+b_4)}{3812480(\boldsymbol{\mu}-9)} \right\}$
 & $\left\{ \left\{ \frac{c_4-b_4}{13 366 080}, 1 \right\} , \left\{ \frac{c_4+b_4}{13 366 080}, 1 \right\} \right\}$
\\ \\[-3mm]
\midrule \\[-3mm]

$\mathcal{H}_5$ & $\left\{
\begin{array}{l}
93 u_1- (\boldsymbol{\mu}-9) u_5 > 0 \\[2mm]
124899 u_1+(174789-10382\boldsymbol{\mu}) u_5 < 0
\end{array}\right.$  &
$\left(
\begin{array}{cc}
 -\frac{3 (23883 \boldsymbol{\mu} -4222210)}{59570 (\boldsymbol{\mu} -18)} & -\frac{79 \left(52 \boldsymbol{\mu} ^2-10969 \boldsymbol{\mu} +612630\right)}{1846670 (\boldsymbol{\mu} -18)} \\[2mm]
 -\frac{64464}{4255 (\boldsymbol{\mu} -18)} & \frac{16 (10382 \boldsymbol{\mu} -174789)}{131905 (\boldsymbol{\mu} -18)} \\
\end{array}
\right)$
 & $\left\{ \frac{a_5-79b_5}{3 693 340(\boldsymbol{\mu}-18)}, \frac{a_5+79b_5}{3693340(\boldsymbol{\mu}-18)} \right\}$
 & $\left\{ \left\{ \frac{c_5+b_5}{708 288}, 1 \right\}, \left\{ \frac{c_5-b_5}{708 288}, 1 \right\} \right\}$\\
\bottomrule
\end{tabular}

} 
} 

\vspace{.1cm}
\begin{align*}
a_2 & = 122787543-123 922 \boldsymbol{\mu}, &
b_2 & = \sqrt{320 140 324 \boldsymbol{\mu}^2-60 787 796 412 \boldsymbol{\mu}+2 893 236 225 489}, &
c_2 & = 17 927\boldsymbol{\mu}-1 701 498, \\
a_4 & = 142 050 954-75 491 \boldsymbol{\mu}, &
b_4 & = \sqrt{2 615 265 201 481 \boldsymbol{\mu}^2-504 216 904 560 828 \boldsymbol{\mu}+24 312 026 567 983 716}, &
c_4 & = 1 612 579 \boldsymbol{\mu} - 155 884 746, \\
a_5 & = 353 512 794 + 104 449 \boldsymbol{\mu}, &
b_5 & = \sqrt{3 386 009 761 \boldsymbol{\mu}^2-644 714 033 868 \boldsymbol{\mu}+30 745 583 285 796}, &
c_5 & = 57 553 \boldsymbol{\mu}-5 466 054.
\end{align*}

\vspace{.1cm}
\caption{\footnotesize{Cycles in $\Pi_{\gamma_8}$: defining equations of $\Pi_{\mathcal{H}}\subset\Pi_{\gamma_8}$, the matrix of $\pi_\mathcal{H}$, and their eigenvalues and eigenvectors, for each $\mathcal{H}\in\{\mathcal{H}_1, \mathcal{H}_2, \mathcal{H}_4, \mathcal{H}_5\}$ in $\Pi_{\gamma_8}$.}}
\label{branch:table_cycles_gamma_8}
\end{table}
\end{landscape}

\newpage
\section{Notation}
\label{AppendixB}
We list the main notation for constants and auxiliary functions used in this paper in order of appearance with the reference of the section containing a definition.

\begin{tiny}
\begin{table}[h]
\begin{center}
\begin{tabular}{|c|l|c|}
\toprule
{Notation}  & \qquad  Definition/meaning \qquad  \qquad &\quad  Section \qquad \\[1mm]
\toprule \\[-3mm]

\vspace{-1mm}
&&\\
$\mathcal{V}, \mathcal{E}$ & Set of equilibria (vertices), set of edges of the cube    & \S \ref{sec:model} \\  &&\\[-1mm]
\midrule

\vspace{-1mm}
&&\\
$\mathcal{F}_v$ &  $\text{Set of three faces $\sigma_j$ defined by  the component $x_j=0$ at $v\in \mathcal{V}$}$    & \S \ref{sec:model} \\  &&\\[-1mm]
\midrule

\vspace{-1mm}
&&\\
$\mathcal{F}$ &  Set of all faces of the cube  & \S \ref{sec:model}  \\  &&\\[-1mm]
\midrule

\vspace{-1mm}
&&\\
$\mathcal{I}, \mathcal{I}_1, \mathcal{I}_2, \mathcal{I}_3 $  &$\mathcal{I}_1=\left[ \frac{850}{11}, \boldsymbol{\mu}_1 \right[,
\mathcal{I}_2=\left] \boldsymbol{\mu}_1, \boldsymbol{\mu}_2 \right[,
\mathcal{I}_3=\left] \boldsymbol{\mu}_2, \frac{544}{5} \right],
$
$\boldsymbol{\mu}_1 = 102$ and $\boldsymbol{\mu}_2\approx 105.04$.  & \S \ref{ss:interior}  \\  &&\\[-1mm]
\midrule

\vspace{-1mm}
&&\\
$\mathcal{H}, \mathcal{H}_i$  &  Heteroclinic network, heteroclinic cycle  & \S \ref{ss:heteroclinic}  \\  &&\\[-1mm]
\midrule

\vspace{-1mm}
&&\\
$N_v$ & Cubic neighbourhood of $v\in \mathcal{V}$ where the flow may be $C^1$--linearized   & \S \ref{ss: global map} \\  &&\\[-1mm]
\midrule

\vspace{-1mm}
&&\\
$N_\gamma$ & Tubular neighbourhood of $\gamma\in \mathcal{E}$    & \S   \ref{ss: global map}  \\  &&\\[-1mm]
\midrule

\vspace{-1mm}
&&\\
$\Psi_\varepsilon$ & Quasi-change of coordinates ($\varepsilon$: blow-up parameter)    & \S \ref{subsec:dual_cone} \\  &&\\[-1mm]
\midrule

\vspace{-1mm}
&&\\
$\Pi_v, \Pi_\gamma$ &  $\Psi_\varepsilon (N_v), \quad \Psi_\varepsilon (N_v)=\Pi_{v^*}\cap \Pi_{v}$   &  \S \ref{subsec:dual_cone} \\  &&\\[-1mm]
\midrule

\vspace{-1mm}
&&\\
$\chi^v, \chi^v_j$ & Character vector field at $v$; $j$-component $\sigma$ of $\chi^v$ where $j\in \{1, ..., 6\}$  & \S  \ref{ss:skeleton}  \\  &&\\[-1mm]
\midrule

\vspace{-1mm}
&&\\
$D_{\gamma, \gamma'},D_{\gamma, \gamma'}^*$ & Set of points in $\Sigma_\gamma$ that follows the admissible path $\{\gamma, \gamma'\}$\, , \, $\Psi_\varepsilon(D_{\gamma, \gamma'})$   & \S \ref{ss:5.5}  \\  &&\\[-1mm]
\midrule

\vspace{-1mm}
&&\\
$P_{\gamma, \gamma'}$ & Map carrying points from $D_{\gamma, \gamma'}$ to $Out(v)\cap \gamma'$  & \S  \ref{ss:5.5} \  \\  &&\\[-1mm]
\midrule

\vspace{-1mm}
&&\\
$F_{\gamma, \gamma'}$ & $\Psi_\varepsilon \circ P_{\gamma, \gamma'} \circ (\Psi_\varepsilon )^{-1}|_{D^*_{\gamma, \gamma'}}$  & \S  \ref{ss:5.5}    \\  &&\\[-1mm]
\midrule

\vspace{-1mm}
&&\\
$\Pi_{\gamma, \gamma'}$ & $\left\{  y \in \inter(\Pi_\gamma): y_\sigma >\frac{\chi_\sigma^v}{\chi_{\sigma^*}^{v}}y_{\sigma_*}, \quad \forall \sigma\in \mathcal{F}_\sigma, \quad \sigma\neq \sigma_*  \right\}$   & \S  \ref{ss:5.5}    \\  &&\\[-1mm]
\midrule

\vspace{-1mm}
&&\\
$L_{\gamma, \gamma'}$ & Induced linear map from $\Pi_{\gamma, \gamma'}$ to $\Pi_{\gamma'}$  & \S  \ref{ss:5.5}    \\  &&\\[-1mm]
\midrule

\vspace{-1mm}
&&\\
$\Pi_\xi$  &  $\inter(\Pi_{\gamma_0}) \cap \bigcap_{j=1}^m \left( L_{\gamma_{j-1}, \gamma_j } \circ ...\circ L_{\gamma_0, \gamma_1} \right)^{-1} \left( \inter \left( \Pi_{\gamma_j}\right) \right) \subset (\RR_0^+)^6$  & \S \ref{ss:5.6}  \\  &&\\[-1mm]
\midrule

\vspace{-1mm}
&&\\
$\pi_\xi$  &  Linear map $L_{\gamma_{m-1}, \gamma_m}\circ ... \circ L_{\gamma_{0}, \gamma_1}$  & \S \ref{ss:5.6}    \\  &&\\[-1mm]
\midrule

\vspace{-1mm}
&&\\
$\pi_\mathcal{S}$  &  Skeleton map defined in any sector of $\Pi_\mathcal{S}$ ($\mathcal{S}$: structural set) & \S \ref{ss:5.6}    \\  &&\\[-1mm]
\midrule

\vspace{-1mm}
&&\\
$\Delta_\xi$  &   $\{\, u=(u_1, ..., u_6)\in {\rm int}(\skDomPoin{\chi}{\xi})\,\colon \, \overline{u} = 1 \,\}$ where $ \overline{u}=\sum_{i=1}^6 u_j$& \S \ref{ss: paths}   \\  &&\\[-1mm]
\midrule

\vspace{-1mm}
&&\\
$\hat{\pi}_\xi$  &   Projective map along the $S$-branch $\xi$ & \S \ref{sec:projective_map}    \\  &&\\[-1mm]
\bottomrule

\end{tabular}
\end{center}
\label{notationB}
\bigskip
\caption{\small Notation.}
\end{table}
\end{tiny}

\end{document}